\documentclass[11 pt, oneside, a4paper]{amsart}
\usepackage{hyperref}
\usepackage[utf8]{inputenc}
\usepackage{enumerate}
\usepackage[english]{babel}
\usepackage[T1]{fontenc}
\usepackage{ tipa }
\usepackage[utf8]{inputenc}
\usepackage{amsmath, amssymb,amsthm}
\usepackage{array}
\usepackage{graphicx}
\usepackage{float}
\usepackage[normalem]{ulem} 
\usepackage{caption,subcaption}
\usepackage{tikz}
\usepackage{tikz-cd}
\usepackage{ upgreek }
\usetikzlibrary{matrix,arrows,decorations.pathmorphing}
\usepackage{xfrac} 
\usepackage{mathrsfs}
\usepackage{physics}
\usepackage{faktor}

\makeatletter
\newcommand*{\rom}[1]{\expandafter\@slowromancap\romannumeral #1@}
\makeatother

\usepackage{graphicx}
\newtheorem{thm}{Theorem}[section]
\newtheorem{lemma}[thm]{Lemma}

\newtheorem{rmk}[thm]{Remark}
\newtheorem{defi}[thm]{Definition}
\newtheorem{cor}[thm]{Corollary}
\newtheorem{prop}[thm]{Proposition}

\def\A{\mathbb{A}}
\def\B{\mathbb{B
}}
\def\C{\mathbb{C}}
\def\D{\mathbb{D}}

\def\H{\mathbb{H}}

\def\N{\mathbb{N}}

\def\P{\mathbb{P}}
\def\Q{\mathbb{Q}}
\def\R{\mathbb{R}}

\def\P{\mathbb{P}}

\def\cC{\mathcal{C}}
\def\cD{\mathcal{D}}
\def\cE{\mathcal{E}}
\def\cF{\mathcal{F}}
\def\cG{\mathcal{G}}
\def\cH{\mathcal{H}}
\def\cI{\mathcal{I}}
\def\cJ{\mathcal{J}}
\def\cK{\mathcal{K}}

\def\cM{\mathcal{M}}

\def\cO{\mathcal{O}}

\def\sH{\mathscr{H}}

\def\s{\sigma}
\def\mf{\mathfrak{f}}

\def\f{\varphi}
\def\x{\times}

\def\mucan{\mu_{\mathrm{can}}}
\def\htop{h_{\mathrm{top}}}
\def\na{\mathrm{na}}

\DeclareMathOperator{\SH}{SH}
\DeclareMathOperator{\QSH}{QSH}

\DeclareMathOperator{\Rat}{Rat}
\DeclareMathOperator{\rat}{rat}

\DeclareMathOperator{\lV}{\lVert}
\DeclareMathOperator{\rV}{\rVert}
\DeclareMathOperator{\res}{\mathrm{res}}
\DeclareMathOperator{\rluo}{\mathrm{r}_L}
\DeclareMathOperator{\rluof}{\mathfrak{r}_L}
\DeclareMathOperator{\spec}{\mathrm{Spec}}
\DeclareMathOperator{\PGL}{\mathrm{PGL}}

\DeclareMathOperator{\SL}{\mathrm{SL}}
\DeclareMathOperator{\cl}{\mathrm{cl}}

\usepackage{mathtools}

\title{non-Archimedean techniques and dynamical degenerations}

\author{Charles Favre}
\address{CNRS - Centre de Math\'ematiques Laurent Schwartz, 
	\'Ecole Polytechnique, 
	91128 Palaiseau Cedex, France}
\email{\href{charles.favre@polytechnique.edu}{charles.favre@polytechnique.edu}}

\author{Chen Gong}
\address{Centre de Math\'ematiques Laurent Schwartz, 
	\'Ecole Polytechnique, 
	91128 Palaiseau Cedex, France}
\email{\href{chen.gong@polytechnique.edu}{chen.gong@polytechnique.edu}}

\date{\today}
\thanks{C.G. is  supported by a CSC-202108070159 grant  from the chinese government.}

\begin{document}

\begin{abstract}
We develop non-Archimedean techniques to analyze the degeneration of a sequence of rational maps of the complex projective line. We provide 
an alternative to Luo's method which was based on ultra-limits of the hyperbolic $3$-space.
We build hybrid spaces using Berkovich theory which enable us to prove the convergence of equilibrium measures, and to determine
the asymptotics of Lyapunov exponents.
\end{abstract} 

\maketitle

\setcounter{tocdepth}{2}
\setcounter{secnumdepth}{4}

\tableofcontents

\section*{Introduction}
The main goal of this paper is to develop non-Archimedean techniques to analyze the degeneration of a sequence of rational maps of the complex projective line.

\medskip

\noindent \emph{The moduli space of rational maps.}

A complex rational map $f\colon \P^1_\C \to \P^1_\C$
of degree $d\ge 2$ is determined in homogeneous coordinates by 
two homogeneous polynomials $P,Q$ of degree $d$ without common factors
so that $f[z_0\colon z_1]=[P(z_0,z_1) \colon Q(z_0,z_1)]$.
The condition on $P$ and $Q$ to have no non-trivial common zeroes can be expressed as the non-vanishing of their homogeneous resultant $\Res(P,Q)$.  It follows that the space $\Rat_{d}$ of complex rational maps of degree $d\ge2$ 
carries a natural structure of algebraic variety, and is a Zariski open subset of $\P^{2d+1}_\C$.  

The automorphism group of the projective line $\PGL_2$ acts by conjugacy on $\Rat_d$, and Silverman~\cite{SJ98} has proved the existence of an affine algebraic variety $\rat_d$ whose complex points
$\rat_{d}(\C)$ coincide with the set of conjugacy classes $\Rat_d(\C)/\PGL_2(\C)$. The  canonical class map $\cl\colon \Rat_d\to \rat_d$
 is open. It is customary to write $[f]=\cl(f)$ for a rational map $f\in\Rat_d(\C)$.  

The moduli space $\rat_d$ has dimension $2d-2$ and is a rational variety by a theorem of Levy~\cite{AL11}. Milnor has shown that $\rat_2$ is isomorphic to the affine space $\A^2$. When $d\ge 3$, $\rat_d$ admits orbifold singularities by Miasnikov-Stout-Williams~\cite{zbMATH06789350modulisingular}. 

\medskip

\noindent \emph{Characteristics of holomorphic dynamical systems.}

Any rational map $f\in \Rat_d(\C)$ defines a holomorphic map on the Riemann sphere $f\colon \hat{\C}\to\hat{\C}$
whose dynamics is described by the Fatou-Julia theory~\cite{Mi06}. The sphere can be decomposed as the disjoint union of the Fatou set, a totally invariant open set on which the dynamics of $f$ is regular; and the Julia set $J_f$, a compact totally invariant set on which the dynamics is chaotic. By a theorem of Lyubich and Freire-Mané-Lopès, $f$ admits a unique probability measure of maximal entropy $\mu_f$ called the equilibrium measure, whose support is the Julia set, see, e.g.,~\cite{Ber13}. The Lyapunov exponent of $f$ with respect to $\mu_f$ is defined by the integral
\[\chi_{f}=\int \log|df|d\mu_{f},\]
where $|df|$ is the norm of the derivative of $f$ computed in terms of the spherical metric on $\hat{\C}.$ This integral converges, and we have  $\infty>\chi_f\ge \frac{1}{2}\log d>0$ (see op.cit.) so that $f$ is weakly expanding on $J_f$.

\medskip

\noindent \emph{Presentation of the main problem.}

The complex affine varieties $\Rat_d$ and $\rat_d$ defined above can be analytified and give rise to complex analytic (Stein) varieties $\Rat_{d}^{an}(\C)$ and $\rat_{d}^{an}(\C)$. Note that as sets 
$\Rat_{d}^{an}(\C)$ and $\rat_{d}^{an}(\C)$ are canonically in bijection with $\Rat_d(\C)$ and $\rat_d(\C)$ respectively. In the sequel,  we shall drop
the exponent $an$ to keep the notations lighter. 

We say that a sequence $\mf_n\in\rat_d(\C)$ is \emph{degenerating}
if $\mf_n$ leaves any fixed compact subset of $\rat_{d}(\C)$ for $n$ large enough. 
Our main goal is to construct limits of $\mf_n$ which are dynamically meaningful in order to analyze
how the characteristics of the dynamics of $\mf_n$ evolve as $n\to\infty$. 
 More precisely, we shall exhibit a sequence $f_n\in\Rat_d(\C)$ such that $[f_n]=\mf_n$ and
suitable subsequences of $f_n$ converge in a natural way to a rational map defined 
over a non-Archimedean metrized field. 
In the applications, it is essential to  transfer dynamical informations between the complex and the non-Archimedean rational maps.
This is in general a very delicate issue. We shall explain how the complex equilibrium measures and the complex Lyapunov exponents converge
in a natural way to their non-Archimedean counterparts.

\medskip

\noindent \emph{Algebraic methods to understand degenerations.}

Early tentatives to understand degenerations of complex dynamical systems relied on purely complex theoretic or algebraic  techniques.
Silverman~\cite{SJ98} constructed a GIT (projective) compactification $\overline{\rat}_d(\C)$ of $\rat_d(\C)$ extending former works by Milnor~\cite{JM93} and Epstein~\cite{Ep00}
in the case $d=2$. Demarco~\cite{LD07} observed that the iteration maps
$I_l\colon \rat_d(\C) \to \rat_{d^l}(\C)$ do not extend as regular maps to this compactification, and analyzed
the indeterminacy locus of these maps. Her work was completed by Kiwi and Nie~\cite{KN23}.  
Demarco~\cite{zbMATH05004325demarcoiteration} also proposed a way to compactify $\rat_d(\C)$ using limits of their measures of maximal entropy. 
Among other things this technique enabled her to prove that $I_n$ was proper. 

\medskip

\noindent \emph{Meromorphic families define dynamics over the field of Laurent series.}

The next step was undertaken by Kiwi in a series of influential works. Pick any holomorphic family $(f_t)_{t\in \D^*}$
of rational maps parameterized by the complex unit disk $\D^*= \{0<|t|<1\}$, and suppose it extends meromorphically 
to the puncture. Such a family is determined by a pair of homogenous polynomials $P_t,Q_t$ having coefficients 
in the ring of holomorphic functions on $\D^*$ that are meromorphic at $0$.  Kiwi realized that one could interpret 
$P_t, Q_t$ as polynomials with coefficients in the non-Archimedean metrized field $\C((t))$ of formal Laurent series, so that 
$(f_t)$ naturally defines a rational map $f_{\na}\colon \P^1_{\C((t))} \to  \P^1_{\C((t))}$. 
This simple and  fruitful  idea was successfully applied in the cubic polynomial case~\cite{Ki06} and then in the quadratic rational case~\cite{Ki14}.
It then played a key role in the systematic study of rescaling limits~\cite{Ki15}, and
in the work of Demarco-Faber~\cite{DF14,DF16} in which the limit in $\hat{\C}$ of equilibirum measures of $(f_t)$ was described as $t\to0$. More recently Pilgrim and Nie~\cite{NP22} used this idea to prove the boundedness of 
bicritical hyperbolic components of disjoint type extending a theorem by Epstein~\cite{Ep00}.

\medskip

\noindent \emph{Hybrid spaces.}

Hybrid spaces are geometric objects mixing complex and non-Archimedean varieties.
The appearance of such spaces 
can be traced back to the construction of a compactification of the $\SL_2$-character variety by Morgan-Shalen~\cite{MS85}. Similar spaces appeared as compactifications of the space of polynomial maps by Demarco-McMullen~\cite{DMM08} using $\R$-trees. Hybrid spaces were formalized by Berkovich~\cite{VB09}.

To motivate their introduction, note that to fully exploit the connection between the complex and non-Archimedean worlds, it is important to understand 
the Fatou-Julia of a rational map $g$ defined on a non-Archimedean field $k$. This  theory has been extensively developed in the early 2000's 
by Rivera-Letelier and Benedetto (see~\cite{Ben19} for a detailed account on this theory). It turns out that the natural space
over which the Fatou-Julia theory  of $g$ can be carried over is the Berkovich analytification $\P^{1,\mathrm{an}}_{k}$
of the projective line. One can for instance construct an equilibrium measure $\mu_g$ having support $J_g$ as in the complex case, and define its Lyapunov exponent $\chi_g$ (see~\cite{Ok15}).  
We refer to \S\ref{sec:P1overring} for details of the construction of $\P^{1,\mathrm{an}}_{k}$. Suffice it to say that
building blocks of Berkovich spaces are compact sets $\cM(B)$ defined as the set of multiplicative semi-norms
on a suitable Banach ring $B$ (the space $\cM(B)$ is called the Berkovich spectrum of $B$). 

Varying the Banach ring allows for a lot of flexibility in the theory, and lies at the root of the idea of hybrid spaces 
which originates in the paper of Berkovich cited above. 
These spaces combine harmoniously complex varieties and non-Archimedean analytic varieties, and are natural objects over which one can make sense
and study degenerations of measures. In holomorphic dynamics, they were introduced by the first author building on a work by 
Boucksom-Jonsson~\cite{zbMATH06754324tropicalvolume}. Namely, he showed in~\cite{zbMATH07219254CFdegeneration} the convergence
of the equilibrium measures $\mu_{f_t}$ of a meromorphic family $(f_t)_{t\in\D^*}$ to $\mu_{f_{\na}}$ in a natural hybrid space, and proved that the asymptotics
of  the Lyapunov exponents $\chi_{f_t}$ was governed by $\chi_{f_{\na}}$ (an alternative proof of the latter fact was given by Okuyama-Gauthier-Vigny in~\cite{gauthier2020approximation}). 
This method have been applied since then to a variety of contexts including automorphisms of algebraic varieties~\cite{irokawa2023hybrid,irokawa2024hybrid}, or $\SL_2$-character varieties~\cite{dang2024variation,DuFa19}.
In a series of recent works, Poineau~\cite{poineau2024dynamique} have extended the continuity of the equilibrium measures to a quite large category of Berkovich spaces which encompasses
hybrid spaces. 

To summarize the discussion of the latter two paragraphs, non-Archime\-dean degeneration techniques are now well-understood in the case of meromorphic families $(f_t)_{t\in\D^*}$
and are systematically used in many works in arithmetic dynamics~\cite{DHY20,FG22,JX23,poineau2024dynamique2}.

\medskip

\noindent \emph{Degeneration of a sequence of rational maps and Luo's construction.}

Recently, Luo~\cite{luo2021trees,YLuo22} classified nested hyperbolic components and found a way to analyze the degeneration of a sequence of rational maps, even
when they do not arise from a meromorphic family.  
As our paper aims to propose an alternative method to his approach, let us spend some time
discussing his construction. 

Pick any $f\in \Rat_d(\C)$. The first step is to measure in a dynamical way the distance of $[f]$ to the boundary of $\rat_d(\C)$ in $\overline{\rat}_d(\C)$. 
Denote by  $(\H^3 ,d_\H)$ the hyperbolic $3$-space, and pick any base point $x_\star\in\H^3 $. Recall
that the visual boundary of $\H^3 $ is the Riemann sphere $\hat{\C}$. 
Luo considers a  natural extension $\cE(f)\colon \H^3  \to \H^3 $ of $f\colon \hat{\C}\to\hat{\C}$ 
based on the notion of conformal barycenter (see also~\cite{petersen2011conformally}), and introduces the quantity
\[
\rluo(f):= \sup_{\cE(f)(y)=x_\star} d_\H (y,x_\star)
\in \R^*_+\] 
He proves that
$\rluof(\mf):= \inf_{[f]=\mf}\rluo(f)$ is a proper function $\rat_d(\C)$,  so that the bigger $\rluof(\mf)$ is, the closer $\mf$ is to the boundary
$\overline{\rat}_d(\C)\setminus \rat_d(\C)$.

Now choose any degenerating sequence $\mf_n\in\rat_d(\C)$, and set $\epsilon_n := \rluo(\mf_n)^{-1}$ (so that $\epsilon_n\to0$). 
Choose lifts $f_n\in\Rat_d(\C)$ such that $[f_n]=\mf_n$ and $\rluo(f_n)= \rluof (\mf_n)$ for all $n$.
Then Luo considers the sequence of pointed metric spaces $(\H^3 , x_\star, \epsilon_n d_\H)$
and  its ultra-limit $\H_\omega$ in Gromov's sense~\cite[\S 9]{Ka09}. This construction depends on the choice of 
a (non-principal) ultra-filter $\omega$ which is an element of the power set of $\N$ (see \S\ref{sectionnonstandardanalysis} for details).
The defining property of $\omega$ ensures
the existence of a limit (along $\omega$) of sequence of many classes of objects (like bounded sequences in a 
locally compact metric space, or pointed metric spaces as above).
Using this tool, Luo proves that the sequence $\cE(f_n)$ defines a self-map $f_\omega$ on $\H_\omega$, and the crucial fact 
that this map is $d$-to-one,~\cite[Theorem 1.2]{luo2021trees}.
He goes further by proving that $\H_\omega$ naturally embeds in the Berkovich analytification of 
$\P^{1,\mathrm{an}}_{\sH(\omega)}$ for some complete non-Archimedean metrized field $\sH(\omega)$
and building on the fact that $f_\omega$ is finite-to-one, he concludes that $f_\omega$ extends as
a rational map in $\Rat_d(\sH(\omega))$. 

\medskip

\noindent \emph{Degeneration of a sequence of rational maps using Berkovich spaces.}

Let us now explain  our strategy. 
For any $f\in \Rat_d(\C)$ defined by homogenous polynomials $P,Q$ of degree $d$ and normalized so that the maximum
of their coefficients equal $1$,  we set $|\Res(f)|:= |\Res(P,Q)|$. The function
$|\Res| \colon \Rat_d(\C) \to \R^*_+$ is
continuous, bounded from above, and proper. 
We prove (Proposition~\ref{propminres}) that
$|\res (\mf)|:= \sup_{[f]=\mf} |\Res(f)|$ defines a function $\rat_d(\C)\to\R^*_+$ which is also  proper, continuous and bounded from above.
 Set $C_d= e\, \sup_{\rat_d(\C)} |\res|$.

Choose any degenerating sequence $\mf_n\in\rat_d(\C)$, and set this time
$\epsilon_n:= (-\log|\res(\mf_n)|/C_d)^{-1}$ (so that $\epsilon_n\to 0$). 
Pick any lift $f_n\in \Rat_d(\C)$ such that $[f_n]=\mf_n$ and $|\Res(f_n)|= |\res(\mf_n)|$ for all $n$.
We now consider the Banach ring 
\begin{equation}\label{eq:def-aep}
A^\epsilon
= \{ z\in \C^\N \mid \sup_\N |z_n|^{\epsilon_n} <\infty\}~.
\end{equation}
Our main observation is that the sequence $f_n$
naturally defines  a rational map of degree $d$ over the ring $A^\epsilon$. 
Considering the Berkovich analytification $\P^{1,\mathrm{an}}_{A^\epsilon}$ 
of the projective line over $A^\epsilon$, we get a hybrid space 
over which we can prove the continuity of the equilibrium measures, 
of the Lyapunov exponents, and recover Luo's construction.   
 The next theorem summarizes our main results. 
 
\begin{thm}
\label{thm: sequential hyridation construction}
Let $\mf_n$ be a degenerating sequence in $\rat_d(\C)$. Set 
\[\epsilon_n:= (-\log(|\res(\mf_n)|/C_d))^{-1}~,\] 
and choose  $f_n\in \Rat_d(\C)$ such that $[f_n]=\mf_n$ and $|\Res(f_n)|= |\res(\mf_n)|$ for all $n$. 
Let $A^\epsilon$ be the Banach ring defined by~\eqref{eq:def-aep}.
Then the  following statements hold true. 
\begin{enumerate}
 \item
 The Berkovich analytification $\P^{1,\mathrm{an}}_{A^\epsilon}$ is a compact space, coming with a canonical continuous
 map $\pi\colon \P^{1,\mathrm{an}}_{A^\epsilon}\to\beta\N$ where $\beta\N$ is the Stone-Čech compactification of the integers. 
 For any integer $n\in \N$, the fiber is homeomorphic to the Riemann sphere $\pi^{-1}(n)\simeq \hat{\C}$, and $\pi^{-1}(\N)$
 is dense in $\P^{1,\mathrm{an}}_{A^\epsilon}$. 
 For any $\omega\in\beta\N\setminus \N$, 
the fiber $\pi^{-1}(\omega)$ is isomorphic to the Berkovich analytification of the projective line over a non-Archimedean
 field $\sH(\omega)$. 
  \item 
 There exists a rational map $f\in\Rat_d(A^\epsilon)$ which 
induces a continuous self-map $f\colon\P^{1,\mathrm{an}}_{A^\epsilon}\to\P^{1,\mathrm{an}}_{A^\epsilon}$
such that $\pi\circ f= \pi$, and for any $n\in \N$, we have $f|_{\pi^{-1}(n)} =f_n$
under the above identification $\pi^{-1}(n)\simeq \hat{\C}$. 
When  $\omega\in\beta\N\setminus \N$, the map $f_\omega =f|_{\pi^{-1}(\omega)}$ 
is a rational map of degree $d$ over $\sH(\omega)$ which does not have potential good reduction.
    \item 
The family of equilibrium measures $\mu_{f_\omega}$ supported on $\pi^{-1}(\omega)$ 
where $\omega$ ranges over $\beta \N$ forms a continuous family of probability measures on $\P^{1,\mathrm{an}}_{A^\epsilon}$.
Furthermore, the Lyapunov exponent 
\begin{equation}\label{eq:conti-lyap}
\chi_\omega
:=
\int_{\pi^{-1}(\omega)} \log|df|_\omega d\mu_{f_\omega},
\end{equation}
is continuous on $\beta\N$.
\end{enumerate}
\end{thm}
As an application of our techniques, we provide new proofs of a result by Kiwi on rescaling limits (see Corollary~\ref{cor:kiwi}), and of the properness of the iteration map, a result originally due to DeMarco, see~\cite[Corollary~0.3]{zbMATH05004325demarcoiteration}. 
\begin{cor}
\label{cor: iterationintro}
 The iteration map $I_l\colon\rat_d(\C)\to\rat_{d^l}(\C)$
 given by $I_l(f)=  f^l$ is proper for any $d\ge2$ and any $l\in\N$. 
\end{cor}

A few comments on our theorem are in order. The Stone-Čech compactification of $\N$ is characterized by the universal property that 
any map $f\colon \N\to K$ to a compact space $K$ extends continuously $\bar{f} \colon \beta\N\to K$. It can be defined as the set of all
ultra-filters on $\N$. The Banach ring $A^\epsilon$ is a  product of countably many copies of $\C$ with the norm $|\cdot|^{\epsilon_n}$,
and Berkovich proved that the spectrum $\cM(A^\epsilon)$ is isomorphic to $\beta \N$, see~\cite[Proposition 1.2.3]{berkovich2012spectral}. The map $\pi\colon \P^{1,\mathrm{an}}_{A^\epsilon}\to\beta\N$ 
is the composition of the canonical map $\pi\colon \P^{1,\mathrm{an}}_{A^\epsilon}\to\cM(A^\epsilon)$ with the homeomorphism  $\cM(A^\epsilon)\simeq\beta \N$.

The field $\sH(\omega)$ is a complex Robinson field, and appears classically 
in non-standard analysis\footnote{Ducros, Hrushovski and Loeser~\cite{DHL23} have recently studied non-Archimedean integrals as limits of complex ones.
They also used a non-standard model of the field of complex numbers endowed with both an archimedean and non-Archimedean to analyze these, although
in a way different to us.}. It turns out to be both spherically complete and algebraically closed. A map $g \in\Rat_d(\sH(\omega))$ has good reduction if 
its  Julia set $J_g$ is reduced to a single point.  By~\cite[Theorem C]{FR10}, it is equivalent to say that $g$ has zero entropy, or that it is defined by homogenous polynomials
$P_\omega, Q_\omega$ normalized so that their coefficients have norm $\le 1$, and their reductions in the residue field of $\sH(\omega)$ defines a rational
map of degree $d$.
A map has potential good reduction if it is conjugated to a map having good reduction. The statement that $f_\omega$ does not have potential good reduction reflects the 
fact that the classes $[f_n]$ are degenerating in the moduli space $\rat_d(\C)$.

The continuity of Lyapunov exponents might be surprising as there exists degenerating sequences of rational maps $\mf_n$
such that $\chi_{\mf_n}\to\infty$ whereas the Lyapunov exponent $\chi_\omega$ appearing in~\eqref{eq:conti-lyap} is always finite. 
Choose any integer $n\in\N$. In fact, we have the relation $\chi_n= \epsilon_n \times \chi_{f_n}$, so that an alternative way to express the continuity of Lyapunov exponent in the hybrid space is to say that for any ultra-filter $\omega\in\beta\N$, the sequence $\epsilon_n \times \chi_{f_n}$ converges to the non-Archimedean Lyapunov exponent $\chi_{f_\omega}$ along $\omega$.

\medskip

\noindent \emph{Comparison with previous approaches.}
Suppose first that $(f_t)_{t\in\D^*}$ is a meromorphic family, 
and $f_{\na}\in\Rat_d( \C((t)))$ does not have potential good reduction. Pick any sequence $t_n\in\D^*$ tending to $0$. 
Then it is not difficult to show that 
\[\epsilon_n= -\log (|\res(f_{t_n})|/C_d)^{-1}\asymp -\log |t_n|^{-1}.\]
In this case, for any $\omega \in \beta\N\setminus \N$ the field $\sH(\omega)$ is
a metrized field extension of $\C((t))$, and $f_\omega$ is the base field
extension of $f_{\na}$ to $\sH(\omega)$. The continuity of the equilibrium measure and the Lyapunov exponents
generalizes results by Demarco-Faber and the first author mentioned above.

Let us now compare our approach with the ideas of Luo. 
We shall prove:
\begin{thm}
\label{thm:equivalence}
There exists a constant $C>1$ such that 
\[
\frac1{C}
\le
\frac{1+\rluof(\mf)}{-\log(|\res(\mf)|/C_d)}
\le C
\]
for all $\mf\in\rat_d(\C)$.
\end{thm}
This statement says that our choice of normalization $\epsilon_n$ coincides with Luo's choice up to a bounded universal constant. 
We shall also prove (Theorem~\ref{theorem:asympH3}) that Luo's ultra-limits of $f_n$ are identical to our $f_\omega$ for any $\omega\in\beta\N$.
Since  our techniques bypass the use of hyperbolic geometry, we expect it  to be amenable to various other contexts, such as in higher dimensions 
and over different base fields.

\subsection*{Organization of the text}
In \hyperref[sectionberkovich space]{\S 1}, we recall some basic notions on Berkovich analytic spaces over a general Banach ring, and discuss in details the case of the projective line. We then briefly describe potential theory on the Berkovich projective line over a metrized field, with an emphasis on how the Laplacian operator varies when we replace a norm on a field by a power of it. 

In \hyperref[sec:rational maps on the projective line]{\S 2}, we define the moduli space of rational maps, and discuss the function $|\res|$, proving it is continuous using basic results in geometric invariant theory. In the non-Archimedean case, this function is directly related to the minimal resultant function introduced by Rumely~\cite{rumely2013minimal}.
We conclude this section by recalling how to define the equilibrium measure of a rational map defined over a metrized field.

The aim of \hyperref[sectionnonstandardanalysis]{\S 3}, is to prove the first two items of Theorem~\ref{thm: sequential hyridation construction}. We begin with a discussion of the Stone-Čech compactification including the notion of ultra-filters. We then discuss the structure of the Berkovich projective line over a product Banach ring of the form $A^\epsilon$ as above. We explain how to choose the sequence $\epsilon_n$ from a given sequence $\mf_n\in\rat_d(\C)$ so that 
it  naturally induces an endomorphism of $\P^1_{A^\epsilon}$, and prove the properness of the iteration map (Corollary~\ref{cor: iterationintro}). 
Note that our argument to show that 
$f_\omega$ cannot have potential good reduction when $\mf_n$ degenerates relies on the minimal resultant function of Rumely.

The next section \hyperref[sec:family of measures over a Banach ring]{\S 4} is technical in nature. It contains a discussion on continuous families of measures on the hybrid space $\P^{1,\mathrm{an}}_{A^\epsilon}$. We introduce a notion of model (and quasi-model) functions that play the role of smooth functions in complex manifolds and of model functions in Berkovich analytic spaces (see, e.g.,~\cite[\S 2.5]{BFJ16}). Using Stone-Weierstrass theorem, we prove that model functions form a dense subset of the set of continuous functions on $\P^{1,\mathrm{an}}_{A^\epsilon}$. We then prove the continuity of the push-forward of a continuous function by a rational map defined over $A^\epsilon$ (Proposition~\ref{prop:push-beta}), a technical yet crucial result to obtain the continuity of equilibrium measures.

In \hyperref[sec:convergence of measures]{\S 5}, we analyze the convergence of measures 
in the hybrid space $\P^{1,\mathrm{an}}_{A^\epsilon}$, and prove in particular the third item in Theorem~\ref{thm: sequential hyridation construction}. Our strategy follows closely~\cite{zbMATH07219254CFdegeneration,poineau2024dynamique}. 
First we consider a sequence of smooth volume form $\mu_n$ on $\pi^{-1}(n)$ arising from metrics $g_n$ with constant positive curvature on the Riemann sphere. In Theorem~\ref{convergence-conformal}, we describe the possible limits of $\mu_n$ on the non-Archimedean fibers $\pi^{-1}(\omega)$
 in terms of behaviour of the point $x_n\in\H^3_\R$ determined by $g_n$. 
 Using the continuity of the push-forward proved in the previous section, and the convergence of potentials defining the equilibrium measures, we prove
the continuity of equilibrium measures of an endomorphism defined over $A^\epsilon$, as
well as the continuity of the Lyapunov exponents.

The last section \hyperref[sec:luo's approach]{\S 6} is devoted to the proof of Theorem~\ref{thm:equivalence}, and to the comparison of our construction with that of Luo.
We exploit Theorem~\ref{convergence-conformal} on the convergence of measures associated with metrics of constant curvature to build a direct connection between the ultra-limit of $\H^3_\R$ along a ultra-filter $\omega$, and the Berkovich projective line $\P^{1,\mathrm{an}}_{\sH(\omega)}$ (Theorem~\ref{theorem:asympH3}).

\subsection*{Acknowledgements}
We warmly thank Marco Maculan for discussions on
GIT on non-Archimedean fields; and to Mattias Jonsson for his insightful
comments.

\subsection*{Notations}

\begin{itemize}
\item $B$ a Banach ring,  $B^\times$ its group of units, $\cM(B)$ the Berkovich spectrum of $B$, see~\S \ref{berkovich spaces over a Banach ring}).
\item $\P^{1,\mathrm{an}}_B$ the Berkovich projective line over $B$, $\hat{\C}$ the Riemann sphere, see~\S\ref{sec:P1overring}.
\item $\C_\epsilon$ the complex field endowed with the norm $|\cdot|^\epsilon$, where $|\cdot|$ is the standard Euclidean norm on $\C$; $s_\epsilon$ the canonical homeomorphism from $\P_{\C_\epsilon}^{1,\mathrm{an}}$ to the Riemann sphere; $s_\epsilon^\#$ the morphism between the sheaf over the Riemann sphere and $\P_{\C_\epsilon}^{1,\mathrm{an}}$, see~\S\ref{sec:The Archimedean case}.
\item $\H_k$ the hyperbolic space of $\P_k^{1,\mathrm{an}}$; $d_k$ the hyperbolic distance on $\H_k$; $d_{\P^1(k)}$ the spherical metric on $\P^1(k)$, see~\S\ref{sec:Berkovich projective line over an algebraically closed metrized field}, \S\ref{sec:Berkovich projective line over field}.
\item $\Rat_d$ the space of rational maps of degree $d\ge1$; $\rat_d$ the moduli space of degree $d$; $|\Res|\colon \Rat_d^{an}\to\R_+$ the resultant function of a rational map; $|\res|\colon \rat_d^{an}\to\R_+$ the minimal resultant function, see~\S\ref{sec:The space of rational maps and the resultant}, \S\ref{sec:minimalresultant}.
\item $\mu_{can}$ the canonical measure on $\P_k^{1,\mathrm{an}}$, i.e., the Haar measure on the unit circle if $k$ is Archimedean, and the Dirac measure supported on the Gauss point if $k$ is non-Archimedean; $\mu_f$ the equilibrium measure of a rational map $f$, see~\S \ref{sec:equilibriummeasure})
\item $\beta\N$ the Stone-Čech compactification of $\N$; $A^\epsilon$ the product Banach ring associated with a sequence $\epsilon\in(0,1]^\N$, see~\S\ref{sec:The Stone-Čech compactification of N}, \S\ref{sec:productbanachring}.
\item $M^+(\P^1,B)$ the space of continuous families of positive measures on $\P_{B}^{1,\mathrm{an}}$; $\cD(\P^{1,\mathrm{an}}_B)$ the set of model functions on $\P_{B}^{1,\mathrm{an}}$, see~\S\ref{sec:continuousfamilyofmeasures}, \S\ref{sec:modelfunction}.
\item $\H^3$ the upper-half space model of the hyperbolic space of dimension $3$; $\bar{\H}^3$ the hyperbolic space along with its conformal boundary; $d_\H$ the hyperbolic distance on $\H^3$; $\mathcal{G}(\C)$ the space of conformal metrics on the Riemann sphere
that have constant curvature $4\pi$; $\bar{\D}(x)$ the projective disk on $\P^{1}(\C)$ induced by $x\in\H^3$, see~\S\ref{sec: hyperbolicgeometry}.
\item $x_\star$ the point $(0,1)\in\H^3$; $\mu(x)$ the probability measure on the Riemann sphere which is fixed by any Möbius transformation fixing $x\in\H^3$; $\mu_{FS}$ the Fubini-Study measure, see~\S\ref{sec: hyperbolicgeometry}.
\item  $|df|$ the derivative of $f$ with respect to the spherical metric on $\P^1(k)$; $\chi_f$ the Lyapunov exponent of $f$, see~\S\ref{sec:convergenceoflyapunovexponent}.
\item $\rluo$ Luo's degeneration map on $\Rat_d(\C)$; $\rluof$ Luo's degeneration map on $\rat_d(\C)$; $\cE(f)$ the barycentric extension of $f$, see~\S\ref{sec:Barycentric extension}, \S\ref{sec:Asymptotic cone and Luo's degeneration map}.
\end{itemize}

\section{Berkovich spaces}
\label{sectionberkovich space}
 
\subsection{Analytic spaces over a Banach ring}
\label{berkovich spaces over a Banach ring}

A Banach ring $(B,\lV\cdot\rV)$ is by convention a commutative ring with unity $1$, endowed
with a norm $\lV\cdot\rV \colon B\to \R_+$ which is complete, such that 
$\lV a+ b \rV \le \lV a \rV + \lV b\rV$, and 
$\lV ab \rV \le K \lV a \rV\, \lV b\rV$ for some $K>0$ and for all $a,b\in B$,
see~\cite[A.1.2.1]{bosch1984non}. We do not require the norm to be multiplicative.

A multiplicative semi-norm $|\cdot|\colon B\to \R_+$ on a Banach ring is bounded if there exists $C>0$ such that for all $b\in B$, we have 
$|b|\leq C\lV b \rV.$
Observe that if $\lV\cdot\rV$ is power multiplicative, i.e., if $\lV b^n\rV=\lV b\rV^n$ for all $b\in B$ and $n\in\N$, then $|\cdot|$ is bounded if and only if $|b|\leq \lV b\rV$ for all $b\in B$.

The Berkovich spectrum $\cM(B)$ is the set of all bounded multiplicative seminorms on $B$ endowed with the weakest topology with respect to which all real valued functions on $\cM(B)$ of the form $|\cdot|\rightarrow |f|,$ $f\in B$, are continuous. By~\cite[Theorem 1.2.1]{berkovich2012spectral}, it is a compact space. It also carries a structural sheaf that we define in more generality below. 

\smallskip

We now recall the notion of analytic spaces over $B$
following~\cite{lemanissier2022espaces}.
We first define the affine space $\A_{B}^{n,an}$ of dimension $n$ over $B$. As a topological space, it is the set of all multiplicative semi-norms on $B[T_{1},...,T_{n}]$ whose restrictions to $B$ belongs to $\cM(B)$. We endow $\A_{B}^{n,an}$ with the weakest topology making all real valued functions of the form: $|\cdot|\rightarrow |P|$, $P\in B[T_{1},...,T_{n}]$ continuous.  

If $x\in \A_{B}^{n,an}$, $P\in B[T_{1},...,T_{n}]$, then we usually denote by $|P(x)|=|P|_{x}\in \R_+$ the value of $P$ at $x$. Consider the prime ideal \[\ker(x)=\left\{P\in B[T_{1},...,T_{n}],|P|_{x}=0\right\}.\] 
The complete residue field $\sH(x)$ of $x$ is defined as the completion of the fraction field of $B[T_{1},...,T_{n}]/\ker(x)$ with respect to the norm induced by $x$. 

We now define the structural sheaf on $\A_{B}^{n,an}$. For any open subset $U\subseteq \A_{B}^{n,an}$, let $S_{U}$ be the elements in $B[T_{1},...,T_{n}]$ that do not vanish on
$U$. We set $\cK(U)=S_{U}^{-1}B[T_{1},...,T_{n}]$ and define $\cO(U)$ to be the set of applications 
\[f\colon U\rightarrow \bigsqcup_{x\in U}\sH(x)\]
such that for all $x\in U$, $f(x)\in\sH(x)$ and there exists a neighbourhood $V_{x}$ such that $f|_{V_x}$ is a uniform limit of elements in $\cK(V_x).$ One can verify that the map sending $U$ to $\cO(U)$ is a sheaf~\cite{berkovich2012spectral}.

An analytic map $f\colon U\to V$ between two open sets of $\A_{B}^{n,an}$
and $\A_{B}^{m,an}$ is a morphism of ringed spaces such that 
the induced residue field extension $f^\#\colon \sH(f(x)) \to \sH(x)$
is isometric for all $x\in U$.

Lemanissier and Poineau~\cite[Chapitre~2]{lemanissier2022espaces} defined the category of $B$-analytic spaces as follows. Local models $(X,\cO_X)$ are determined by the choice of an open subset $U\subset \A_{B}^{n,an}$, and a coherent subsheaf $\cI$ of the structural sheaf of $U$
such that $X$ is the support of the quotient sheaf $\cO_U/\cI$, and 
$\cO_X=\cO_U/\cI$. Note that we have a canonical closed immersion $X\hookrightarrow U$. 

A $B$-atlas on a locally ringed space is an open cover $U_i$ where each $U_i$ is isomorphic to a local model, and the patching maps are $B$-analytic. 
A $B$-analytic space is then a locally ringed space with an equivalence 
class of $B$-atlas. 

Since we shall only use the projective line over $B$, we do not develop the general theory of $B$-analytic spaces here, referring to op. cit.

Observe that any $B$-analytic space $X$ is equipped with a canonical (continuous) structure map $\pi\colon X\to \cM(B)$. The map $\pi\colon\A_{B}^{n,an}\rightarrow \cM(B)$ is sending a semi-norm on $B[T_1, \cdots , T_n]$ to its restriction on $B$. We extend it to each local model $U$ by composition with the immersion $U\subseteq \A_{B}^{n,an}$.  These maps patch together on a general 
$B$-analytic space.

The fiber $\pi^{-1}(x)$ in $\A_{B}^{n,an}$ is canonically isomorphic
(as an analytic space) to $\A_{\sH(x)}^{n,an}$, so that 
the fiber $\pi^{-1}(x)$ in a general $B$-analytic space $X$
is a Berkovich analytic space defined over the complete field $\sH(x)$.
\begin{rmk}
When $B$ is a metrized field, then the category of $B$-analytic spaces consists of $k$-analytic spaces without boundary in the sense of Ber\-kovich~\cite{berkovich2012spectral}. 
\end{rmk}

\begin{rmk}
Under additional assumptions on the Banach ring,  
(for example when $B$ is either a metrized field, or a ring of integers of number field, or a hybrid field, etc.), then the 
structure sheaf of any $B$-analytic space enjoy many good properties: local rings are Henselian, Noetherian, and excellent by \cite[Corollary 2.5.2]{poineau2010droite} and  \cite[Corollaries 8.19 and 9.3]{poineau2013espaces}; 
the structural sheaf is coherent by \cite[Corollary 10.10]{poineau2013espaces}; the affine spaces $ \A_{B}^{n,an}$ are
 locally arcwise connected~\cite[Theorem 7.2.17]{lemanissier2022espaces}.
\end{rmk}

A standard example for analytic spaces is the Berkovich analytification of an algebraic variety $X$ defined over a metrized field $(k,|\cdot|)$, see~\cite[\S 3.4, 3.5]{berkovich2012spectral} for details. When $X = \spec A$ is an affine variety, and $A$ is a finitely generated $k$-algebra, then $X^{an}$
is by definition the space of all multiplicative semi-norms on $A$ whose restriction to $k$ is $|\cdot|$.

\subsection{The Berkovich projective line over a Banach ring}
\label{sec:P1overring}
We define the Berkovich projective line over a general Banach ring $(B,\lV\cdot\rV)$ and discuss its structure. 

\subsubsection{General construction of the Berkovich projective line  }
The projective Berkovich line $\P_{B}^{1,\mathrm{an}}$ over $B$ is the $B$-analytic space given by a $B$-atlas containing two charts 
$X_0$ and $X_1$, each isomorphic to $\A_{B}^{1,\mathrm{an}}$
and patched along $\A_{B}^{1,\mathrm{an}}\setminus\{0\}$
as follows. A point $x_{0} \in X_{0}\setminus\{0\}$ (resp. $x_{1}\in X_{1}\setminus\{0\}$) is a semi-norm on $B[z_0]$ (resp. on $B[z_1]$), and 
we identify them when \[|P(1/T)|_{x_0}=|P(T)|_{x_1}\]
for all polynomial $P\in B[T]$.
This construction naturally gives homogeneous coordinates $[z_{0}\colon z_{1}]$ on $\P_{B}^{1,\mathrm{an}}$. 
To simplify notation we usually write 
$z=z_0$, $0 = [0\colon1]$ and $\infty= [1\colon0]$
so that $\P_{B}^{1,\mathrm{an}}= \A_{B}^{1,\mathrm{an}} \sqcup \{\infty\}$.

Every homogeneous polynomial $P(z_{0},z_{1})\in B[z_{0},z_{1}]$ of degree $d$ determines a function $\psi_{P}\colon \P_{B}^{1,\mathrm{an}}\to \R$:
\[\psi_{P}(x)= 
\frac{|P(z_{0},z_{1})|}{\max\{|z_{0}|^{d},|z_{1}|^{d}\}}= \begin{cases}
    \frac{|P(z_0,1)|_{x}}{\max\{|z_0|_{x}^{d},1\}},&x\in X_{0}\\
     \frac{|P(1,z_1)|_{x}}{\max\{1,|z_1|_{x}^{d}\}},&x\in X_{1}.
\end{cases}\]
Observe that $\psi_P$ is a continuous function vanishing exactly on $\{P=0\}$ and that the topology on $\P_{B}^{1,\mathrm{an}}$ is the weakest topology such that all the functions $\psi_{P}$ are continuous. 
The space $\P_{B}^{1,\mathrm{an}}$ is compact.

The canonical continuous map 
\[\pi\colon\P_{B}^{1,\mathrm{an}}\rightarrow\cM(B)\] sending a point $x\in \P_{B}^{1,\mathrm{an}}$ to its restriction to $B$ is proper, and for all $s\in \cM(B)$, the fiber $\pi^{-1}(s)$ is canonically identified with $\P_{\sH(s)}^{1,\mathrm{an}}$.

For the record, recall that a point $p\in \P^1(B)$ is determined by a pair $z_0,z_1 \in B^2$ such that $z_0B+z_1B=B$, and this pair is unique up to multiplication by a unit in $B^\times$.

\subsubsection{Berkovich projective line over an algebraically closed me\-trized field $k$}
\label{sec:Berkovich projective line over an algebraically closed metrized field}
More details can be found in~\cite{Jonsson}.

Supppose first that the field $k$ is Archimedean. In this case, by Gelfand-Mazur's theorem and Ostrowski's theorem, there exists $0<\epsilon\leq 1$ such that $k$ is isometric to $\C_\epsilon=(\C,|\cdot|^{\epsilon}),$ where $|\cdot|$ is the standard Euclidean norm on $\C$. When $\epsilon=1$, then $\P_{k}^{1,\mathrm{an}}$ is the Riemann sphere $\hat{\C}$; when $\epsilon \in (0,1)$ then the map on $X_i$ sending a semi-norm $|\cdot|_x$ on $\C[z_i]$ to $|\cdot|^{1/\epsilon}_x$
induces a homeomorphism $s_\epsilon \colon \P_{\C_\epsilon}^{1,\mathrm{an}} \to \hat{\C}$.

\smallskip

Assume now that the field $k$ is non-Archimedean. In that case, a point $x \in \P_{k}^{1,\mathrm{an}}= \A_{k}^{1,\mathrm{an}} \sqcup \{\infty\}$ can be classified as follows. 
Note first that $x$ is either $\infty$ or a multiplicative semi-norm 
on $k[z]$ whose restriction to $k$ is the standard norm.
\begin{itemize}
    \item If the norm on $k$ is trivial, that is, for any $a\in k^*$, $|a|=1$, then $x$ is of the following form:
    \begin{itemize}
 \item[(T1)] Either $x=\infty$, or there exists $a\in k$ such that  $|P|_{x}=|P(a)|$ for all $P\in k[z]$;
 \item[(T2)] $x$ is the trivial norm, that is, $|P|_x=1$ for any non-zero $P\in k[z]$;
\item[(T3)] there exists $a\in k$ and $0<r<1$, such that 
\[|P|_{x}=r^{m},\]
if $P(z)= (z-a)^m Q(z)$ with $Q(a)\neq0$.
        \end{itemize}
    \item If the norm is non-trivial, then $x$ falls into one of the following four categories:
   \begin{itemize}
       \item[Type-$1$]: $x=\infty$ or there exists $a\in k$ such that for any $P\in k[z]$, $|P|_{x}=|P(a)|$;
       \item[Type-$2$]: there exist $a\in k$ and $r\in |k^*|$ such that 
       \[|P|_{x}=\max_{z\in \bar{B}(a,r)}|P(z)|=\max\{|b_{i}|r^{i}\},\] where $P(z)=b_{d}(z-a)^{d}+...+b_1 (z-a)+b_{0}$, and $\bar{B}(a,r)$ is the closed ball of radius $r$ centered at $a$ in $k$;
      \item[Type-$3$]: there exist $a\in k$ and $r\in \R_{+}\setminus|k^*|$ such that for all $P\in k[z]$, $|P|_{x}=\max_{z\in \bar{B}(a,r)}|P(z)|$ as in the previous case;
       \item[Type-$4$]: there exists a sequence of decreasing closed balls $\bar{B}_{n}$  with $\bigcap \bar{B}_{n}=\emptyset$, such that $|P|_{x}=\inf_{n}\max_{z\in \Bar{B}_{n}}|P(z)|$ for all $P\in k[z]$.
   \end{itemize}
\end{itemize}
Usually, we write $\zeta(a,r)$ for the (Type-2 or-3) point associated with the ball $\bar{B}(a,r)$, and denote the Gauss point by $x_g= \zeta(0,1)$.

A Type-1 point is also referred to as a rigid point. 
Observe that the set of Type-1 points can be identified with $\P^1(k)$. 

\begin{rmk}
\label{rmktype4}
A field $k$ is spherically complete if for any decreasing sequence $\Bar{B}_{n}$ of closed disks, we have $\bigcap \Bar{B}_{n}\not=\emptyset$. When $k$ is spherically complete, there is no Type-4 point on $\P_{k}^{1,\mathrm{an}}.$
\end{rmk}

The compact space $\P^{1,\mathrm{an}}_k$ has a structure of $\R$-tree, in the sense that any two points can be joined by a unique continuous injective map $[0,1]\to \P^{1,\mathrm{an}}_k$ up to reparameterization (see~\cite[\S 3]{novacoski2012valuations} for a thorough discussion on $\R$-trees). In particular, it is homotopic to a point. A tree structure naturally induces a topology (also called the observer's topology) admitting as a basis of open sets 
connected components of complements of finitely many points
(see~\cite[\S 3.2]{zbMATH02122441valuative}). One can verify that the tree topology on $\P_{k}^{1,\mathrm{an}}$ coincides with the topology as a Berkovich analytic space. 

On a tree, one can naturally define the notion of endpoints, branched points and regular points (see~\cite[\S 3.1.2]{zbMATH02122441valuative}). In the trivially valued case, 
$\P^{1,\mathrm{an}}_k$ is star-like: it contains a single branched point (the trivial norm);  endpoints are exactly of the form (T1); and we have a single branch joining the trivial norm to an endpoint consisting of regular points of Type (T3). 
In the non-trivially valued case, Type-1 and-4 points are endpoints; Type-2 (resp. Type-3)  are branched (resp. regular) points.

\smallskip

In the remainder of this section, we assume that $k$ is not trivially valued. The hyperbolic space $\H_{k}$ is defined as the complement in $\P_k^{1,\mathrm{an}}$ of Type-1 points.
For any pair of Type-2 or-3 points $x,y\in\H_{k}$ defined by the closed balls $\bar{B}(a_1,r_1)$ and $\bar{B}(a_2,r_2)$ respectively, we 
let \begin{align*}
    d_{k}(x,y)=  2\log\max\{r_{1},r_{2},|a_{1}-a_{2}|\}-\log r_{1}- \log r_{2}.
\end{align*}
\begin{prop}
    The map $d_{k}(\cdot,\cdot)$ extends continuously to $\H_{k}\times \H_k$ and defines a complete distance
    on $\H_k$. The group of projective linear transformations $\PGL_2(k)$ acts by isometries on $(\H_k,d_{k})$ and transitively on Type-2 points. 
\end{prop}
We refer to~\cite[Proposition 2.29]{BR10} for a proof.
\subsubsection{Berkovich projective line over a general complete valued field}
\label{sec:Berkovich projective line over field}

Denote by $\hat{k}^{a}$ the completion of the algebraic closure $k^{a}$ of $k$. This is an algebraically closed complete field. Denote by $\mathrm{Gal}(k^{a}/k)$ the absolute Galois group of $k$. It acts by isometries
on $\hat{k}^{a}$ fixing $k$, hence continuously 
on $\P_{\hat{k}^{a}}^{1,\mathrm{an}}$ and by isometries on $\H_{\hat{k}^{a}}$ by the following formulas: 
\[|b_{d}z^{d}+...+b_{0}|_{\sigma(x)}= |\sigma(b_{d})z^{d}+...+\sigma(b_{0})|_{x},\]
where $b_i\in k$, $x\in \A^{1,\mathrm{an}}_k$ and $\sigma\in \mathrm{Gal}(k^{a}/k)$.
This action also 
preserves the type of points as discussed in the previous paragraph so that we may speak of 
the type of points in $\P_{k}^{1,\mathrm{an}}$.
Finally we have 
\[\P_{k}^{1,\mathrm{an}}=\P_{\hat{k}^{a}}^{1,\mathrm{an}}/\mathrm{Gal}(k^{a}/k),\]
by~\cite[Proposition 1.3.5]{berkovich2012spectral},
so that $\P_{k}^{1,\mathrm{an}}$ is still an $\R$-tree.

\medskip

We shall use at several places the spherical distance on $\P^{1}(k)$. 
When $k=\C_\epsilon$ with $\epsilon\in (0,1]$ is Archimedean, it is defined by the formula: 
\begin{equation}\label{eq:proj-dist-arch}
d_{\P^1(k)}([z_{0}\colon z_{1}],[w_{0}\colon w_{1}])=\frac{|z_{0}w_{1}-z_{1}w_{0}|}{(|z_{0}|^{2/\epsilon}+|z_{1}|^{2/\epsilon})^{\epsilon/2}(|w_{0}|^{2/\epsilon}+|w_{1}|^{2/\epsilon})^{\epsilon/2}}.
\end{equation}
This formula agrees with the standard spherical metric in the case $\epsilon=1$, and the exponents are chosen so that $d_{\C_\epsilon}=(d_{\C}\circ s_\epsilon)^{\epsilon}$.

When $k$ is non-Archimedean, then we set:
\begin{equation}\label{eq:proj-dist-nonarch}
d_{\P^1(k)}([z_{0}\colon z_{1}],[w_{0}\colon w_{1}])=\frac{|z_{0}w_{1}-z_{1}w_{0}|}{\max\{|z_{0}|,|z_{1}|\}\max\{|w_{0}|,|w_{1}|\}}.
\end{equation}
The diameter of $\P^1(k)$ is equal to $1$, 
and the topology induced by this distance coincides with the Berkovich topology restricted to $\P^{1}(k)$.

\subsection{Potential theory on the Berkovich projective line over a field}\label{sec:potential theory}
In this section, $(k,|\cdot|)$ is an algebraically closed and complete valued field. We briefly recall how to define a Laplace operator on $\P^{1,\mathrm{an}}_k$, and discuss
what happens when one replaces $(k,|\cdot|)$
by $(k,|\cdot|^\epsilon)$ for some $\epsilon>0$. 

\subsubsection{The Archimedean case}
\label{sec:The Archimedean case}
Let $\C_\epsilon$
be the metrized field $(\C,|\cdot|^{\epsilon})$
with $\epsilon \in (0,1]$. Recall that $\hat{\C} = \P^{1,\mathrm{an}}_{\C_1}$ is the Riemann sphere, and that we have 
a homeomorphism 
$s_\epsilon \colon \P_{\C_\epsilon}^{1,\mathrm{an}} \to \hat{\C}$,  sending a semi-norm $|\cdot|$ on $\C_\epsilon[z]$
to $|\cdot|^{1/\epsilon}$. For any $x\in \P_{\C_\epsilon}^{1,\mathrm{an}}$, we have a canonical isometry 
$\tau_x\colon \sH(x)\simeq \C_\epsilon$.
Observe that  the identity map $I_\epsilon \colon \C_\epsilon \to (\C,|\cdot|)$ 
is Hölder-continuous (but is not an isometry).

Pick any open subset $U\subset\hat{\C}$ and any $f\in \cO(U)$ (by Mittag-Leffler' theorem $\cO(U)$ is the space of holomorphic functions on $U$). 
For any $x\in \P_{\C_\epsilon}^{1,\mathrm{an}}$, set  
$s_\epsilon^\# f (x)=   \tau_x^{-1}(I_\epsilon^{-1}(f(s_\epsilon (x))))\in \sH(x)$.
We get an isomorphism of $\C$-algebras
$s_\epsilon^\# \colon \cO(U) \to\cO(s_\epsilon^{-1}(U))$  
which satisfies
\begin{equation}\label{eq:compare-epsilon}
|s_{\epsilon}^{\#}h(x)|=|h\circ s_\epsilon (x)|^{\epsilon},
\end{equation}
for any $x\in s_\epsilon^{-1}(U)$.
In particular,  we get Hölder-continuous field isomorphisms $s_\epsilon^\#\colon \sH(s_\epsilon(x)) \to \sH(x)$ for all $x\in \P_{\C_\epsilon}^{1,\mathrm{an}}$. 

\smallskip

\paragraph{Subharmonic functions.}
Pick any $\epsilon \in (0,1]$, and any open subset $U$ of $\P_{\C_\epsilon}^{1,\mathrm{an}}$. The space of subharmonic functions $\SH(U)$ on $U$ is defined as the smallest class of upper-semicontinuous (usc)  functions that contains
$\log|h|$ for any $h\in \cO(U)$ and is stable under multiplication by positive constants, taking maxima, and decreasing limits. 
It follows from, e.g.,~\cite[Theorem 6.1]{zbMATH01500726subharmonic}, that in the case $\epsilon=1$ these functions are exactly those usc functions whose mean value on any round disk $D\subset U$ centered at a point $x$ exceeds $u(x)$. 
The latter definition being local, $\SH$ is in fact a sheaf on $\hat{\C}$, and by \eqref{eq:compare-epsilon}, 
we have an isomorphism
$s_\epsilon^\#\colon \SH(U)\to \SH(s_\epsilon^{-1}(U))$
so that $\SH$ is a sheaf on $\P^{1,\mathrm{an}}_{\C_\epsilon}$
for any $\epsilon$. 

\smallskip

\paragraph{Laplace operator.}
Let $\mathrm{M}^+(U)$ be the set of positive Borel measures on some open set $U\subset\hat{\C}$.
The Laplace operator on $\SH(U)$ can be defined as the differential operator $\Delta =  \frac{i}\pi\partial \overline{\partial}$, or as the unique
operator $\Delta\colon \SH(U) \to \mathrm{M}^+(U)$
which is continuous under decreasing limits and such that the Poincaré formula holds
\[
\Delta \log |h| = \sum_{x\in U} m_{x,h} \delta_x
\]
for any $h\in\cO(U)$ having a discrete set of zeroes, 
where $m_{x,h}\in \N$ is the multiplicity of $x$ as a zero of $h$.

One can proceed in the same way on $\C_\epsilon$, and 
define a Laplace operator $\Delta$ on $\SH(U)$ with $U\subset\P^{1,\mathrm{an}}_{\C_\epsilon}$ with values in $\mathrm{M}^+(U)$ which is continuous under decreasing limits and such that
\begin{equation}
\Delta \log |h|_\epsilon = \sum_{x\in U} m_{x,h} \delta_x
\end{equation}
for any $h\in \cO(U)$.
Comparing the latter equation with~\eqref{eq:compare-epsilon}, and observing that zeroes of $s_\epsilon^\#(h)$ are mapped by $s_\epsilon$ to zeroes of $h$ preserving the multiplicity,
we obtain for any $u\in \SH(U)$ with $U\subset \hat{\C}$ the identity:
\begin{equation}\label{eq:pot-epsilon}
(s_{\epsilon})_*\Delta (u \circ s_\epsilon) =
\frac1{\epsilon}\, \Delta u 
\end{equation}

\smallskip

\paragraph{Quasi-subharmonic functions.}
Since $\P^{1,\mathrm{an}}_{\C_\epsilon}$ is compact, the maximum principle implies that any subharmonic function on  $\P^{1,\mathrm{an}}_{\C_\epsilon}$ is constant. 
We thus define the class 
$\QSH(\P^{1,\mathrm{an}}_{\C_\epsilon})$
of usc functions $u\colon \P^{1,\mathrm{an}}_{\C_\epsilon} \to [-\infty, +\infty)$  such that locally near any point
$u$ is the sum of a smooth function and a subharmonic function. Since $\Delta$ is a local differential operator, one can define $\Delta u\in\mathrm{M}(\P^{1,\mathrm{an}}_{\C_\epsilon}) $ as a signed measure. 
One can prove that $\int_{\P^{1,\mathrm{an}}_{\C_\epsilon}} \Delta u =0$ for any $u\in\QSH(\P^{1,\mathrm{an}}_{\C_\epsilon})$. 
Conversely, pick any Borel measure $\mu\in \mathrm{M}(\P^{1,\mathrm{an}}_{\C_\epsilon})$ such that $\mu(\P^{1,\mathrm{an}}_{\C_\epsilon})=0$, and decompose it 
as a difference $\mu = \mu^+ - \mu^-$
such that $\mu^\pm \in \mathrm{M}^+(\P^{1,\mathrm{an}}_{\C_\epsilon})$, and $\mu^+ \perp \mu^-$,~\cite[6.14]{rudin}. 
Then  there exist
$u^\pm\in\QSH(\P^{1,\mathrm{an}}_{\C_\epsilon})$ such that 
$\Delta (u^+-u^-)= \mu$.
When $\mu^-$ is a (smooth) volume form, then 
$u^+-u^-$ is quasi-subharmonic, see, e.g.,~\cite[Theorem~10]{donaldson}.

\subsubsection{The non-Archimedean case}
To avoid trivialities, we shall assume that $(k,|\cdot|)$
is not trivially valued. 
Potential theory on curves has been developed in~\cite{zbMATH02122441valuative,BR10,thuillier2005theorie}.

\smallskip

\paragraph{Subharmonic functions.}
Pick any connected open subset 
$U\subset \P^{1,\mathrm{an}}_k$. Since the projective line is an $\R$-tree, $U$ is also a tree, and $U_\H =  U\cap \H_k$ endowed with the metric $d_{\H_k}$ is a metric $\R$-tree in the sense of \cite[\S 2.2]{Jonsson}. 
For the sake of convenience, we defined a subgraph of $U$ as a connected closed subtree $\Gamma$  having finitely many endpoints lying in $\H_k$. A function 
$u\colon \Gamma \to \R$ is subharmonic if it is continuous, its restriction to any segment of $\Gamma$ is convex, and the sum of its derivatives at all outward directions at a branched point is non-negative.
A function $u\colon U\to [-\infty,\infty)$
is subharmonic if and only if it is usc, and its restriction to any subgraph is subharmonic. 

As above, we denote by $\SH(U)$ the space of all subharmonic functions: it is stable under multiplication by positive constants, under taking maxima and by decreasing limits, and contains all functions  $\log|h|$ for any analytic functions $h\in \cO(U)$. Observe that $\SH$ defines a sheaf, and given any subgraph $\Gamma$ and any $u\in \SH(U)$, then 
$u\circ r_\Gamma \in \SH(U)$ where $r_\Gamma \colon U\to \Gamma$ denotes the canonical retraction.

When $U$ is an open ball, then a theorem of Stevenson~\cite[Example 4.8, Theorem 4.9]{zbMATH07051962stevenson} states $\SH(U)$
is the smallest class of usc functions stable under the preceding operations. The result is unclear in full generality. 

\smallskip

\paragraph{Laplace operator.}
Pick any subgraph $\Gamma\subset U$ and any $u\colon \Gamma \to \R$ subharmonic. Then $\Delta_\Gamma u$
is defined as the unique positive measure
whose restriction on any segment $e$ not containing a branched and end point is equal to the second derivative (in the distributional sense) of the convex function $u|_e$; and the mass at any branched or end point $v$ is the sum of the derivatives of $u$ at all outward directions at $v$.

There exists a unique linear map 
$\Delta \colon \SH(U) \to \mathrm{M}^+(U)$
that is continuous under decreasing limits and satisfies the following property. 
For any subgraph $\Gamma$ of $U$ and for any 
$u\in \SH(U)$, then $(r_\Gamma)_* \Delta u = \Delta_\Gamma u|_\Gamma$. 

With the same notation as in the previous section, Poincaré formula holds: 
\[
\Delta \log |h| = \sum_{x\in U} m_{x,h} \delta_x
\]
for any $h\in\cO(U)$ having a discrete set of zeroes.

\smallskip

\paragraph{Quasi-subharmonic functions.}
The notion of smooth functions do not make sense in non-Archimedean geometry. We say that $u\colon \P^{1,\mathrm{an}}_k\to \R$ is tropical if and only if there exists a subgraph $\Gamma$ 
such that $u|_\Gamma$ is piecewise linear, and
$u = u \circ r_\Gamma$. 

A function is said to be quasi-subharmonic if
it can be written locally as the sum of a subharmonic function and a tropical function. 
We let $\QSH(\P^{1,\mathrm{an}}_k)$ be the space of all quasi-subharmonic functions. For any $u\in \QSH(\P^{1,\mathrm{an}}_k)$, one can define $\Delta u$ as a signed Borel measure. 

As in the Archimedean case, for any signed Borel measure $\mu\in \mathrm{M}(\P^{1,\mathrm{an}}_k)$ such that 
$\mu(\P^{1,\mathrm{an}}_k)=0$, there exist
$u^\pm \in \QSH(\P^{1,\mathrm{an}}_k)$ such that 
$\Delta(u^+-u^-) = \mu$, see~\cite{BR10}. 

For any $x,y\in \P^{1,\mathrm{an}}_k$, denote by $x\wedge y$ the unique point $z$ such that $[x_g,x]\cap[x_g,y]=[x,z]$.
When $\mu = \rho - \delta_{x_g}$ where
$\rho$ is a probability (positive) measure, 
then 
\[g_{\rho}(x)=  -\int d(x_g,x\wedge y)d\rho(y) \in \QSH(\P^{1,\mathrm{an}}_k)\]
and $\Delta g_\rho = \mu$. 
We call the function $g_\rho$ the (global) potential of $\rho$.

\section{Rational maps on the projective line}
\label{sec:rational maps on the projective line}
\subsection{The space of rational maps and the resultant}
\label{sec:The space of rational maps and the resultant}
Let $k$ be an algebraically closed field. For each $d\in\N^{*}$, a rational map $f$ of degree $d$ with coefficients in $k$ is represented by a pair of 
homogeneous polynomials $P(z_{0},z_{1})$, $Q(z_{0},z_{1})\in k[z_{0}, z_{1}]$ of degree $d$ having no common
factor, such that the map \[[z_{0}\colon z_{1} ]\rightarrow [P(z_{0}, z_{1} )\colon Q(z_{0},z_{1})]\] represents the action of $f$ on
$\P_{k}^{1}$. Such a pair is unique up to scaling, so that $f$ naturally defines a point in
$\P^{2d+1}(k)$. The condition that $P$ and $Q$ are coprime can be expressed as the non-vanishing of their homogeneous resultant:
\begin{align}
\label{eqq: resultantdet}
    \Res(P,Q)=\det\begin{pmatrix}a_{0}&0&\cdots &0&b_{0}&0&\cdots &0\\a_{1}&a_{0}&\ddots &\vdots &\vdots &b_{0}&\ddots &\vdots \\\vdots &a_{1}&\ddots &0&\vdots &&\ddots &0\\\vdots &\vdots &\ddots &a_{0}&b_{d-1}&&&b_{0}\\a_{d}&&&a_{1}&b_{d}&\ddots &\vdots &\vdots \\0&\ddots &&\vdots &0&\ddots &b_{d-1}&\vdots \\\vdots &\ddots &a_{d}&\vdots &\vdots &\ddots &b_{d}&b_{d-1}\\0&\cdots &0&a_{d}&0&\cdots &0&b_{d}\\\end{pmatrix}
\end{align}
when $P(z_{0},z_{1})=a_{0}z_{0}^{d}+a_{1}z_{0}^{d-1}z_{1}+...+a_{d}z_{1}^{d}$, and $
      Q(z_{0},z_{1})=b_{0}z_{0}^{d}+b_{1}z_{0}^{d-1}z_{1}+...+b_{d}z_{1}^{d}$. Note that the resultant is homogeneous of degree $2d$ in the variables $a_i, b_j$.

It is thus natural to define the space $\Rat_d$ as the Zariski open subset of $[P\colon Q]\in \P^{2d+1}_k$ satisfying $\Res(P,Q)\neq0$. 
Its set of $k$-points is precisely the set of rational maps of degree $d$ on $\P^1_k$. Note that $\Rat_1$
coincides with the algebraic group $\PGL_2$. 

\smallskip

Assume now that $(k,|\cdot|)$ is a complete valued field. For any rational map $f\in\Rat_{d}(k)$, we  set
\[|\Res(f)|= \frac{|\Res(P,Q)|}{\max\{|a_{i}|^{2d},|b_{j}|^{2d}\}} \in \R_+^*\]
with the same notation as above.  Observe that 
this value is well-defined because of the homogeneity properties of the resultant.

For convenience, we say that 
a pair of homogeneous polynomials
$P,Q$ of degree $d$ is a \emph{normalized representation} of a rational map $f\colon \P^1_k\to \P^1_k$ of degree $d$,  if $f=[P\colon Q]$, and 
$\max\{|a_{i}|,|b_{j}|\}=1 $ where
$P=\sum_i a_i z_0^iz_1^{d-i}$ and $Q=\sum_i b_i z_0^iz_1^{d-i}$.
Observe that $|\Res(f)|=|\Res(P,Q)|$
for any normalized representation of $f$.

\begin{prop}
\label{propord}
     The function $-\log|\Res|\colon \Rat_{d}(k)\rightarrow \R$ can be continuously extended to the Berkovich analytification of $\Rat_{d}$. It is plurisubharmonic, proper and bounded from below.
\end{prop}
In the Archimedean case, plurisubharmonic functions
are defined as usc function whose restriction to any analytic curve is subharmonic, see~\cite{demailly1997complex} for basics on this class. In the non-Archimedean case, the general theory has been developed by Ducros and Chambert-Loir in~\cite{chambertloir2012formes}. Suffice it to say that 
the class of psh functions contains $\log|h|$ for any analytic map $h$, and is stable under taking maxima. 
\begin{proof}
Keeping the same notation as above, we first observe that
\[s_{i}=\frac{a_{i}^{2d}}{\Res(P,Q)}, t_{j}=\frac{b_{j}^{2d}}{\Res(P,Q)}\]
define regular functions on $\Rat_{d}$. 
It follows that $-\log |\Res|=\log \max\{|s_i|,|t_j|\}$ 
extends canonically to $\Rat_{d}^{an}$ as a continuous function, and is plurisubharmonic.

Observe that $|\Res|$ extends as a continuous function to $\P^{2d+1,an}_k$, and  $\Rat_d^{an}=\{|\Res|>0\}.$ It follows that $-\log |\Res|$ is proper and bounded from below on $\Rat_d^{an}$. 
\end{proof}

\subsection{The moduli space and the minimal resultant}
\label{sec:minimalresultant}
Let $(k,|\cdot|)$ be any complete valued field which is algebraically closed, and pick any integer $d\ge2$.

The group $\PGL_2$ has a natural left action on $\Rat_d$ by conjugacy: $M\cdot f =  M\circ f\circ M^{-1}$.
In order to apply results from geometric invariant theory, it is better to work with the action of the affine reductive group $\SL_2$ (note that the canonical morphism $\SL_2\to\PGL_2$ has finite kernel).

By Hilbert theorem, the ring of invariant regular functions 
\[k[\Rat_d]^{\SL_2}= \left\{h \in k[\Rat_d], \, h(M\cdot f) = h(f) \text{ for all } M,f\right\}\] is finitely generated. We may thus define the affine variety 
\[\rat_d=  \spec k[\Rat_d]^{\SL_2}.\] 
 Denote by $\pi\colon \Rat_d\to\rat_d$ the canonical map.

According to \cite{SJ98}, all $\SL_2$-orbits
are closed, and the stabilizer
of any element in $\Rat_d$ is finite, so that 
$\Rat_d$ is included in the set of stable points for the $\SL_2$-action. 
It follows that
$\rat_d(k)$ is naturally in bijection with conjugacy classes of rational maps of degree $d$ on $\P^1_k$. 
It is customary to write $[f]$ for the class in $\rat_d(k)$ of a rational map $f\in \Rat_d(k)$.
In order to clarify the difference between an actual rational map and a conjugacy class of rational maps, we shall use gothic fonts like $\mf$ to denote a
point in $\rat_d$.

We refer to~\cite{AL11} for more details on the geometry of $\rat_d$. 

\smallskip

Pick any $\mf\in\rat_{d}(k)$, and choose
$f\in\Rat_d(k)$ such that $[f]=\mf$. We define
\[-\log|\res(\mf)|=  \inf_{M\in\mathrm{PGL}_{2}(k)} -\log|\Res(M\cdot f)|\in\R.\]
This is well-defined because $-\log|\Res|$ is bounded from below by Proposition~\ref{propord}.

The next result follows quite directly from Kempf-Ness theory. We refer to~\cite[Chapter~3]{Ma17} for general results
on minimizing functions on orbits of action of general affine reductive groups over arbitrary metrized fields.
\begin{prop}
    \label{propminres}
    The function $-\log|\res|\colon \rat_{d}(k)\rightarrow \R$ extends as a continuous
    function on the Berkovich analytification of $\rat_d$.
    This extension is proper and bounded from below.
\end{prop}
\begin{rmk}
Except if $d=2$ (see~\cite{zbMATH06789350modulisingular,JM93,SJ98}), the algebraic variety $\rat_d$ is singular. One can still define a notion of psh function in this singular setting (see~\cite{zbMATH03117746plurisubharmonic} for the Archimedean case, and~\cite{chambertloir2012formes}
for the non-Archimedean one). It is not clear whether $-\log|\res|$ is psh on $\rat_d$.   
\end{rmk}
We shall need the following lemma.
\begin{lemma}\label{lem:proper-GIT}
The map 
\[\Psi\colon \SL_2\times \Rat_d \to \Rat_d \times \Rat_d\]
defined by $\Psi(M,f)= (M\cdot f, f)$ is proper. Moreover, there exist some positive real numbers $C,\alpha>1$ such that 
\begin{multline}
\label{eq:loj-GIT}
C \min\{|\det(M)|, |\Res(f)|\}^{1/\alpha}
\ge 
\min\{|\Res(M\cdot f)|,|\Res(f)|\}
\\
\ge 
C^{-1} \min\{|\det(M)|, |\Res(f)|\}^\alpha.
\end{multline} 
\end{lemma}
Note that since $\SL_2 \to\PGL_2$ has finite kernel, the induced map 
\[\Psi\colon \PGL_2\times \Rat_d \to \Rat_d \times \Rat_d\] is also proper.
\begin{proof}
This is a general result from geometric invariant theory.  All points in $\Rat_d$ are $\SL_2$-stable hence  
 $\Rat_d/\SL_2$ is a geometric quotient.  
 It follows from~\cite[Proposition~0.8]{MR0719371} that the map 
\[\Psi\colon \SL_2\times \Rat_d \to \Rat_d \times \Rat_d\]
defined by $\Psi(M,f)= (M\cdot f, f)$ is proper. 

The variety $\Rat_d$ is a Zariski open subset of $\overline{\Rat}_d= \P^{2d+1}$
and $\Psi$ extends as a rational map 
\[\bar{\Psi}\colon \overline{\Rat}_1\times \overline{\Rat}_d \dashrightarrow \overline{\Rat}_d \times \overline{\Rat}_d\]
Recall that $\PGL_2(\C)$ and $\Rat_d(\C)$ 
are the complex analytification of the respective $\C$-affine schemes, and keep $\Psi$ for the holomorphic map 
induced by the morphism above on $\PGL_2(\C)\times \Rat_d(\C)$.
We write $|\det|$ for the restriction of $|\Res|$ to $\Rat_1(\C)=\PGL_2(\C)$ as 
\[
|\Res(M)|
= \frac{|ad-bc|}{\max\{|a|,|b|,|c|,|d|\}^2}
\text{ if }
M=
\left(\begin{array}{cc}
 a    & b \\
  c   & d
\end{array}\right)
~.
\]
Now, the meromorphic map 
\[\bar{\Psi}\colon \overline{\Rat}_1(\C)\times \overline{\Rat}_d(\C) \dashrightarrow \overline{\Rat}_d(\C) \times \overline{\Rat}_d(\C)\]
is holomorphic on $X= \Rat_1(\C)\times\Rat_d(\C)$ and its restriction to $X$ is proper by GAGA.
Denote by $\Delta$ and $\Delta'$ the boundary divisors:
\begin{align*}
\Delta&:=\overline{\Rat}_1(\C)\times \overline{\Rat}_d(\C)\setminus \Rat_1(\C)\times\Rat_d(\C),
\\
\Delta'&:=\overline{\Rat}_d(\C)\times \overline{\Rat}_d(\C)\setminus \Rat_d(\C)\times\Rat_d(\C)~.
\end{align*}
Locally near any point 
$p\in \Delta'$,
there exists a polynomial $P$
such that $\{P=0\}$ cuts out the divisor $\Delta'$, and $|P|\asymp \min\{|\Res(f_1)|, |\Res(f_2)|\}$. 
In a similar way, for any point $q\in \Delta$ with  for any polynomial $Q$ cutting out the boundary divisor 
$\Delta$, we have
$|Q|\asymp \min \{|\det(M)|,|\Res(f)|\}$. 
Since $\bar{\Psi}$ is proper, $P\circ \bar{\Psi}$ cuts out 
$\Delta$, and 
by Lojasiewicz' inequalities, we may thus find positive real numbers $C,\alpha>1$ such that 
\begin{multline*}
C \min\{|\det(M)|, |\Res(f)|\}^{1/\alpha}
\ge 
\min\{|\Res(M\cdot f)|,|\Res(f)|\}
\\
\ge 
C^{-1} \min\{|\det(M)|, |\Res(f)|\}^\alpha.
\end{multline*} 
The proof is complete.
\end{proof}

\begin{proof}[Proof of Proposition~\ref{propminres} in the Archimedean case]
Suppose $k=\C$. 
Since
all points are stable, all orbits are Zariski closed hence closed for the euclidean topology on 
$\Rat_d(\C)$. 

For a fixed $f\in\Rat_d(\C)$, the map 
$\SL_2(\C) \to \Rat_d(\C)$ sending 
$M$ to $M\cdot f$ is proper by the previous lemma, hence the function 
$M\mapsto - \log|\Res(M\cdot f)|$ is also proper. It follows that 
$- \log|\Res|$ attains its minimum on the orbit of $f$.

Let us prove the properness of $-\log|\res|$. Choose any sequence $\mf_n\in\rat_d(\C)$
such that $\sup_n - \log|\res(\mf_n)|$ is bounded. For each $n$, choose $f_n\in\Rat_d(\C)$ such that $[f_n]=\mf_n$, and $- \log|\res(\mf_n)|=- \log|\Res(f_n)|$.
Since $-\log|\Res|$ is proper, $f_n$ is bounded, hence $\mf_n$ is also bounded.

To see the continuity of $-\log|\res|$ we proceed in a similar way. Suppose $\mf_n\to \mf \in \rat_d(\C)$.
Choose elements $f_n$ (resp. $f$) with $[f_n]=\mf_n$ (resp. $[f]=\mf$) minimizing $-\log|\Res|$ on their respective orbits.
Any cluster point of the $f_n$'s belongs to the orbit of $f$. It follows that
\[-\log|\res(\mf)| = -\log|\Res(f)|
\le \liminf_n -\log|\res(\mf_n)|\]
hence $-\log|\res|$ is lsc. 

To prove $-\log|\res|$ is usc, choose any sequence $f'_n$
such that $f'_n\to f$ and 
$[f'_n]=\mf_n$. To see that such a sequence exists, consider the $\SL_2(\C)$-invariant set $A:= \bigcup_n \pi^{-1}(\mf_n)$. If $A$ is not closed, it admits a cluster point $g$ which necessarily lies in the orbit of $f$, and
the existence of $f'_n$ follows. If $A$ is closed, then $\pi(A)$ is also closed in the euclidean topology by Kempf-Ness theorem (see, e.g.,~\cite[Theorem~1.6 (3) Chapter 3]{Ma17}), which contradicts the fact that $A=\pi^{-1}(\pi(A))$. 
Now we have 
\begin{align*}
\limsup_n -\log|\res(\mf_n)| 
&\le 
\limsup_n -\log|\Res(f_n')|
\\
&= -\log|\Res(f)|
=- \log|\res(\mf)|
\end{align*}
proving $-\log|\res|$ is also usc.
\end{proof}

\begin{proof}[Proof of Proposition~\ref{propminres} in the non-Archimedean case]
In this case, one has to be careful with the definition of orbits since points in $\Rat_d^{an}$ have different field of definitions. 
We have the following maps between $k$-varieties 
\[\begin{tikzcd}
\SL_2\times \Rat_d \arrow[d, "\mathrm{pr}_2"] \arrow[r, "\varpi"] & \Rat_d \\
\Rat_d  &                        
\end{tikzcd}\]
where $\mathrm{pr}_2$ denotes the projection onto the second factor, and 
$\varpi(M,f)=  M\cdot f$.   
This diagram is also valid in the analytic category.
The orbit of a point $f\in\Rat^{an}_d$ is by definition the set $\SL_2^{an}\cdot f = \varpi \mathrm{pr}_2^{-1}(f)\subset \Rat_d^{an}$.

Choose $f\in\Rat_d^{an}$, and denote by $\sH$ its complete residue field. Then $\mathrm{pr}_2^{-1}(f)$ is isomorphic to $\SL_{2,\sH}^{an}$, and we have a canonical map $\SL_{2,\sH}^{an} \to \Rat_d^{an}$
given by $M\mapsto M\cdot f =\mathrm{pr}_1(\Psi^{an}(M,f))$, where $\Psi$ is the map from Lemma~\ref{lem:GIT} whose image is the orbit of $f$. Since $\Psi$ is proper, its analytification is also proper, hence $M\mapsto M\cdot f$ is proper.
As in the complex case, it follows the function $-\log|\Res|$ is proper hence admits a minimum in the orbit of $f$. The properness of $-\log|\res|$ follows as in the archimedean case.

We aim at proving that $-\log|\res|$ is continuous. 
By Poineau~\cite{Po13}, $\rat^{an}_d$ is Fréchet-Urysohn
(meaning that for any subset $A$, and any point $x\in\bar{A}$ there exists a sequence $x_n\in A$ such that $x_n\to x$). It follows that one can show that $-\log|\res|$ is continuous using sequences.  
Now all arguments used in the Archimedean case works exactly the same way (including Kempf-Ness theory as discussed by Maculan in~\cite{Ma17}) and prove that $-\log|\res|$ is continuous.
 \end{proof}

\subsubsection*{Rumely's minimal resultant}
Suppose now that $k$ is an algebraically closed non-Archimedean field. In that case, the previous function $\log|\res|$ has strong connections with the minimal resultant function 
$\mathrm{ordRes}_{f}$ introduced by Rume\-ly in~\cite{rumely2013minimal}, and further studied in~\cite{zbMATH07643722JKminimal,zbMATH06333797SLTminimalresultant}.

The basic observation (see~\cite[Exercice~2.12]{Sil07}) is that for any $f\in\Rat_{d}(k)$, and for any $M\in\mathrm{PGL}_{2}(k^\circ)$, we have \[|\Res(f)|=|\Res(M\cdot f)|.\] 
Since $\PGL_{2}(k^\circ)$ is precisely the subgroup of $\PGL_2(k)$ fixing the Gauss point, and $\PGL_2(k)$ acts transitively on Type-2 points, we obtain a well-defined function $\mathrm{ordRes}_{f}$ on the set of Type-2 points which sends the point $M(x_g)$ to $-\log|\Res(M\cdot f)|$.

A theorem of Rumely~\cite{rumely2013minimal} states that  $\mathrm{ordRes}_{f}$ can be canonically extended as a function 
$\mathrm{ordRes}_{f}\colon \H_k\to[0,\infty]$
that is continuous with respect to the metric $d_{\H_k}$ and proper. Moreover, the locus where $\mathrm{ordRes}_{f}$ attains its minimum is either a single Type-2 point or a segment containing at least one Type-2 point.

Finally suppose that $\mathrm{ordRes}_{f}(x_g)=0$ (i.e., $|\Res(f)|=1$), and choose a normalized representation
$f= [P\colon Q]$. Then the resultant of $P$ and $Q$ has norm $1$, hence the reduction map $\tilde{f}= [\tilde{P}\colon\tilde{Q}]\in \Rat_d(\tilde{k})$
has degree exactly $d$. 
In that case, we say that $f$ has \emph{good reduction}. 

Altogether we obtain the following statement due to Rumely:
\begin{thm}\label{theorem:min-res-Rumely}
Let $(k,|\cdot|)$ be any  algebraically closed  non-Archimedean complete field, and pick $f\in\Rat_d(k)$. 

Then $-\log |\res([f])|=0$ if and only if $f$ has potential good reduction, 
i.e., is  conjugated by a projective transformation to 
a rational map having  good reduction. 
\end{thm}

\subsection{Rational map on $\P_{k}^{1,\mathrm{an}}$ and the equilibrium measure}
\label{sec:equilibriummeasure}

Let $(k,|\cdot|)$ be any complete and algebraically closed  metrized field, and pick any rational map $f\in\Rat_{d}(k)$
of degree $d\ge2$. This is an endomorphism of the projective variety $\P^1_k$ hence it induces a continuous map on the analytification $\P^{1,\mathrm{an}}_k$. In concrete terms, $f=[P\colon Q]$ maps a rigid point 
$[z_0\colon z_1]\in \P^1(k)$ to $[P(z_0,z_1)\colon Q(z_0,z_1)]$; and
a point $x\in \H_k$ having a trivial kernel to $f(x)$
so that $|h(f(x))| = |(h\circ f) (x)|$
for any rational map $h\in k(z)$. 

\smallskip

We aim at defining a canonical (equilibrium) measure which is invariant by $f$. The construction proceeds as follows. 
The endomorphism $f\colon \P_{k}^{1,\mathrm{an}}\rightarrow \P_{k}^{1,\mathrm{an}}$ is continuous, finite, open and surjective, see, e.g.,~\cite[Proposition 4.3]{Jonsson}. It follows that one can define a local degree at any point $x\in\P^{1,\mathrm{an}}_k$ by setting 
\[\deg_{x}(f)= \dim_{\kappa(f(x))}\left(\cO_{x}/m_{f(x)}\cO_{x}\right).\]
where $\cO_x$ is the local ring of analytic functions at $x$, $m_x$ its maximal ideal, and
$\kappa(x)= \cO_x/m_x$ (observe that $\sH(x)$ is the completion of $\kappa(x)$). 
\begin{prop}
\label{propprof}
For every connected open set $V$ and every connected component $U$ of $f^{-1}(V)$, the integer 
\[\sum_{f(y)=x,y\in U}\mathrm{deg}_{y}(f)\]
is independent of the point $x\in V$.   
In particular, for every $y\in\P_{k}^{1,\mathrm{an}}$, we have 
   \[\sum_{x\in f^{-1}(y)}\mathrm{deg}_{x}(f)=d.\]
\end{prop}
For any function $\varphi\colon \P_{k}^{1,\mathrm{an}}\to\R$, define
\begin{align}
\label{eq:pushforwardcontinuousfunction}
    f_{*}\varphi(x)=\sum_{y\in f^{-1}(x)}\mathrm{deg}_{y}(f)\varphi(y).\
\end{align}

The previous proposition implies that if $\varphi$ is continuous, then $f_*\varphi$ is continuous, and its sup-norm
is bounded by $d$ times the sup-norm of $\varphi$, so that for any Radon measure $\mu\in \mathrm{M}(\P_{k}^{1,\mathrm{an}})$, the pullback of $\mu$ by $f$ can be dually 
defined by the identity $\langle f^{*}\mu,\varphi\rangle= \langle\mu,f_{*}\varphi\rangle$
for any continuous function $\varphi$ (we use the notation $\langle\mu,\varphi\rangle= \int \varphi d \mu$).
We refer the readers to~\cite[\S 9.4]{BR10} for more details.

\smallskip 

Denote by $\mucan\in\mathrm{M}^+(\P^{1,\mathrm{an}}_k)$ the Haar (probability) measure on $\{|z|=1\}$ if $(k,|\cdot|)$ is Archimedean, or the Dirac mass at the Gauss point if it is non-Archimedean. 

By the discussion of \S\ref{sec:potential theory}, there exists a continuous function $g_f$ (which is the difference of two functions in $\QSH(\P^{1,\mathrm{an}}_k)$) such that 
\[
\frac1d f^*\mucan - \mucan =\Delta g_{f,0}~.
\]
In homogeneous coordinates, we have 
$f=[P\colon Q]$, and this function is explicitly given by 
\[
g_{f,0} [z_0\colon z_1] = \frac1{d}  
\log \frac{\max\{|P|,|Q|\}}{\max\{|z_0|^d,|z_1|^d\}}
~.\]
It follows that the sequence of probability measures $\frac1{d^n}f^{n*} \mucan$ converges, and we denote by $\mu_f\in \mathrm{M}^+(\P^{1,\mathrm{an}}_k)$ its limit. 
It can be written as 
\begin{equation}\label{eq:localpotential}
    \mu_f = \mucan + \Delta g_f
\end{equation}
where $g_f=\sum_{k\ge0} d^{-k} g_{f,0}\circ f^k$
is continuous. 
It also satisfies
$f^* \mu_f = d \mu_f$. One can also show that $\mu_f$ is mixing (hence ergodic), see, e.g.,~\cite{Ber13,FR10}.

 By definition, the Julia set $\cJ(f)$ of $f$ is the support of $\mu_f$, and the Fatou set 
$\cF(f)$ is its complement. 
 The Fatou set is open and the Julia set is compact. Both are totally invariant.

One can show that the Fatou set is precisely the locus where the iterates $\{f^n\}$ 
form a normal family (see~\cite[Theorem 5.4]{zbMATH06051033FKTmontel} in the non-Archimedean setting).

In the Archimedean setting, $\cJ(f)$ is a
perfect (uncountable) set. In the non-Archimedean setting, either $\cJ(f)$ is reduced to a singleton in which case $f$ has potential good reduction; or it is perfect and uncountable (see \cite[Theorem 8.15(e)]{Ben19}).

\section{Sequential hybridation using non-standard analysis}
\label{sectionnonstandardanalysis}
\subsection{The Stone-Čech compactification of $\N$}
\label{sec:The Stone-Čech compactification of N}
We start by discussing the notion of ultrafilters on the set $\N$ of natural numbers. We refer to \cite{comfort2012theory} for a detailed treatment. 
\begin{defi}
    A subset $\omega$ of the power set of $\N$ is called an ultrafilter, if 
    \begin{enumerate}
        \item $\emptyset\notin\omega$;
        \item if $E,F\in\omega$, then $E\cap F\in\omega;$
        \item if $E\in\omega$ and $E\subseteq F$, then $F$ is also in $\omega;$
        \item if $E\subseteq \N$, either $E$ or $E^c=  \N\setminus E$ is in $\omega.$
    \end{enumerate}
    \end{defi}

A set in $\omega$ is called an $\omega$-big set. A set whose complement lies in $\omega$ is said to be $\omega$-thin.
 If a proposition is satified for a $\omega-$big set, we say that this proposition is true $\omega-$almost surely.

 An ultrafilter is said to be principal if there exists an integer $n$ such that this ultrafilter is exactly $\{F\subseteq \N, n\in F\}$, we denote such an ultrafilter by $\omega_{n}.$
\begin{prop}
\label{existence ultrafilter}
For each infinite set $E\subseteq \N$, there exists a non-principal ultrafilter $\omega$ that contains $E$.
\end{prop}

\begin{proof}
Consider the family of sets $\{E\setminus F\}$
with $F$ finite. By~\cite[Theorem~7.1]{comfort2012theory}, there exists a ultrafilter $\omega$ containing all these sets (the existence is a consequence of Zorn's lemma).
It cannot be principal, hence the proposition.\end{proof}

\begin{defi}
    Let $X$ be a Hausdorff topological space, $\omega$ an ultrafilter and $x_{n}$ be a sequence on $X$. We define $x$ to be an $\omega$-limit of $x_{n}$, if, for all neighbourhood $V$ of $x$, $\{n\in\N,\text{$x_{n}\in V$}\}$ is an $\omega$-big set.
\end{defi}
We briefly summarize the properties of $\omega$-limits of a sequence here:
\begin{prop}
\label{omegalimit}Let $x_{n}$ be a sequence on a Hausdorff space $X$, then:
\begin{enumerate}
\item if $X$ is compact, the $\omega$-limit of $x_{n}$ exists and is unique;
 \item if $x=\lim_{n\rightarrow\infty}x_{n}$ in the usual sense, then $x=\lim_{\omega}x_{n}$ for any non-principal ultrafilter $\omega$;
\item if $X$ is first countable, and $x=\lim_{\omega}x_{n}$, then there exists a subsequence $x_{n_k}$ of $x_{n}$ such that 
    $$\lim_{k\rightarrow\infty}x_{n_k}=x.$$
\end{enumerate}
\end{prop}
\begin{proof}
If the first statement is not true, then for every $x\in X$, there exists a neighbourhood $V_{x}$ of $x$ such that 
    \[N_{x}=\{n\in\N, \text{$x_{n}\in V_{x}$}\}\] is not $\omega$-big. The entire space $X$ is covered by a finite number of such $V_{x}$'s. Therefore, $\N$ is covered by a finite number of subsets, each of which is not big. Since the intersection of two big sets is big, it follows that the union of two sets, each of which is not $\omega$-big, is also not $\omega$-big. This leads to a contradiction.
 
   The second statement follows from a simple observation that a non-prin\-cipal ultrafilter $\omega$ contains all cofinite sets. Otherwise, if $\{n_{1},\cdots,n_{k}\}$ belongs to the ultrafilter at least one of the $n_{i}$ is in $\omega$. If not, $\N\setminus \{n_{i}\}\in \omega$, hence 
   \[\N\setminus\{n_{1},...,n_{k}\}=\bigcap_{i=1}^{k}\N\setminus\{n_{i}\}\in\omega,\]
which contradicts the assumption.   
    For the last statement, let $U_{k}$ be a basis of neigbourhood of $x$, take $x_{n_k}\in U_{k}$, then $\lim_{k\rightarrow\infty}x_{n_k}=x.$
\end{proof}
\begin{rmk}
    Note that if $\omega$ is a principal ultrafilter associated with the integer $m$, then $\lim_{\omega}x_{n}=x_{m}.$
\end{rmk}

Denote by $\beta\N$ the set of all ultrafilters on $\N$. For each subset $E\subseteq \N$, define $U_{E}=\{\omega\in\beta\N,~E\in\omega\}.$ Note that for all $E,F\subseteq \N$, $U_{E}\cap U_{F}=U_{E\cap F}$. Hence, the collection $U_{E}$ forms a basis for a topology on $\beta\N$, and the open sets are unions of such $U_{E}'s.$ Also note that $\bigcup_{i=1}^{k}U_{E_i}=U_{\cup_{i=1}^{k}E_{i}}$.

We have a canonical injection $\N\to\beta\N$, sending $n$ to the principal ultrafilter $\omega_n$. 
Pick any $\omega \in \beta\N$ and
any open neighborhood $\omega$
of the form $U_E$. Then for any $n\in E$
we have $E\in \omega_n $, hence $\N$ is dense in $\beta\N$. 
Since $U_{n}=\{\omega_{n}\}$ is open,
it follows that $\N$ (with its discrete topology) is also open in $\beta\N$.

\begin{thm}\label{extension}
The topological space $\beta\N$ 
is compact and totally disconnected, and contains $\N$ (endowed with its discrete topology) as an open dense subset. 

It satisfies the following universal property. 
Any map $f\colon \N\to K$
with $K$ compact can be  
 continuously extended to 
 a map $F\colon \beta\N\to K$ by setting
 $F(\omega)= \lim_{\omega}f(n)$.
\end{thm}
The space $\beta\N$ is referred to as the Stone-Čech compactification of $\N$.

\begin{proof}
Let us show that $\beta\N$ is Hausdorff. 
Pick two ultrafilters $\omega \neq \omega'$. 
Then there exists $E\subset \N$
such that $E\in \omega$ and $E\notin \omega'$. 
Then $\omega \in U_E$, $\omega'\in U_{E^c}$ and 
we have $U_E\cap  U_{E^c}=\emptyset$.

Suppose $U_i$ is an open cover of $\beta\N$. 
We may suppose that $U_i = U_{E_i}$ for some 
$E_i \subset \N$. If we cannot find any finite sub-cover, then the collection $\{E_i^c\}$ has the finite intersection property: any finite intersection of $E_i^c$ is non-empty. It follows from~\cite[Theorem~2.2]{comfort2012theory} that 
there exists some $\omega$ containing all $E_i^c$, a contradiction. Hence $\beta\N$ is compact.

Pick any $\omega\neq \omega'\in\beta\N$ and  $E\in \omega$, $E^c \in \omega'$ as above. 
Thus $\{U_E,U_{E^c}\}$ is an open cover by disjoint open sets. 
Since $\omega\in U_E$ and $\omega'\in U_{E'}$,  $\beta\N$ is totally disconnected.

Finally given $f\colon \N\to K$, we define $F\colon \beta\N\rightarrow K$ by setting $F(\omega)= \lim_{\omega}f(n)$. We claim that this map is continuous. Let $U$ be a neighbourhood of $F(\omega)$. There exists an $\omega$-big set $E$ such that for all $n\in E$, $f(n)\in V$ with $V$ open and $\bar{V}\subset U$. Therefore, $F(U_{E})\subseteq U$, and the proof is complete.
 \end{proof}

\subsection{Product Banach rings}
\label{sec:productbanachring}
In all this section, we fix $\epsilon = (\epsilon_n)$ any sequence of positive real numbers in $(0,1]$. 

\subsubsection{The Berkovich spectrum and the Stone-Čech compactification.}
We define the product Banach ring associated with $\epsilon$ as follows:
\[A^{\epsilon}=\left\{(x_{n})\in\C^{\N}, |x_{n}|^{\epsilon_n}\text{ is bounded}\right\},\] 
where $|\cdot|$ denotes the standard Euclidean norm on the field $\C$ of complex numbers. We endow $A^{\epsilon}$ with the norm $\lV\cdot\rV$ defined by:
\[\lV(x_{n})\rV=\sup_{n}|x_{n}|^{\epsilon_n}.\]
This definition makes $A^{\epsilon}$ into a Banach ring whose norm is power multiplicative. Recall that its Berkovich spectrum
$\cM(A^\epsilon)$ is the collection of multiplicative seminorms $|\cdot|$ on $A^{\epsilon}$ such that  $|(x_{n})|\leq \sup_{n}|x_{n}|^{\epsilon_n}$.

Be aware that $A^{\epsilon}$ is a Banach $\C$-algebra if and only if $\epsilon = (1)$
because
\[\lV c\cdot (1)\rV=\sup_{n}|c|^{\epsilon_n}\]
for all $c\in \C$. 

\begin{rmk}
\label{rmkinverse}
An element $x=(x_{n})$ is a unit in $A^{\epsilon}$ if and only if $(\frac{1}{x_n})\in A^{\epsilon}$. This proves that the set of all units $(A^\epsilon)^\times$ is equal to those $x$ for which there exists $M>1$ such that $\frac{1}{M}\leq |x_{n}|^{\epsilon_n}\leq M$ for all $n.$
\end{rmk}
\begin{thm}\label{theorem:spectrumAeps}
The map $\jmath\colon \beta \N\rightarrow \cM(A^{\epsilon})$ sending a ultrafilter $\omega$ to the seminorm charaterized by 
    \[|(x_{n})_{n\geq 1}|_{\jmath(\omega)}=\lim_{\omega}|x_{n}|^{\epsilon_n}\]
    is a homeomorphism.
\end{thm}

The case of an arbitrary product of metrized fields is treated in~\cite[Proposition 1.2.3]{berkovich2012spectral}.
\begin{proof} 
The continuity of $\jmath$ is clear.

We first show that $\jmath$ is injective. Pick $\omega_{1}\not= \omega_{2}\in \beta\N$, and $E\subset \N$ such that $E\in\omega_{1}$ and $E^c\in\omega_{2}.$ Let $x_E\in A^\epsilon$ be the sequence such that 
$x_{n}=0$ on $E$ and $1$ on $E^c$.
We have
$|x_E|_{\jmath(\omega_{1})}=0$ and $|x_E|_{\jmath(\omega_{2})}=1$
so that $\jmath(\omega_{1})\not=\jmath(\omega_{2})$

We claim that $\jmath$ is surjective. 
We shall use the fact that for any semi-norm
$|b|=0$ implies $|a+b|=|a|$ for any $a$.

Observe that for any $y\in\cM(A^{\epsilon})$, 
and for any $E\subset \N$
we have $|x_{E}|_{y}\leq \lV x_{E}\rV=1$, and 
$x_{E}\,x_{E^{c}}=0$, $x_{E}+x_{E^{c}}=1$.
Since $|\cdot|_{y}$ is a multiplicative seminorm,  by the trick above either $|x_{E}|_{y}$ or $|x_{E^c}|_{y}$ equals $1$ while the other equals $0$.

Let $\omega=\{E\subseteq\N, |x_{E}|_{y}=1\}$. We shall prove that $\omega$ is an ultrafilter and that $y=\jmath(\omega)$. By the preceding argument, for any subset $E\subseteq \N$, either $E$ or $E^c$ belongs to $\omega$. Let $E,F\in\omega$, then $|x_{E\cap F}|_{y}=|x_{E}|_{y}|x_{F}|_{y}=1$, thus $E\cap F$ also belongs to $\omega$.  Additionally, if $E\in \omega$ and $E\subseteq F$, then $|x_{F}x_{E}|_{y}=|x_{E}|_{y}=1$. Consequently, we have $|x_{F}|_{y}=1$, implying $F\in \omega$. Therefore, $\omega$ is an ultrafilter.

Pick $x=(x_{n})\in A^{\epsilon}$, and set $\alpha=\lim_\omega|x_{n}|^{\epsilon_n}$. 
To conclude the proof it is sufficient to show that $|x|_y = \alpha$. 
For any $\eta>0$, there exists $E\in\omega$ such that for all $n\in E$, 
\[\alpha-\eta<|x_{n}|^{\epsilon_n}<\alpha+\eta.\]
Define \[x_{n}'=\begin{cases}
    x_{n},&\text{ if } n\in E\\
    \alpha^{1/\epsilon_n},&\text{ otherwise.}
\end{cases}\]
Then $|(x_{n}')|_{y}=|(x_{n})x_{E}+(x_{n}')x_{E^c}|_{y}=|(x_{n})x_{E}|_{y}=|(x_{n})|_{y}$.

Since $y$ is a bounded seminorm, we have
\begin{align*}
    |(x_{n}')|_{y}&\leq \sup|x_{n}'|^{\epsilon_n}\leq\alpha+\eta\\
    \bigg|(x_{n}')^{-1}\bigg|_{y}&\leq \sup\bigg|(x_{n}')^{-1}\bigg|^{\epsilon_n}\leq(\alpha-\eta)^{-1}.
\end{align*}
Thus $\alpha-\eta\leq |(x_{n}')|{y}=|(x_{n})|_{y}\leq \alpha+\eta.$ We conclude by letting $\eta\rightarrow 0$.

Since both spaces are compact and Hausdorff, the inverse map is also continuous, which concludes the proof.
\end{proof}

\subsubsection{Complete residue fields}
\label{sec:complete residue fields}
Let $\omega$ be an ultrafilter, which we identify with its image in $\cM(A^\epsilon)$. Recall that the complete residue field $\sH(\omega)$ is obtained as the completion 
of the quotient ring $A^{\epsilon}/\ker(\omega)$ where
\[ \ker(\omega)=\left\{(x_{n})\in \C^{\N}, \lim_{\omega}|x_n|^{\epsilon_n}=0\right\}.
\]
Observe that $\ker(\omega)$ is a maximal ideal of $A^{\epsilon}$. Indeed, pick $y\notin \ker(\omega)$. Then $\lim_{\omega}|y_n|^{\epsilon_n}>0$. The set of integers $E$ such that $|y_n|^{\epsilon_n} <2^{-1}\min\left\{|y|_\omega,1\right\}$ is $\omega$-thin, hence $x_E\in \ker(\omega)$, and $y+x_E$ is invertible by Remark~\ref{rmkinverse}.
It follows that $A^\epsilon/\ker(\omega)$ is a field. Since 
$\omega$ is bounded, $\ker(\omega)$ is closed hence 
$A^\epsilon/\ker(\omega)$ is complete for the residue norm, and it follows that $\sH(\omega)=A^\epsilon/\ker(\omega)$. 
\begin{prop}
\label{propresidue}
For any $\omega\in\beta\N$, the complete residue fields $\sH(\omega)$ is
 algebraically closed and spherically complete, and $|\sH(\omega)|=\R_{+}$.
\begin{enumerate}
     \item If $e=\lim_{\omega}\epsilon_{n}>0$, then $\sH(\omega)$ is isometric to $(\C,|\cdot|^{e})$ hence is an Archimedean metrized field.
    \item If $\lim_{\omega}\epsilon_{n}=0$, then $\sH(\omega)$ is a non-Archimedean metrized field. 
\end{enumerate}
\end{prop}
 
\begin{rmk}
\label{rmkrobinson}
When $\lim_{\omega}\epsilon_{n}=0$, $\sH(\omega)$ is a complexified Robinson field. We refer to~\cite{lightstone2016nonArchimedean} for a  detailed treatment of such fields. 
\end{rmk}

\begin{proof}
Suppose $e= \lim_{\omega}\epsilon_{n}>0$.
Pick $x\in A^\epsilon$. Then
$|x|_\omega=\lim_{\omega}|x_n|^{\epsilon_{n}}$, hence $|x_n|$ is bounded from above. It follows that $x_\omega =  \lim_\omega x_n$ exists in $\C$. The map $x\mapsto x_\omega$
is a morphism $A^\epsilon \to \C$ which factors through $\ker(\omega)$ hence induces
an isometric field morphism $\sH(\omega)= A^\epsilon/\ker(\omega)\to (\C, |\cdot|^{e})$.

Suppose now that $\lim_{\omega}\epsilon_{n}=0$.
Then $|(2)|_{\omega}=\lim_{\omega}|2|^{\epsilon_n}=1$ so that $\sH(\omega)$ is a non-Archimedean field. 
Observe that
$\ker(\omega)= \{(a_n)\in A^\epsilon, \lim_\omega |a_n|^{\epsilon_n}=0\}$. 
We already observed that  $A^\epsilon/\ker(\omega)$ is a field. This implies that the image of $A^\epsilon$ is equal to $\sH(\omega)$. 

\smallskip

We now prove that $\sH(\omega)$ is spherically complete. We may suppose that $\lim_{\omega}\epsilon_{n}=0$.
Let $\bar{B}_{i}=  \bar{B}(\alpha_i,r_i)$ be any decreasing sequence of closed balls, 
with $r_i>0$ and $\alpha_i=(\alpha_{i,n})_n\in A^\epsilon$
so that $|\alpha_{i}-\alpha_{j}|_{\omega}<r_{i}$ for all $j\geq i$.

We define a strictly decreasing sequence of $\omega$-big sets $\N=N_{0}\supseteq N_{k}\supseteq N_{k+1}$ such that for any $i\leq j\leq k$ and $l\in N_{k}$, we have $|\alpha_{i,l}-\alpha_{j,l}|^{\epsilon_l}\le r_{i}.$

Suppose $N_0, \cdots, N_k$ have been constructed. Since for any $i\leq k+1$, 
we have
$|\alpha_{i}-\alpha_{k+1}|_{\omega}<r_{i}$, one can find an $\omega$-big set $N$ such that for all $l\in N$, we have
\[|\alpha_{i,l}-\alpha_{k+1,l}|^{\epsilon_l}<r_{i}.\] We set $N_{k+1}=  (N\cap N_k)\setminus\{ \min N_k\}$.

Set $b_{n}=  \alpha_{j_n,n}$
with $j_n=  \min N_n$, and let $\beta=(b_n)$. 
Choose any integer $i$. For any $n\ge i$, we have
\[|\alpha_{i,n}-b_{n}|^{\epsilon_n}=|\alpha_{i,n}-\alpha_{j_n,n}|^{\epsilon_n}\le r_i.\]
It follows that $|\alpha_{i}-\beta|\leq r_{i}$, hence $\beta\in \bar{B}_{i}$, thus $\bigcap_{i=0}^{\infty}\bar{B}_{i}$ is not empty.

Finally we prove that $\sH(\omega)$ is algebraically closed. Again, it is sufficient to treat the case $\lim_\omega \epsilon_n=0$, so that $\sH(\omega)$ is non-Archimedean.
Pick any polynomial $P_{\omega}\in\sH(\omega)[z]$. Then we can find a sequence of polynomials
\[P_{n}(z)=a_{0,n}z^{d}+a_{1,n}z^{d-1}+...+a_{d,n} \in \C[z] \]
such that $a_i=  (a_{i,n})\in A^{\epsilon}$, 
$a_0 \in (A^\epsilon)^\times$,
and the image of $P=  \sum a_i z^i$ in $\sH(\omega)$ is equal to $P_\omega$. 
We may factorize each polynomial
$P_n(z) = a_{0,n} \prod_{i=1}^d (z-\alpha_{i,n})$. Since $a_0$ is invertible in $A^\epsilon$, the sequence $\alpha_i =(\alpha_{i,n})$ belongs to $A^\epsilon$
since \begin{align*}
    |\alpha_{i,n}| \le \max\left\{ \left(2d|a_{i,n}|/|a_{0,n}|\right)^{1/i}, \,
 i=1, \cdots, d\right\}.
\end{align*}
It is moreover a zero of $P$, and its image in $\sH(\omega)$ is a zero of $P_\omega$ which concludes the proof. 
\end{proof}

The last argument relates the sequence of zeroes of $P_n$ to the zeroes of $P_\omega$. The following converse holds.

\begin{prop}
\label{propalgc}
Pick any polynomial  
$P(z)\in A^{\epsilon}[z]$ of degree $d\ge 1$, determined by a sequence $P_n\in \C[z]$.

For any $\omega \in\beta\N$, and for any zero $\alpha\in\sH(\omega)$ of multiplicity $m$ of $P_{\omega}$, there exist sequences $\alpha_{i,n}\in\C$, $1\leq i\leq m$ such that 
$\Pi_{i=1}^{m}(z-\alpha_{i,n})$ is a factor of $P_{n}$ for each $n$, and $(\alpha_{i,n})=\alpha$ in $\sH(\omega)$ for each $i$.
\end{prop}
\begin{proof}
Without loss of generality, we may suppose that $\alpha=0$, so that 
$P_\omega(z)=z^m Q_\omega(z)$ with $Q_\omega \in\sH(\omega)[z]$, and $Q_\omega(0)\neq0$.
Pick any polynomial $Q\in A^\epsilon[z]$  projecting to $Q_\omega$. We can then write
$P(z)=z^{m}Q(z)+H(z),$
where $H$ is determined by a sequence of complex polynomials 
$H_n(z)= \sum_{i=0}^d h_{i,n} z^i$
such that $(h_{i,n})\in A^\epsilon$ for each $i$, and $\lim_\omega |h_{i,n}|^{\epsilon_n}=0$. 

Also,
$Q$ is determined by a sequence of polynomials $Q_n(z)=\sum_{i=0}^d q_{i,n} z^i\in\C[z]$, such that 
$|q_{i,n}|\le M^{1/\epsilon_n}$ for some $M>1$, and all $i, n$, and 
$|q_{0,n}|^{\epsilon_n}\ge 2\rho$ for some $\rho>0$ on a $\omega$-big set $E$, since $Q_\omega(0)\neq0$.

Pick any $r\in (0,1)$ such that 
$Mdr < \rho$, 
and $\eta>0$ so that $\eta < r^m \times \rho$. 
On the circle 
$|z|= r^{1/\epsilon_n}$, we have
$|H_n(z)| \le \eta^{1/\epsilon_n}$
for all $n$ in an $\omega$-big set $E'\subset E$. 
On the other hand, 
for all $|z|= r^{1/\epsilon_n}$, we have $|Q_n(z)| \ge \rho^{1/\epsilon_n}$
for all $n\in E'$.
We get $|z^{m}Q_{n}(z)| \geq 
\rho^{1/\epsilon_n} r^{m/\epsilon_n}
>|H(z)|$ on $|z|=r^{1/\epsilon_n}$
for all $n\in E'$. 

By Rouché's theorem, for all $n\in E'$ the polynomial $P_{n}(z)=z^{m}Q_{n}(z)+H_{n}(z)$ admits exactly $m$ zeros $\alpha_{1,n}, \cdots, \alpha_{m,n}$ in the disk $|z|\leq r^{1/\epsilon_n}$. 

By letting $r\to0$, we obtain $\lim_\omega |\alpha_{j,n}|^{\epsilon_n}=0$ for all $j$ which concludes the proof.
\end{proof}

\subsubsection{The projective line over $A^\epsilon$.}
\label{sec:projective line over A}
The analytification of the projective line over $A^\epsilon$
admits a canonical continuous proper map 
$\pi\colon \P^{1,\mathrm{an}}_{A^\epsilon} \to \cM(A^\epsilon) \simeq \beta\N$. 
When $\sH(\omega)=\C_e$ is Archimedean, then there is a canonical isomorphism 
$s_e\colon \pi^{-1}(\omega)\to \hat{\C}$ (see \S\ref{sec:potential theory}). When $\sH(\omega)$ is non-Archimedean, then $\pi^{-1}(\omega)$ contains only Type-1 and Type-2 points by Proposition~\ref{propresidue}.

Pick any sequence of point $p_n\in \hat{\C}$. 
Then we may write in homogeneous coordinates $p_n=[z_{0,n}\colon z_{1,n}]$
with $\max\{|z_{0,n}|,|z_{1,n}|\}=1$, so that $z_i=(z_{i,n})\in A^\epsilon$. 
Note that  $z_0 A^\epsilon + z_1 A^\epsilon = A^\epsilon$ so that $p=[z_0\colon z_1]$ defines a point in  $\P^1(A^\epsilon)$, and observe it does not depend on the choice of $z_{i,n}$.

Conversely pick any point $p=[z_0\colon z_1]\in \P^1(A^\epsilon)$. By construction, we have $z_0 A^\epsilon + z_1 A^\epsilon = A^\epsilon$. 
The pair $(z_0,z_1)$ is unique up to multiplication by a unit in $(A^\epsilon)^\times$. Write $z_i=(z_{i,n})$ and set $\lambda = (\max\{|z_{0,n}|,|z_{1,n}|\})$. Observe that the condition $z_0 A^\epsilon + z_1 A^\epsilon = A^\epsilon$ implies $\lambda \in (A^\epsilon)^\times$. Replacing $z_i$ by $\lambda^{-1} z_i$
we may suppose that $\max\{|z_{0,n}|,|z_{1,n}|\}=1$. 
We have proved
\begin{lemma}\label{lem:P1Ae}
  The map defined above $\alpha \colon \hat{\C}^\N \to \P^1(A^\epsilon)$ is a bijection. 
\end{lemma}

Pick any $\omega\in \beta\N$, and suppose that $\sH(\omega)$ is non-Archimedean. Let $\tilde{\sH}(\omega)$ be the residue field of the complete valued field $\sH(\omega)$. We claim that there is a canonical map 
\begin{equation}\label{eq:theta-om}
    \theta_\omega \colon \P^1(\tilde{\sH}(\omega))\to \hat{\C}
\end{equation}
defined as follows. 

Pick a point $\tilde{p}=[\tilde{z}_0\colon\tilde{z}_1]\in \P^1(\tilde{\sH}(\omega))$; 
choose some lifts $z_0,z_1\in \sH(\omega)$ with  $\max\{|z_{0}|_\omega,|z_{1}|_\omega\}=1$. 
Since $A^\epsilon/\ker(\omega)$ is dense in $\sH(\omega)$ we may suppose that
$z_0,z_1\in A^\epsilon$. Replacing by $z_{0,n},z_{1,n}$ by $1$ on an $\omega$-thin set, we may suppose that $z_0A^\epsilon + z_1 A^\epsilon= A^\epsilon$. 
We let $\theta_\omega(\tilde{p})=  \lim_\omega [z_{0,n}\colon z_{1,n}]\in \hat{\C}$.

This map is well-defined because if $[z'_0\colon z'_1]\in \P^1(A^\epsilon)$ 
defines the same point $\tilde{p}$, then in the spherical distance  (see \S\ref{sec:Berkovich projective line over field}), we have 
\[\lim_\omega d_{\P^1\left(\sH(\omega)\right)}([z_0\colon z_1],[z'_0\colon z'_1])<1\] which implies
$\lim_\omega [z_{0,n}\colon z_{1,n}]= \lim_\omega [z'_{0,n}\colon z'_{1,n}]$ in $\hat{\C}$.

Observe that $\theta_\omega$ is surjective, and the following diagram is commutative.
\begin{figure}[h]
\centering
\begin{tikzcd}
\hat{\C}^{\N} \arrow[r, "\alpha"] \arrow[rrrd, "\lim_{\omega}", bend right=10] & \P^{1}(A^\epsilon) \arrow[r, "r_\omega"] & \P^{1}(\sH(\omega)) \arrow[r] & \P^{1}(\tilde{\sH}(\omega)) \arrow[d, "\theta_\omega"] \\
 &                            
 &                               
 & \P^{1}(\C)                           
\end{tikzcd}
\caption{Relating projective lines}
\label{diagram:residue}
\end{figure}

Finally, we discuss the spherical distance on $\P^1(A^\epsilon)$, see~\eqref{eq:proj-dist-arch} and~\eqref{eq:proj-dist-nonarch} for the definition. We use the above lemma to identify points in $\P^1(A^\epsilon)$
and sequences of points in $\hat{\C}$.

\begin{lemma}\label{lem:limit-proj-dist}
For any points $x,y\in \P^1(A^\epsilon)$, and for any $\omega\in \beta\N$, we have
 \[
\lim_\omega d_{\P^1(\C)}(x_n,y_n)^{\epsilon_n}
=
\lim_\omega d_{\P^1(\C_{\epsilon_n})}(s_{\epsilon_n}^{-1}x_n,s_{\epsilon_n}^{-1}y_n)
=
d_{\P^1(\sH(\omega))}(x_\omega,y_\omega)~.
  \]

\end{lemma}
\begin{proof}
 Write $x=(x_n)\in\hat{\C}^\N$, and $y=(y_n)\in\hat{\C}^\N$. We may suppose
 that neither $x$ nor $y$ are equal to $[1\colon0]$, and up to replace $x_n$ and $y_n$ in an $\omega$-thin set, we have $x_n=[z_n\colon1]$, and $y_n=[w_n\colon1]$  with $(z_n),(w_n)\in A^\epsilon$. 
In the standard euclidean norm $|\cdot|$ on $\C$, we have 
\begin{align*}
d_{\P^1(\C)}(x_n,y_n)^{\epsilon_n}
= 
d_{\P^1(\C_{\epsilon_n})}(s_{\epsilon_n}^{-1}x_n,s_{\epsilon_n}^{-1}y_n)
=
\frac{|z_n-w_n|^{\epsilon_n}}{(|z_n|^{2}+1)^{\epsilon_n/2}(|w_n|^{2}+1)^{\epsilon_n/2}}~.
\end{align*}
When $\lim_\omega \epsilon_n>0$, the result is clear.
When $\lim_\omega \epsilon_n=0$, since $\lim_\omega |z_n|^{\epsilon_n} = |z_\omega|$, we have 

\[
\lim_\omega
(|z_n|^{2}+1)^{\epsilon_n/2}
=\max\{|z_\omega|,1\}
\]

and the result follows.
\end{proof}
%

\subsubsection{Product Banach rings with different characteristic exponents.}

We discuss the relationship between 
$A^\epsilon$ and $A^{\epsilon'}$ and their Berkovich spectra
for two different sequences $\epsilon, \epsilon'$.
To simplify notation, given any two sequences of positive real numbers $\epsilon, \epsilon'$ we write
$\epsilon \lesssim \epsilon'$ iff
$\epsilon_{n}\leq C\epsilon_{n}'$ for some $C>0$; 
and $\epsilon \asymp \epsilon'$ if and only if 
$\epsilon \lesssim \epsilon'$ and $\epsilon' \lesssim \epsilon$.
\begin{prop}
\label{propequivalence}
Let  $\epsilon'=(\epsilon_n')$ be another sequence with $0<\epsilon_n'\leq 1$. Then, $A^{\epsilon}\subseteq A^{\epsilon'}$ if and only if $\epsilon' \lesssim \epsilon$.

In particular, $A^{\epsilon}=A^{\epsilon'}$ if and only if 
$\epsilon \asymp \epsilon'$.
\end{prop}
\begin{rmk}
In general, we may have $A^{\epsilon}=A^{\epsilon'}$ but the norm on $A^{\epsilon}$ is not necessarily equivalent to that of $A^{\epsilon'}$.
\end{rmk}

\begin{proof}
Suppose $\epsilon_{n}'\leq C\epsilon_{n}$, and pick $x=(x_{n})\in A^{\epsilon}$.
Then
\[
|x_{n}|^{\epsilon_n'}\leq
\max\{ 1, |x_{n}|\}^{\epsilon_n'}
\leq 
\max\{ 1, |x_{n}|\}^{C\epsilon_n}
\le \max\{1, |x|_\epsilon\}^C <\infty~,
\]
so that  $x\in A^{\epsilon'}$.

    Conversely, assume $A^{\epsilon}\subseteq A^{\epsilon'}$ and write $\epsilon_{n}'=\alpha_{n}\epsilon_{n}$ for some $\alpha_n>0$.  The sequence $x_n=2^{\frac{1}{\epsilon_n}}$, belongs to $A^{\epsilon}\subseteq A^{\epsilon'}$. Thus there exists $M>0$ such that 
\[\sup_n |x_{n}|^{\epsilon_n'}=\sup_{n}2^{\alpha_n}<M.\]
Therefore, we have $0<\alpha_{n}<\frac{\log M}{\log 2}<\infty$, which proves $\epsilon'\lesssim \epsilon$.
\end{proof}

Pick $\epsilon, \epsilon'\in (0,1]$, and write $\epsilon' = \alpha \epsilon$.  Recall that  $\C_\epsilon=(\C,|\cdot|^\epsilon)$.
Recall from \S\ref{sec:potential theory}, 
that we have a canonical homeomorphism
$\s_\alpha \colon \P_{\C_\epsilon}^{1,\mathrm{an}} \to \P_{\C_{\epsilon'}}^{1,\mathrm{an}}$ sending a semi-norm $|\cdot|$ on $\C_\epsilon[z]$
to $|\cdot|^{\alpha}$. This map
induces field continuous isomorphisms $\s_\alpha^\#\colon \sH(\s_\alpha(x)) \to \sH(x)$ for all $x\in \P_{\C_\epsilon}^{1,\mathrm{an}}$ (which are isometric if and only if $\alpha=1$).

 \begin{prop}
    \label{propisoepsilon}
    Suppose $\epsilon',\epsilon\in (0,1]^\N$ satisfy $\epsilon'\asymp \epsilon$, and write $\alpha_n=  \epsilon'_n/\epsilon_n$. The locally ringed space isomorphisms
    $\s_{\alpha_n} \colon  \P_{\C_{\epsilon_n}}^{1,\mathrm{an}}\to  \P_{\C_{\epsilon'_n}}^{1,\mathrm{an}}$
 extend to an isomorphism of locally ringed spaces $\s\colon \P_{A^\epsilon}^{1,\mathrm{an}}\to \P_{A^{\epsilon'}}^{1,\mathrm{an}}$ over $\beta\N$ .
 \end{prop}

\begin{proof}
Since $\alpha_n$ is bounded from above, 
it admits a unique extension to $\beta\N$
by setting $\alpha_\omega =  \lim_\omega \alpha_n$. Since $\alpha_n$ is bounded away from zero, we have $\alpha_\omega \in (0,1]$ for all $\omega\in\beta\N$. 

For any $\omega \in \beta\N$, set $
\epsilon_\omega =  \lim_\omega \epsilon_n$, and 
$\epsilon'_\omega =  \lim_\omega \epsilon'_n$, and define
\[\s_\alpha \colon \P_{\sH(\epsilon_\omega)}^{1,\mathrm{an}} \to \P_{\sH(\epsilon'_\omega)}^{1,\mathrm{an}}\] by sending the semi-norm $|\cdot|$ on $\sH(\epsilon_\omega)[z]$ to $|\cdot|^{\alpha_\omega}$. This is a homeomorphism.
Recall that we denote by $\pi\colon \P^{1,\mathrm{an}}_{A^\epsilon} \to \cM(A^\epsilon)\simeq \beta\N$ the canonical projection. Then for any polynomial 
$P\in A^\epsilon[z]$ and for any $x\in\P^{1,\mathrm{an}}_{A^\epsilon}$, we have 
\begin{equation}\label{eq:poly}
    |P(\s(x))|=|P(x)|^{\alpha_{\pi(x)}}
\end{equation}
Since $x\mapsto \alpha_{\pi(x)}$ is continuous, it follows that the map 
$\s\colon \P_{A^\epsilon}^{1,\mathrm{an}}\to \P_{A^{\epsilon'}}^{1,\mathrm{an}}$ given by
$\s(x)=  \s_{\alpha_{\pi(x)}}(x)$ is a homeomorphism.

Since analytic functions on an open subset $U\subset \P_{A^{\epsilon'}}^{1,\mathrm{an}}$ are obtained as locally uniform limits
of rational functions having poles outside $U$, the equality~\eqref{eq:poly}
induces isomorphisms
$\s^\# \colon \cO(U) \to \cO(\s^{-1}(U))$, and
continuous field isomorphisms
$\s^\# \colon \sH(\s(x)) \to \sH(x)$
for any $x\in \P_{A^\epsilon}^{1,\mathrm{an}}$. 
This concludes the proof.
\end{proof}

 \subsection{Sequential hybridation}
 \label{sec:sequentialhybridation}
Recall from \S\ref{sec:productbanachring} that we have a canonical identification $\cM(A^\epsilon) \simeq \beta\N$.

If $B$ is any ring, a rational map $f\colon\P^1_B\to \P^1_B$ defined over $B$ is given in homogeneous coordinates by a pair of homogeneous polynomials $P,Q$ of degree $d$
whose resultant is a unit in $B$. When $B$ is a Banach ring, 
then $f$ induces a continuous map $f\colon\P^{1,\mathrm{an}}_B\to \P^{1,\mathrm{an}}_B$ over $\cM(B)$, and for any $b\in\cM(B)$, we denote by $f_b$
the restriction of $f$ on the fiber $\pi^{-1}(b)$
where $\pi\colon\P^{1,\mathrm{an}}_B\to\cM(B)$ is the canonical morphism. It is a rational map of degree $d$ defined over $\sH(b)$.

Recall that $\rat_d$ is an affine variety. We fix $\overline{\rat}_d$ any projective variety containing $\rat_d$ as a Zariski open dense subset (for instance
the one described by Silverman~\cite{SJ98} using geometric invariant theory).

Let $C_d=e\sup_{\rat_d(\C)} |\res|$.  

\begin{thm}
\label{theorem:sequentialhybridation}
Let $\mf_n$ be any sequence in $\rat_d(\C)$. Set 
\[\epsilon_n = (-\log(|\res(\mf_n)|/C_d))^{-1}~,\] 
and choose  $f_n\in \Rat_d(\C)$ such that $[f_n]=\mf_n$ and $|\Res(f_n)|= |\res(\mf_n)|$ for all $n$. 
Let $A^\epsilon$ be the Banach ring defined by~\eqref{eq:def-aep}.
\begin{enumerate}
\item
There exists a rational map $f\in\Rat_d(A^\epsilon)$ which 
induces a continuous self-map $f\colon\P^{1,\mathrm{an}}_{A^\epsilon}\to\P^{1,\mathrm{an}}_{A^\epsilon}$
such that $\pi\circ f= \pi$, and for any $n\in \N$, we have $f|_{\pi^{-1}(n)} =f_n$
under the above identification $\pi^{-1}(n)\simeq \hat{\C}$. 
 \item 
 If $\sH(\omega)$ is non-Archimedean, then $f_\omega$ does not have potential good reduction.
 \item  
The field $\sH(\omega)$ is non-Archimedean if and only if $\lim_\omega \mf_n\in \overline{\rat}_d^{an}\setminus\rat_d^{an}$.
 \end{enumerate}
\end{thm}

\begin{rmk}
This theorem implies a slightly more general version of Theorem~\ref{thm: sequential hyridation construction} (2). Indeed we assumed in the introduction 
that the sequence $\mf_n$ is degenerating, in which case $\sH(\omega)$ is non-Archimedean if and only if $\omega\in\beta\N\setminus \N$.
Observe that Theorem~\ref{thm: sequential hyridation construction} (1) is a consequence of the discussion of \S\ref{sec:projective line over A}.
\end{rmk}

\begin{rmk}
\label{rmk: equivsequential}
A priori, the sequence $f_n$ is not uniquely defined. However it follows from Lemma~\ref{lem:proper-GIT} and Proposition~\ref{propequivalence}, that for any sequence $g_n\in\Rat_d(\C)$ such that 
$|\Res(f_n)|\asymp|\Res(g_n)|$, then we get a rational map $g\in \Rat_d(A^{\epsilon'})$
and a homeomorphism $\P^1_{A^\epsilon} \to \P^1_{A^{\epsilon'}}$ conjugating $f$ to $g$. 
We refer to any sequence $\epsilon\in(0,1]^\N$ and any rational map $f\in\Rat_d(A^\epsilon)$ satisfying
Properties (1) -- (3) of the theorem to a \emph{sequential hybridation} of $\mf_{n}$.
\end{rmk}

\begin{rmk}
    By~\cite[Theorem E]{FR10}, the topological entropy $\htop(f)$ of a rational map $f$ defined over a complete valued field $(k,|\cdot|)$ of degree $d\ge2$ is equal to $0$ if and only if the norm on $k$ is non-Archimedean  and $f$ has potential good reduction. In other words, Property (2) of the previous theorem can be alternatively formulated by saying that $\htop(f_\omega)>0$ for any $\omega\in\beta\N$.
\end{rmk}

\begin{proof}
For each $n$, pick any rational map $f_n\in \Rat_d(\C)$
representing the class $\mf_n$ such that
$|\res(\mf_n)|=|\Res(f_n)|$, see \S\ref{sec:minimalresultant}. 
Observe that by construction \[\epsilon_n=(-\log(|\res(\mf_n)|/C_d))^{-1}\in(0,1].\]
In homogeneous coordinates, we may write $f_{n}=[P_{n}\colon Q_{n}]$, 
and normalize the coefficients of $P_n$ and $Q_n$ so that the maximum of their moduli is equal to $1$.  It follows that 
by construction we have 
\[
|\Res(P_n,Q_n)|^{\epsilon_n}
=|\Res(f_n)|^{\epsilon_n}
=\exp\left(-\log\frac{|\Res(f_n)|}{C_d}\right)
\in [e^{-1},C_d]~.\]
It follows that 
$P=(P_n)$ and $Q=(Q_n)$ are homogeneous polynomials with coefficients in $A^{\epsilon}$
and $\Res(P,Q)$ is a unit.
Consequently, 
$f=(f_n)$ defines an endomorphism of degree $d$ of $\P^1_{A^\epsilon}$.
By construction, (1) holds. 

Suppose that $\sH(\omega)$ is non-Archimedean. By Proposition~\ref{propresidue}, $\omega$ is a non-principal ultrafilter and $\lim_\omega \epsilon_n=0$. It follows that 
$|\Res(P_n,Q_n)|\to 0$. Let $M_n\in\PGL_{2}(\C)$ be a sequence of Möbius transformations such that $M_\omega\in\PGL_{2}(\sH(\omega))$ and $|\res(f_\omega)|=|\Res(M_\omega\cdot f_\omega)|.$
Since 
\[|\res(f_\omega)|
=\lim_\omega |\Res(M_n\circ f_n)|^{\epsilon_n}\leq
\lim_\omega |\Res(P_n,Q_n)|^{\epsilon_n}=
e^{-1} <1,\]
the rational map
$f_\omega$ does not have potential good reduction by Theorem~\ref{theorem:min-res-Rumely}.

Let us prove (3). Suppose that $\sH(\omega)$ is non-Archimedean. We have already seen that $\omega$ is non-principal and $\lim_\omega \epsilon_n=0$. Suppose by contradiction that $\lim_\omega [f_n]$ is the conjugacy class of a rational map $g\in\Rat_d(\C)$. By Proposition~\ref{propminres}, 
$\lim_\omega|\res(f_n)|= |\res(g)|$ which implies 
$\lim_\omega \epsilon_n>0$, impossible.
 
Conversely suppose $\sH(\omega)$
is Archimedean. By Proposition~\ref{propresidue}, we have $\lim_\omega \epsilon_n>0$ hence
$- \log |\res(\mf_n)|$ is bounded from above  on a $\omega$-big
set. Since $- \log |\res|$
is proper on $\rat_d(\C)$ by Proposition~\ref{propminres}, we conclude
that $\lim_\omega \mf_n\in \rat_d^{an}$. 
\end{proof}

An application of the previous techniques is a new proof of a result by DeMarco, see~\cite[Corollary~0.3]{zbMATH05004325demarcoiteration}.

\begin{cor}
\label{cor: iteration}
 The iteration map $I_l\colon\rat_d(\C)\to\rat_{d^l}(\C)$
 given by $I_l(f)=  f^l$ is proper for any $d\ge2$ and any $l\in\N$. 
\end{cor}
\begin{rmk}
The iteration map is also proper 
as a scheme morphism. It follows that 
$I_l\colon \rat_d(k)\to\rat_{d^l}(k)$ is
proper for any field of characteristic $0$. Our techniques being characteristic free, it is likely that 
$I_l$ is proper for any field $k$.
\end{rmk}
\begin{rmk}
Our proof is based on the characterization of properness on locally compact metric spaces in terms of sequences. 
This is not essential, and we could have used the hybridation of families of holomorphic maps. 
\end{rmk}
 
\begin{proof}
Suppose by contradiction that $I_l$ is not proper. Then we can find a degenerating sequence $\mf_n\in\rat_d$ such that $\mf_n^l\to \mathfrak{g}\in\rat_{d^l}(\C)$.

Set $\epsilon_n=-(\log|\res(\mf_n)|/C_d)^{-1}$, and pick any endomorphism $f\in\Rat_d(A^\epsilon)$
satisfying the conclusion of Theorem~\ref{theorem:sequentialhybridation}. Since $\mf_n$ is degenerating, $\epsilon_n \to0$, and
we have $[f_n]= \mf_n$.

Choose any $g\in\Rat_{d^l}(\C)$ such that $[g]=\mathfrak{g}$. By assumption, we can find $g_n\to g_\infty\in \Rat_{d^l}(\C)$, and $M_n\in\PGL_2(\C)$ such that 
$M_n\cdot g_n= f_n^l$ for all $n$. 

Pick any non principal ultrafilter $\omega$. Since $\epsilon_n\to0$, $\sH(\omega)$ is a non-Archime\-dean metrized field, and $f_\omega\in\Rat_d(\sH(\omega))$ does not have potential good reduction by Theorem~\ref{theorem:sequentialhybridation} (2). By~\cite[Corollary~8.14]{Ben19}, it follows that $f_\omega^l$ does not have potential good reduction either. 

Pick any normalized representations  
\[M_n=
\left(\begin{array}{cc}
 a_n    & b_n \\
  c_n   & d_n
\end{array}\right)
\text{ such that }
\max\{|a_n|, |b_n|, |c_n|, |d_n|\} =1.\] 
\begin{lemma}\label{lem:GIT}
We have $\epsilon_n \asymp - \log |\det(M_n)|^{-1}$.     
\end{lemma}
It follows in particular that $M= (M_n)$ is an element
of $\PGL_2(A^\epsilon)$. 

We shall now prove
that $g_\omega = M_\omega^{-1}\cdot f_\omega^l$ has good reduction, 
which gives a contradiction.

Choose normalized representations $g_n = [P_n\colon Q_n]$, and observe that the pair of homogeneous polynomials $P_\omega$ (resp. $Q_\omega$) determined by the sequence $P_n$ (resp. $Q_n$) forms a normalized representation 
 $g_\omega = [P_\omega\colon Q_\omega]$. 
 Reducing $P_\omega$ and $Q_\omega$ in 
 $\tilde{\sH}(\omega)$, we obtain a rational map 
$\tilde{g}_\omega$ on $\P^1_{\tilde{\sH}(\omega)}$ of degree $\le d^l$, and we have
equality if and only if $g_\omega$ has good reduction, see~\cite[\S~2.5]{Sil07}). 

Recall from~\eqref{eq:theta-om} that we have a canonical map $\theta_\omega \colon \tilde{\sH}(\omega)\to \C$, and the image of $P_\omega$ and $Q_\omega$
under this map gives a rational map $g_\C$ on $\P^1(\C)$
of degree $\le \deg(\tilde{g}_\omega)$. 
It follows by construction that $g_\C=g_\infty$
so that $\deg(\tilde{g}_\omega)=d^l$ hence 
$g_\omega$ has good reduction. 

The figure below explains how all maps $g_n, g, g_\omega, \tilde{g}_\omega, g_\infty$ fit into one
commutative diagram. 
\begin{figure}[h]
\centering
\begin{tikzcd}
\P^{1}(\C)^{\N} \arrow[loop left, "g_n"] \arrow[r, "\alpha"] \arrow[rrrd, "\lim_{\omega}", bend right=10] & \P^{1}(A^\epsilon) \arrow[loop, "g"]\arrow[r, "r_\omega"] & \P^{1}(\sH(\omega)) \arrow[loop, "g_\omega"]\arrow[r] & \P^{1}(\tilde{\sH}(\omega)) \arrow[loop right, "\tilde{g}_\omega"] \arrow[d, "\theta_\omega"] \\
 &                            
 &                               
 & \P^{1}(\C)   \arrow[loop right, "g_\infty"]      
\end{tikzcd}
\caption{Infering good reduction}
\label{diagram:good reduction}
\end{figure}
\end{proof}
\begin{proof}[Proof of Lemma~\ref{lem:GIT}]
We have a sequence $g_n\in \Rat_{d^l}(\C)$ converging to $g_\infty\in\Rat_d(\C)$ so that $|\Res(g_n)|$ is bounded away from $0$ and $\infty$, and $M_n\in\PGL_2(\C)$. 
It follows from \ref{lem:proper-GIT} (applied to $M=M_n$ and $f=g_n$), that  $-\log|\det(M_n)| \asymp - \log|\Res(M_n\cdot g_n)|$ which implies the result. 
\end{proof}

We end this section by giving an alternative proof of a result by Kiwi \cite[Corollary 6.2]{Ki15}. 
\begin{cor}\label{cor:kiwi}
Let $f_n\in \Rat_{d}(\C)$ be any sequence converging in $\Rat_{d}(\C)$ to some degree $d$ rational map $f$ . 
Suppose that we can find $q\ge 1$, a finite set $\cH$, and a sequence $M_{n}\in\mathrm{PGL}_{2}(\C)$ such that
$h_n:=M_n\circ f^q_n\circ M_n^{-1}$ is converging locally on compact subsets on $\P^1_\C\setminus \cH$
to a rational map $h$ of degree at least $2$. 

Then $h_n$ does not degenerate, and  $h$ is conjugated to $f^q$.
\end{cor}
\begin{proof}
Suppose by contradiction that  $\epsilon_n:=\left(-\log\Res|h_n|\right)^{-1}$ admits a subsequence tending to $0$. Choose any $\omega\in\beta\N$
such that $\lim_\omega \epsilon_n=0$.

Since $f_n$ is not degenerating, observe that it defines a rational map on $\P_{A^\epsilon}^{1,\mathrm{an}}$, and that 
 $f_{\omega}$ has good reduction, i.e., leaves the Gauss point totally invariant. 
 On the other hand Lemma \ref{lem:proper-GIT}  implies that the sequence $M_n$ defines an element of $\PGL_2(A^\epsilon)$, 
 so that $h_\omega = M_\omega \circ f^q_\omega \circ M_\omega^{-1}$, and $h_\omega$
 has potential good reduction.

 By assumption, $h_n$ converges to a map of degree  $\delta\ge 2$, hence as in the previous proof, we infer that
 $x_g$ is fixed by $h_\omega$, and the local degree at $x_g$ is at least $\delta$. This implies $h_\omega$
 to have good reduction, contradicting  $|\Res(h_{\omega})|=\lim_{\omega}|\Res(h_n)|^{\epsilon_n}=e^{-1}<1$.
\end{proof}

\section{Family of measures over a Banach ring}
\label{sec:family of measures over a Banach ring}
Let $(B,\lV\cdot\rV)$ be any Banach ring. 
In this short section, we discuss the general notion of families of probability measures on $\P^{1,\mathrm{an}}_B$.

\subsection{Continuous family of Radon measures}
\label{sec:continuousfamilyofmeasures}
Recall that $\P_{B}^{1,\mathrm{an}}$ is a compact space, and we have a canonical projection map $\pi\colon\P_{B}^{1,\mathrm{an}}\rightarrow \cM(B)$, such that for any $b\in \cM(B)$, $\pi^{-1}(b)$ is isomorphic to $\P_{\sH(b)}^{1,\mathrm{an}}$ in a canonical way.

A positive Radon measure $\mu$ on a compact space $K$ is a positive linear functional
$\mu\colon \cC^{0}(K) \to \R$, i.e., a linear map such that 
$\mu(\f)\ge0$ when $\f\ge0$. A probability measure is a positive Radon measure with $\mu(K)=1$.
We endow the set of positive (resp. probability) measures $\mathrm{M}^+(K)$ (resp. $\mathrm{M}^1(K)$) on $K$ with the weakest topology such that the function $\mu\rightarrow \int \f d\mu$ is
continuous for any $\f\in\cC^0(K)$.

A linear functional
\[L\colon\cC^{0}(\P_{B}^{1,\mathrm{an}})\rightarrow\cC^{0}(\cM(B))\]is said to be positive if it maps positive functions to positive functions. 
\begin{defi}\label{def:continuous-family}
A continuous family of positive measures on 
$\P_{B}^{1,\mathrm{an}}$ is a positive linear functional  $L\colon \cC^{0}(\P_{B}^{1,\mathrm{an}})\to\cC^{0}(\cM(B))$ such that for all $\f\in \cC^0(\cM(B))$ we have $L(\f\circ\pi)=\f\times L(1),$ where $1$ denotes the constant function that takes the value $1$ on $\P_{B}^{1,\mathrm{an}}.$ 
\end{defi}
We denote by $\mathrm{M}^+(\P^1,B)$ the collection of all continuous families of positive measures on $\P_{B}^{1,\mathrm{an}}$. We endow it with the weakest topology for which all evaluation maps are continuous.

We introduce the set $\cF_{B}$ of families 
of positive measures $\{\mu_b\}_{b\in \cM(B)}$
such that $\mu_b$ is supported on $\pi^{-1}(b)$
for all $b\in\cM(B)$, and the map $b\mapsto \mu_b$
is continuous for the weak-$*$ topology on $\cM(B)$.

\begin{thm}
\label{continuousfamilyofmeasures}
For any family  
$\mu=  (\mu_{b})_{b\in\cM(B)}\in\cF_{B}$, we set
\begin{equation}\label{eq:def234}
L_{\mu}(\f)(b)=\int \f|_{\pi^{-1}(b)}d\mu_{b}    
\end{equation}
for any $\f\in \cC^{0}(\P_{B}^{1,\mathrm{an}})$, and any $b\in \cM(B)$. 

Then $L_\mu$ defines a continuous family of positive measures, and the map 
 $\mu\rightarrow L_{\mu}$ establishes a bijection between $\cF_{B}$ and $\mathrm{M}^+(\P^1,B)$.
\end{thm}

\begin{rmk}
The map $L$ is a homeomorphism when $\mathrm{M}^+(\P^1,B)$ is endowed with the weak-$*$ topology, and 
$\cF_{B}$ with the compact open topology. 
\end{rmk}

\begin{rmk}
The space $\mathrm{M}^+(\P^1,B)$ is not compact, but 
for each $M>0$, the subset of continuous families of measures 
with uniformly bounded mass $\le M$ is compact. 
\end{rmk}

Suppose $B=A^\epsilon$, for some $\epsilon\in (0,1]^\N$, so that $\cM(B) \simeq \beta \N$ by Theorem~\ref{theorem:spectrumAeps}. Recall that $\N$ with its discrete topology is a open dense subset of $\beta\N$. 

Since $\pi^{-1}(n)$ is open in $\P^{1,\mathrm{an}}_{A^\epsilon}$, the extension by $0$ of a continuous function on $\pi^{-1}(n)$ 
remains continuous on $\P^{1,\mathrm{an}}_{A^\epsilon}$. It follows that an operator $L\in \mathrm{M}^+(\P^1,A^{\epsilon})$ defines by restriction a sequence of positive measures on
$\P^{1,\mathrm{an}}_{\C_{\epsilon_n}}$, hence 
a sequence of positive measures
$\mu_{n}$ on $\hat{\C}$. Also $L(1)$ is continuous hence bounded so that the mass of $\mu_n$ is uniformly bounded.

Conversely, let $\mu_n$ be any sequence of positive measures of bounded mass on $\P_{\C}^{1}$. For any $\f\in \cC^0(\P^{1,\mathrm{an}}_{A^\epsilon})$, the function $\N \mapsto \R$
sending $n$ to $\int \f\, s_{\epsilon_n}^*d\mu_n$ is bounded
hence we can define
\[
L(\f)(\omega) =  \lim_\omega \int \f\, s_{\epsilon_n}^*d\mu_n
\]
for any $\omega \in \beta\N$. 
It is continuous by the universal property  Proposition \ref{extension}, and $L$ is bounded 
since the mass of $\mu_n$ are uniformly bounded. 

We have thus obtained
\begin{cor}
\label{cor:continuousmeasure}
Pick any sequence $\epsilon\in (0,1]^\N$.
Any sequence of positive measures $\mu_n\in \mathrm{M}^+(\hat{\C})$ of uniformly bounded mass determines a unique continuous family of positive measures on $\P^{1,\mathrm{an}}_{A^\epsilon}$ such 
that 
\[
L(\f)(\omega) =  \lim_\omega \int \f\,  s_{\epsilon_n}^*d\mu_n
\]
for all $\f\in \cC^0(\P^{1,\mathrm{an}}_{A^\epsilon})$ and all $\omega \in\cM(A^\epsilon) \simeq \beta\N$.

Any continuous family of positive measures on 
$\P^{1,\mathrm{an}}_{A^\epsilon}$ is obtained in this way. 
\end{cor}
 
\begin{proof}[Proof of Theorem~\ref{continuousfamilyofmeasures}]
Let $\mu=  (\mu_b)_{b\in\cM(B)}$  be any weakly continuous 
family of positive measures on $\P^{1,\mathrm{an}}_B$ such that the support of $\mu_b$ is included in $\pi^{-1}(b)$, and  define $L_\mu (\f)(b)=  \int \f d\mu_b$ for any $\f\in \cC^0(\P^{1,\mathrm{an}}_B)$.
Since the family is continuous, $L_\mu(\f)$ is continuous. The operator is  positive, and if $\f\in \cC^0(\cM(B))$ we have
$L_\mu(\f\circ \pi) (b) = \f(b) \times \mathrm{Mass}(\mu_b)$
so that $L$ defines a continous family of positive measures as in Definition~\ref{def:continuous-family}. 

\smallskip

Let us prove that $\mu \mapsto L_\mu$ is injective. 
If $\mu'= (\mu'_b)_{b\in\cM(B)} \neq \mu$, then we can find $b\in\cM(B)$
such that $\mu'_b\neq \mu_b$, hence a continuous function $\f_0\in \cC^0(\pi^{-1}(b))$ such that $\int \f_0 d\mu'_b \neq \int \f_0 d\mu_b$. 
By Tietze-Urysohn's theorem applied on the compact space $\P^{1,\mathrm{an}}_B$, there exists a continuous function $\f$ on $\P^{1,\mathrm{an}}_B$ whose restriction to 
$\pi^{-1}(b)$ is equal to $\f_0$. 
We have $L_{\mu'}(\f)(b)\neq L_\mu (\f)(b)$ as required. 

\smallskip

We now show that $\mu \mapsto L_\mu$ is surjective.
Let $L\colon \cC^0(\P^{1,\mathrm{an}}_B) \to \cC^0(\cM(B))$ be any positive linear functional satisfying~\eqref{eq:def234}.
Choose any $b\in \cM(B)$, and pick any continuous function $\f\in\cC^0(\pi^{-1}(b))$. Recall that Tietze-Urysohn's theorem implies the existence of 
$\tilde{\f}\in \cC^0(\P^{1,\mathrm{an}}_B)$ such that $\sup |\tilde{\f}|\le \sup |\f|$
and $\tilde{\f}|_{\pi^{-1}(b)} = \f$. 
We set 
\[
L_b (\f) =  L(\tilde{\f})(b)~.
\]
Let us prove that this does not depend on the choice of extension $\tilde{\f}$. For this, it is sufficient to prove that if 
$g\in \cC^0(\P^{1,\mathrm{an}}_B)$ is such that $g|_{\pi^{-1}(b)} = 0$, then 
$L(g)(b)=0$. Writing $g= \max\{ g,0\} - \max\{-g,0\}$, we may suppose that $g\ge0$.
Observe that 
$b'\mapsto Sg(b')=  \sup_{\pi^{-1}(b')} g$
is continuous, and 
$0\le g
\le Sg \circ \pi$.
By the positivity of $L$, we obtain
\[
0 \le L(g)(b)  \le 
Sg (b) \times L(1)(b)=0
~.\]
This proves $L_b$ is a positive linear functional, hence defines a positive measure $\mu_b$ supported on the fiber $\pi^{-1}(b)$. 
The fact that $b\mapsto \mu_b$ is weakly continuous follows immediately from the fact that  $L(\f)$ is continuous when $\f$ is. 
\end{proof}

\subsection{Model functions}
\label{sec:modelfunction}
In order to study the convergence of measures on $\P^{1,\mathrm{an}}_B$, it is important to 
work with a suitable class of functions of geometric origin that we now introduce. We refer to~\cite{MR1629925,zbMATH06430693BFJMA,zbMATH07510519BGMMA,zbMATH07219254CFdegeneration,poineau2024dynamique} for a study of this class when $B$ is a field, and $\P^1$ is replaced by any projective variety. 

\smallskip

Recall that for each $l\in\N$ the space of global sections of the line bundle
$\cO(l)\to \P^1_B$ can be identified with the space of homogeneous polynomials
$P\in B[z_{0}, z_{1}]$ of degree $l$.
Recall that we set $\log \psi_P(z_0,z_1)=  \log |P(z_0,z_1)| - l \log \max\{|z_0|, |z_1|\}$.

For any finite collection of sections $\sigma\subset H^0(\P^1_B,\cO(l))$
we set 
\[\varphi_{\sigma}=  \log\max_{P_{i}\in\sigma}\psi_{P_i}~.\]
Note that $\varphi_{\sigma}\colon\P_{B}^{1,\mathrm{an}}\rightarrow [-\infty,+\infty)$ is continuous. Such functions are called \emph{quasi-model} functions.

The following observation is immediate from the local form of a quasi-model function.

\begin{lemma}\label{lem:qmlaplacian}
Let $\f=\f_\sigma$ be any quasi-model function associated with a finite set of sections $\sigma \subset H^0(\P^1_B,\cO(l))$. Then for any $b\in\cM(B)$, $\Delta \f|_{\pi^{-1}(b)}$ is
a well-defined signed Borel measure whose variation has total mass $\le 2l$. 
\end{lemma}

\begin{proof}
   Observe that $\Delta\varphi|_{\pi^{-1}(b)}+l\mu_{can}$ is locally equal to 
   $\Delta \max \log |P_i|$, hence is a positive measure. It follows that we can write $\Delta \varphi|_{\pi^{-1}(b)}= (\Delta\varphi_{\pi^{-1}(b)}+l\mu_{can})-l\mu_{can}$, hence the claim. 
   \end{proof}

\begin{rmk}
It is actually true that $\Delta \f|_{\pi^{-1}(b)}$ forms a continuous family of (signed) measures. We shall neither prove nor use this fact. 
\end{rmk}

A \emph{model} function is a quasi-model function associated with a finite family of sections $P_i$ having no zeroes in common (this is equivalent to say that the linear system induced by these sections has no base point). 
Any model function is a continuous function $\P^{1,\mathrm{an}}_B\to \R$. 
Let $\cD(\P^{1,\mathrm{an}}_B)$ denote the $\Q$-vector subspace of $\cC^0(\P^{1,\mathrm{an}}_B)$ generated by all model functions.

\smallskip

When $l=0$, global sections of the trivial line bundle are exactly $B$. A finite subset $(b_i)\subset B$ does not vanish simultaneously on $\P_{B}^{1}$ if and only if $(b_{i})$ is not included in any proper ideal of $B$. 
The latter condition is equivalent to say $\Sigma b_{i}B=B$. Therefore, $\max \log|b_{i}|$ is a model function if and only if $\Sigma b_{i}B=B$. In particular, $\log|b|$ is a model function if and only if $b$ is a unit in $B$.

Let $\cD(\cM(B))$ be the $\Q$-vector subspace in $\cC^0(\cM(B))$ generated by all functions of the form $\max \log|b_i|$ with $\Sigma b_{i}B=B$. We say that $B$ satisfies property (N) if there exists $\varphi\in\cD(\cM(B))$
such that $\sup \varphi <0$. 

Recall that the spectral norm on $B$ is defined as
$\rho(b) = \lim_n \lV b^n \rV^{1/n} = \sup_{\cM(B)} |b|$.
    Observe that property (N) is implied the next condition:
 \begin{itemize}
     \item there exist $b_1, \cdots , b_n$ such that $\sum b_i B =B$ and $\rho(b_i) <1$ for all $i$.
 \end{itemize}

We discuss some examples of Banach rings at the end of this section. 

\begin{thm}
\label{dense}
When $B$ has property (N), then 
$\cD(\P^{1,\mathrm{an}}_B)$ is dense in $\cC^0(\P^{1,\mathrm{an}}_B)$. 
\end{thm}

\begin{lemma}
\label{lemmadifference}
The space $\cD(\P^{1,\mathrm{an}}_B)$
is stable under maximum, and contains all functions of the form $\varphi \circ \pi$ with $\varphi \in\cD(\cM(B))$.
\end{lemma}
\begin{proof}
Let $\sigma$ be any finite family of sections $(P_i)_{1\le i \le m}$ of $\cO(d)$ having no zeroes in common. 
For any $l\in\N$, define $\sigma^l$ to be the set of sections $\prod_{i_1,\cdots, i_l} P_{i_1}\cdots P_{i_l}$. 
Then $\varphi_{\sigma^l}=l\times \varphi_\sigma$.

If $\tau$ is a family of sections $(Q_j)$
of $\cO(m)$, then 
the family $\sigma \otimes \tau=(P_iQ_j)$ of sections of $\cO(l+m)$
satisfies
\[
\varphi_{\sigma\otimes \tau}
= 
\varphi_{\sigma}
+
\varphi_{\tau}
.\]

Suppose now that $\varphi=\sum_{i=1}^{k}q_{i}\varphi_{\sigma_i}-\sum_{i=1}^{l}p_{l}\varphi_{\tau_l}\in \cD(\P^{1,\mathrm{an}}_B)$, 
with $q_i, p_l\in \Q^+$. 
Then for some sufficiently divisible $m\in \N$, $mq_i, mp_l$ are all integers and we have
\begin{align*}
    m \varphi
= \varphi_{\otimes_i \sigma_i^{mq_i}}
- 
\varphi_{\otimes_j \tau_j^{mp_j}}
\end{align*}

This proves that any element in $\cD(\P^{1,\mathrm{an}}_B)$ is the difference of two functions associated with a finite collection of sections having no common zeroes.

Let us now check that 
$\cD(\P_{B}^{1,\mathrm{an}})$ is stable under taking maxima.
Since $\max\{a-b,c\}=\max\{a,c+b\} -b$, 
the preceding fact implies that 
it is sufficient to treat the case
$\max \{\varphi_{\sigma_1}, \varphi_{\sigma_2}\}$
for two finite collections of sections 
$\sigma_1\subset H^0(\cO(l_1))$, and $\sigma_2\subset H^0(\cO(l_2))$ with $l_1\le l_2$. 
Let $\tau = (z_0^i z_1^{l_2 -l_1 -i})$. 
Then we have 
$\max \{\varphi_{\sigma_1}, \varphi_{\sigma_2}\}= \varphi_\sigma$
with
$\sigma = (\sigma_1\otimes \tau) \cup \sigma_2 \subset H^0(\cO(l_2))$.

Finally suppose $b_i\in B$ satisfy $\sum b_i B=B$, and let $\varphi=  \max \log|b_i|\in\cC^0(\cM(B))$.
Then we have 
$\varphi \circ \pi= \varphi_\sigma$ with
$\sigma =\{b_iz_0, b_iz_1\}_i\subset H^0(\cO(1))$ which implies the last claim.
\end{proof}

\begin{proof}[Proof of the Theorem \ref{dense}]
We apply the Stone-Weierstrass theorem for lattices, see, e.g., \cite[Theorem 7.28]{stromberg1965real}. 

 It suffices to show that $\cD(\P_{B}^{1,\mathrm{an}})$ separates points. Choose $x\not=y\in \P_{B}^{1,\mathrm{an}}$. 
Translating by $1$ if necessary, we may assume that
the two points are in the same affine chart $\{[z\colon1]\} \subset\A^{1,\mathrm{an}}_B$, then we can find a polynomial $P_0\in B[z]$ of degree $d$  such that $|P_0(x)| \neq |P_0(y)|$, and the homogeneous polynomial 
 $P(z_0,z_1) =  z_1^d P(z_0/z_1)$ satisfies 
 $\psi_P(x) \neq \psi_P(y)$.

Pick any integer $n\in\N$. By assumption, there exists a model function $\varphi\in\cD(\cM(B))$ such that $\sup\varphi<0$. Multiplying this function by a suitable integer, we may find $\varphi_n$ such that
$\varphi_n < -n$,  and $\varphi_n = \psi_n^+ - \psi_n^-$
where $\psi_n^\pm= \log \max |b_{i,n}^\pm|$
where $\sum_i b_{i,n}^\pm B=B$. 
Define
 \[\sigma_n^\pm =  \left\{P(z_{0},z_{1}),b_{i,n}^\pm z_{1}^{d-j}z_0^j\right\}_{i,j}~.\]
We have
$\varphi_n=  \varphi_{\sigma_n^+} -  \varphi_{\sigma_n^-}
= \max \{\psi_P, \varphi_n\circ\pi\}\in\cD(\P^{1,\mathrm{an}}_B)$
and $ \varphi_{n} \to \psi_P$ as $n \to \infty$, hence for any $n$ large enough $\varphi_n$
separates $x$ and $y$ as desired.
\end{proof}

\begin{cor}
\label{cordenseproductbanach}
For any sequence $\epsilon\in(0,1]^\N$, the set of model functions $\cD(\P_{A^\epsilon}^{1,\mathrm{an}})$ is dense in the set of continuous functions $\cC^{0}(\P_{A^\epsilon}^{1,\mathrm{an}})$.
\end{cor}
This follows immediately from the previous theorem by considering the unit $b=  (2^{-1/\epsilon_n})\in (A^\epsilon)^\times$ which satisfies $\rho(b) = 1/2$ so that $\log|b|<-\log 2$.

\medskip

We conclude this section by discussing some examples of Banach rings satisfying or not property (N). 
\begin{itemize}
    \item 
A complete valued field $(k,|\cdot|)$ has the property (N) if and only if the norm is non-trivial (even though in all cases $\cM(k)$ is reduced to a single point). 
\item 
If the complete residue field of some $b\in\cM(B)$ is 
trivially valued, then $B$ cannot have property (N). 
This is the case when $B$ is a field $k$ endowed with the hybrid norm $|\cdot|_\mathrm{hyb}= \max\{|\cdot|_0, |\cdot|\}$.
\item 
Suppose $(k,|\cdot|)$ is a non-Archimedean complete valued field. Then the ring $(k^\circ, |\cdot|)$ has never property (N). 
\item 
For any non-trivially valued complete valued field $(k,|\cdot|)$, any Banach $k$-algebra (for which $\lV b \rV = |b|$ for all $b\in k$) satisfies (N). 
\item 
If all complete residue fields are isometric to $(\C,|\cdot|)$, 
then $B$ can be canonically embedded into $\cC^0(K,\C)$ where $K(=\cM(B))$ is a compact space, and it satisfies property (N). 
\end{itemize}

\begin{rmk}
Let $B$ be any Banach ring (with unity). If $\cD(\cM(B))$ 
separates points in $\cM(B)$, then it
is dense in $\cC^0(\cM(B))$. In that case, $B$ has property (N) and $\cD(\P^{1,\mathrm{an}}_B)$ is also dense in $\cC^0(\P^{1,\mathrm{an}}_B)$. It is however unclear whether $\cD(\cM(B))$ 
separates points if and only if $B$ has property (N).
\end{rmk}

\begin{rmk}
Let $B$ be \emph{any} Banach ring. Then the $\Q$-vector space generated by the constant function $1$ and all model functions is dense in $\cC^0(\P^{1,\mathrm{an}}_B)$.
\end{rmk}

\subsection{Push-forward of continuous functions}
In the remaining two subsections, we work in the ring $A^\epsilon$. Let $f\in \Rat_d(A^\epsilon)$.

The pushforward of a function on the Berkovich projective line over a metrized field is defined in \S\ref{sec:equilibriummeasure}, see~\eqref{eq:pushforwardcontinuousfunction}. 
We define the pushforward of a  function $\varphi\colon\P_{A^\epsilon}^{1,\mathrm{an}}\to\R$ by $f\in\Rat_d(A^\epsilon)$ as follows:
\begin{equation}\label{eq:push-continuous-fnts}
    f_{*}\varphi (x)=  (f_{\omega})_*\varphi|_{\pi^{-1}(\omega)}(x),~\text{for all }x\in\pi^{-1}(\omega), \text{ and } \omega \in \beta\N.
\end{equation}

\begin{prop}\label{prop:push-beta}
For any continuous function $\f \in \cC^0(\P_{A^\epsilon}^{1,\mathrm{an}})$, the function $f_{*}\varphi$ is also continuous.
\end{prop}

We rely on the following characterization of continuity in terms of ultrafilters.

\begin{lemma}
\label{lemma: continuoussequence}
A function $\f\colon \P_{A^\epsilon}^{1,\mathrm{an}}\rightarrow \R$ is continuous if and only if for all $\omega\in\beta\N$, the restriction $\f_{\omega}=\f|_{\pi^{-1}(\omega)}$ is continuous and for any sequence $x_n\in\pi^{-1}(n)$, we have $\lim_{\omega} \f(x_n)=\f(\lim_\omega x_n).$
\end{lemma}

Recall that in the situation above, $\lim_\omega x_n$ is always a Type-1 point in $\P^{1,\mathrm{an}}_{\sH(\omega)}$.

\begin{proof}
The direct implication is immediate. 
We proceed by contradiction for the converse implication. Suppose that $\varphi$ is not continuous. Then there exists $x\in \P_{A^\epsilon}^{1,\mathrm{an}}$ and $\eta>0$, such that for all neighbourhood $U\ni x$ in $\P_{A^\epsilon}^{1,\mathrm{an}}$, there exists $x_U\in U$ such that $|\varphi(x_U)-\varphi(x)|>\eta$. Set $\omega= \pi(x)$.

We claim that for any point $y\in \pi^{-1}(\omega)$, and for any open neighborhood $V\ni y$ in $\pi^{-1}(\omega)$, 
there exists a sequence $y_n\in \pi^{-1}(n)$ such that $\lim_\omega y_n \in V$. 
We only need to consider the case  
$\sH(\omega)$ is non-Archimedean. 
We are also reduced to the case
$y$ is a Type-1 point because these points are dense. 
Then $y=[z_{0,\omega}\colon z_{1,\omega}]$
for a pair $z_{i,\omega}\in\sH(\omega)$, and we may find $[z'_0\colon z'_1]\in \P^1(A^\epsilon)$ such that the point 
$y'=[z'_{0,\omega}\colon z'_{1,\omega}]$ belongs to $V$. This implies the claim in this case since $\lim_\omega [z'_{0,n}\colon z'_{1,n}]=y'$. 
 
From this claim and the assumption $\lim_{\omega} \f(x_n)=\f(\lim_\omega x_n)$, it follows that we may (and shall) assume that $x_U\in \pi^{-1}(\N)$.

\smallskip

Since $\varphi_\omega$ is continuous, there exists a compact neighborhood $\bar{W}$ of $x$ in $\pi^{-1}(\omega)$ such that 
$\sup_{\bar{W}} |\f -\f(x)| \le \eta/2$. We may suppose that $\bar{W}$ is of the form 
$\bar{W} =\cap_{i=1}^k \{ r_i \le \psi_{P_{\omega,i}} \le  s_i\}$
where $r_i< s_i \in \R$, and $P_{i,\omega}$ are homogeneous polynomials with coefficients in $\sH(\omega)$.
Since $A^\epsilon/\ker(\omega)$ is equal to  $\sH(\omega)$, we may even suppose that each $P_{\omega,i}$
can be lifted to a polynomial 
$P_i$ with coefficients in $A^\epsilon$.
Write $W'= \cap_{i=1}^k \{ r_i < \psi_{P_i} < s_i\}\subset \P^{1,\mathrm{an}}_{A^\epsilon}$. Observe that $W'$ is an open neighborhood of $x$ 
such that $\bar{W'}\cap \pi^{-1}(\omega)\subset \bar{W}$.

Consider the set $E$ of integers  $n\in\N$ for which 
there exists $x_n\in\pi^{-1}(n) \cap W'$ such that $|\f(x_n) - \f(x)| >\eta$.
If $E$ is $\omega$-thin, then 
its complement $E^c$ is $\omega$-big, 
and $U= \pi^{-1}\{E^c\in\omega'\} $ is an open neighborhood of $x$. 
The point $x_U$ projects to $E^c$
and satisfies $|\f(x_U) -\f(x)|>\eta$, a contradiction. 

We conclude that $E$ is $\omega$-big, 
and we get a sequence $x_n\in\pi^{-1}(n) \cap W'$ such that $|\f(x_n) - \f(x)| >\eta$. The limit
$x'=  \lim_\omega x_n$ belongs to $\bar{W}$, and $\f(x')=\lim_\omega \f(x_n)$ so that $|\f(x_U) -\f(x)|\ge\eta$. This is impossible, hence
$\f$ is continuous.
\end{proof}

\begin{proof}[Proof of Proposition~\ref{prop:push-beta}]
By the previous lemma, it suffices to show for any continuous function $\f$ and any ultrafilter $\omega\in\beta\N$, and any sequence of points $x_{n}\in\pi^{-1}(n)$, 
we have 
\begin{align*}
\lim_{\omega}f_{*}\varphi(x_n)=f_{*}\varphi(x_\omega),
\end{align*}
where $x_{\omega}=\lim_{\omega}x_n\in\P^1(\sH(\omega))$.
We may suppose that we are in a fixed affine chart so that 
$x_\omega=[z_\omega\colon1]$ and $x_n=[z_n\colon1]$, and assume that $f^{-1}(x_\omega)$ does not contain the point $[1\colon0]$.

Let $y_{\omega}$ be any preimage by $f_\omega$ of $x_{\omega}$ with local degree $m\ge1$. We claim that there exist distinct points $y_{n,j}\in f_{n}^{-1}(x_n)$ such that 
$y_{\omega}=(y_{n,j})$ in $\sH(\omega)$ for all $j$,
and $\sum_j\deg_{y_{n,j}}(f_n) =m$. 

To see this, pick any representation $f_n=[P_n\colon Q_n]$, and consider the polynomial $H_n(z)=P_n(z,1)-z_n Q_n(z,1)$ which defines a polynomial of degree $d$ in $A^\epsilon$ (since $y_\omega\neq [1\colon 0]$ by assumption). 
Since $H_\omega$ has a zero of multiplicity $m$ at $y_\omega$, 
Proposition \ref{propalgc} implies the existence of $m$ sequences 
(possibly with repetitions) $y'_{n,j}$, $1\leq j\leq m$, such that $(y'_{n,j})=y_{\omega}\in\sH(\omega)$, and that $\prod_{j}(z-y'_{n,j})$ is a factor of degree $m$ of $H_{n}$. Set $y_{n,1}=y'_{n,1}$, and 
by induction $y_{n,j}=y'_{n,l}$ where $l$ is the minimal integer 
for which $(y'_{n,l})\notin\{(y_{n,1}, \cdots, (y_{n,j-1})\}$ to avoid repetition. This implies the claim. 

Write $f_\omega^{-1}(x_\omega)=\{y^1_\omega, \cdots, y^k_\omega\}$, 
$m^i=\deg_{y^i_\omega}(f_\omega)$, and choose sequences
$y^i_{n,j}\in f_n^{-1}(x_n)$ such that 
$(y^i_{n,j})=(y^i_\omega)$ in $\sH(\omega)$ for all integers $i,j$, 
and $\sum_j\deg_{y^i_{n,j}}(f_n) =m_i$. 

Note that on an $\omega$-big set, the points $y^i_{n,j}$ are distinct, so that  $\sum_i m_i =d$
implies $f_n^{-1}(x_n)=\{y^i_{n,j}\}$ on that set.
We thus get 
\begin{align*}   
\lim_{\omega}f_{*}\varphi(x_n)&=\lim_{\omega}\sum \left(\deg_{y^i_{n,j}}f_n\right)\varphi (y^i_{n,j})=\lim_{\omega}\sum_{i} \sum_{j}\left(\deg_{y^i_{n,j}}f_n\right)\varphi (y^i_{n,j})\\
    &=\sum_{i}m_i \varphi(y^i_{\omega})=f_{*}\varphi(x_\omega),
\end{align*}
and the proof is complete.
\end{proof}

By the previous proposition, we can define the pull-back of a Radon measure $\mu\in \mathrm{M}(\P_{A^\epsilon}^{1,\mathrm{an}})$ on $\P_{A^\epsilon}^{1,\mathrm{an}}$ by duality. For every continuous function $\varphi\in\cC^{0}(\P_{A^\epsilon}^{1,\mathrm{an}})$, we set
\begin{align*}
    \int \varphi df^{*}\mu= \int f_{*}\varphi\,  d\mu.
\end{align*}
The operator $f^{*}\colon\mathrm{M}(\P_{A^\epsilon}^{1,\mathrm{an}})\rightarrow\mathrm{M}(\P_{A^\epsilon}^{1,\mathrm{an}})$ is  continuous.

\subsection{Push-forward of quasi-model functions}
Let $f\in \Rat_d(A^\epsilon)$, as in the previous section, 
and pick a representation $f=[P\colon Q]$
where $P,Q\in A^\epsilon[z_0,z_1]$ are homogeneous polynomials of degree $d$ determined by a sequence of complex polynomials 
$P_n = \sum a_{j,n} z_0^jz_1^{d-j}$,
$Q_n = \sum b_{j,n} z_0^jz_1^{d-j}$ such that 
$C^{-1}\le |\Res(P_n,Q_n)|^{\epsilon_n} \le C$ for some constant $C>1$.

Recall the norm $\lV z\rV= \sup_n |z_n|^{\epsilon_n}$ on $A^\epsilon$. If $R=\sum c_j z_0^jz_1^{l-j}\in A^\epsilon[z_0,z_1]$ is a homogeneous polynomial of degree $l$, then we let $\lV R\rV = \max \lV c_j\rV$. 
Define 
\[F_*R (z_0,z_1) =  
\prod_{F(w_0,w_1)=(z_0,z_1)} R(w_0,w_1)
\]
where $F=(P,Q)$. 
This map is defined on pairs $(z_0,z_1)\in (A^\epsilon)^2$
such that $z_0 A^\epsilon + z_1 A^\epsilon= A^\epsilon$, and 
$(z_{0,n},z_{1,n})\in \C^2\setminus \{(0,0)\}$ is not in the critical set of $F_n=[P_n\colon Q_n]$ (which is a union of at most $2d-2$ lines). 

\begin{lemma}
 The function $F_*R$ is a homogeneous polynomial of degree $dl$ with coefficients in $A^\epsilon$, and 
 \[
 \lV F_* R\rV \le C \lV R\rV
 \]
 for some constant $C>0$ independent on $R$. 
\end{lemma}
\begin{proof}
We claim that there exists a constant $C>1$
such that
\[
C^{-1} \lV(w_0,w_1)\rV^d \le \lV F(w_0,w_1)\rV \le C \lV(w_0,w_1)\rV^d
\]
for all $w_0,w_1\in (A^\epsilon)^2$. 
The upper bound is trivial. For the lower bound, by~\cite[Theorem~2.13]{Sil07} applied to $P$ and $Q$ on the ring $A^\epsilon[z_0,z_1]$, we get homogeneous polynomials 
$U_i,V_i\in A^\epsilon[z_0,z_1]$ of degree $d-1$ such that 
\[
P U_0 + Q V_0 = \Res(P,Q) z_0^{2d-1}
\text{and }
P U_1 + Q V_1 = \Res(P,Q) z_1^{2d-1}
\]
and the claim follows from 
$\lV P U_i + Q V_i \rV \le 2 \max\{\lV P\rV \lV U_i\rV, \lV Q\rV \lV V_i\rV\}$. 
In particular we get
$\lV (z_0,z_1) \rV \ge C^{-1} \lV (w_0,w_1) \rV^d$
when $F(w_0,w_1)=(z_0,z_1)$ hence
\begin{equation}\label{eq:superuseful}
\lV F_* R (z_0,z_1)
\rV \le C \rV R \lV \lV (z_0,z_1) \rV^{dl}    
\end{equation}
whenever it is defined. 
It follows that for each $n$, $F_{n*}R_n$ is a well-defined holomorphic function on the complement of finitely many lines which hence extends to $\C^2$. 
It is also homogeneous of degree $dl$, hence $F_{n*}R_n$ is a homogeneous polynomial.  We conclude that~\eqref{eq:superuseful} holds for any $(z_0,z_1)\in (A^\epsilon)^2$, and this implies the lemma. 
\end{proof}
The previous construction gives 
a way to define 
a linear map $\sigma \mapsto f_*\sigma$
from $H^0(\P^1_{A^\epsilon},\cO(l))$
to $H^0(\P^1_{A^\epsilon},\cO(dl))$. 
Note that we slightly abuse notation, since this map depends on the choice of a lift $F$.

\begin{prop}\label{prop:push-quasimodel}
Let $\sigma\in H^0(\P^1_{A^\epsilon},\cO(l))$. 
The difference $\f_{f_*\sigma}- d f_*\f_\sigma $
is continuous on $\P^{1,\mathrm{an}}_{A^\epsilon}$.
\end{prop}

\begin{rmk}
Since the maximum of a sum is not necessarily the sum of maxima, it is a priori unclear whether $f_*$ preserves the sets of quasi-model and model functions.  
\end{rmk}

\begin{proof}
The difference $\f_{f_*\sigma}- d f_*\f_\sigma $ is equal to
\begin{align*}
    dl\sum_{F(w_0,w_1)=(z_0,z_1)} \deg_{[w_0\colon w_1]}(f)\log \max \{|w_0|, |w_1|\} - dl \log \max \{|z_0|, |z_1|\},
\end{align*}
which is the same as $l\, f_* (d\log\max\{|z_0|,|z_1|\}-\log\max\{|P|, |Q|\})$. This function is continuous 
by Proposition~\ref{prop:push-beta}.
\end{proof}

\section{Convergence of measures on $\P_{A^\epsilon}^{1,\mathrm{an}}$}
\label{sec:convergence of measures}

\subsection{Basics in hyperbolic geometry}
\label{sec: hyperbolicgeometry}

A basic reference for this section is~\cite[\S 4]{zbMATH03838373BFdiscretegroup}.

We fix an affine coordinate $z$ on the Riemann sphere $\hat{\C}$, and
work with the Poincaré model of the  hyperbolic $3$-space $\H^{3}$. 
The latter is the Riemannian threefold
whose underlying space is
\[\H^{3}=\{(z,h)\in\C\x\R_{+}\},\]
endowed with the metric $g_{\H}=h^{-2}(d|z|^2+dh^2)$ (of constant sectional curvature $-1$), where $|\cdot|$ denotes the standard Euclidean norm on $\C$. The 
induced distance is given by 
\begin{align}
\label{distanceequation}
    \cosh d_{\H}\left((z_{1},h_{1}),(z_{2},h_{2})\right)=1+\frac{|z_{1}-z_{2}|^{2}+|h_{1}-h_{2}|^{2}}{2h_{1}h_{2}}.
\end{align}
The group of direct isometries of $\H^{3}$ is the group of Möbius transformations $\PGL_{2}(\C)$. When $M(z) =\frac{az+b}{cz+d}$, then we have
\begin{align*}
    M(z,h)=\left(\frac{(az+b)(\overline{cz+d})+a \bar{c}h^2}{|cz+d|^{2}+|c|^2h^{2}},\frac{|ad-bc|\, h}{|cz+d|^{2}+|c|^2h^{2}}\right).
\end{align*}
In particular $M(z,h)=(az+b,|a|h)$ when 
$M(z)=az+b$ with $a\in \C^*$ and $b\in \C$, and
\begin{align*}
    M(z,h)=\left(\frac{\Bar{z}}{|z|^{2}+h^{2}},\frac{h}{|z|^{2}+h^{2}}\right)
\end{align*}
for the inversion $M(z)=z^{-1}$.
The stabilizer $G$ of the point $x_\star=  (0,1)$ is the compact group
\begin{align*}
   G= 
\left\{
\left(\begin{array}{cc}
a     &  -\bar{c} \\
c & \bar{a}
\end{array}
\right), \text{ with } |a|^2+|c|^2 =1
\right\}
\simeq\mathrm{SO}(3)
\end{align*}
so that $\H^3$ can be identified by the map 
$M\mapsto M\cdot x_\star$ with the homogeneous space $\PGL_2(\C)/G$.
Observe also that $\PGL_2(\C)$ acts naturally and continuously on 
$\bar{\H}^3= \H^3\sqcup \left(\C \cup\{\infty\}\right) = \H^3\sqcup \hat{\C}$.

Let now $\cG(\hat{\C})$ be the space of conformal metrics 
on the Riemann sphere that have constant curvature $4\pi$. 
Define the Fubini-Study metric as
\begin{align*}
   g_{FS} = \frac{d|z|^2}{\pi(1+|z|^2)^2}~.
\end{align*}

Then $\PGL_2(\C)$ acts also on $\cG(\hat{\C})$
and the isotropy group of $g_{FS}$ is again equal to $G$. 
It follows that the map 
$\mu\colon \H^3 \to \cG(\hat{\C})$ given by 
$\mu(M\cdot x_\star)=  M_*g_{FS}$
is well-defined, and  $\mu$ is a homeomorphism. 

\smallskip

Finally pick any $x\in \H^3$ different from $x_\star$, and consider the set \begin{equation} \label{def:Hx}H(x)= \{y\in\H^3, d_\H(y,x) < d_\H(y,x_\star)\}.
\end{equation}
When $x=(0,t)$ with $0<t<1$, we have
$H(x)= \{(z,h)\in \C \times \R^*_+, |z|^2< t- h^2\}$ which is a half-sphere orthogonal to the plane $\C\times \{0\}$, and its closure in $\bar{\H}^3$ intersects $\hat{\C}$
along the euclidean disk $\bar{\D}(x) = \{|z|<\sqrt{t}\}$. 

Recall that the spherical distance on $\hat{\C}$ is defined by 
\begin{align*}
d_{\P^{1}(\C)}([z_{0}\colon z_{1}],[w_{0}\colon w_{1}])=\frac{|z_{0}w_{1}-z_{1}w_{0}|}{\sqrt{|z_{0}|^2+|z_{1}|^2}\sqrt{|w_{0}|^2+|w_{1}|^2}}.
\end{align*}
Spherical disks are preserved by Möbius transformations, and the group $G$ preserves also their radii.
We have thus proved
\begin{lemma}\label{lem:associated-ball}
For any $x\neq x_\star\in\H^3$, the intersection $\bar{\D}(x)$ of
the closure of $H(x)$ in $\bar{\H}^3$ with $\hat{\C}$
is a closed disk of radius $\sqrt{t}$ with \[\cosh d_\H (x,x_\star) = 1+ (t-1)^2/2t\] for the spherical metric on the Riemann sphere. 
\end{lemma}
It is convenient to call $\bar{\D}(x)\subset \hat{\C}$
the \emph{associated disk} with $x$.

\subsection{Convergence of metrics of constant curvature on $\P_{A^{\epsilon}}^{1,\mathrm{an}}$}
\label{sec: covergence of constant curvature}
Fix $\epsilon=(\epsilon_{n})\in (0,1]^{\N}$, and recall that the Banach ring  $A^{\epsilon}$ is defined as
\[A^{\epsilon}=\left\{(a_{n})\in \C^{\N}, \, \sup_n|a_{n}|^{\epsilon_n}<\infty\right\}\]
equipped with the norm $\lV(a_{n})\rV=\sup_{n}|a_{n}|^{\epsilon_n}.$ Recall also the canonical homeomorphism 
$s_e\colon \P^{1,\mathrm{an}}_{\C_e} \to \hat{\C}$ from \S\ref{sec:P1overring} for any $e\in(0,1]$.

Recall from \S\ref{sec:projective line over A}
that a sequence $p\in(\hat{\C})^\N$ defines a point 
$\alpha(p)\in\P^1(A^\epsilon)$. 
Given $\omega\in\beta\N$, we shall write $p_\omega$ for the point
defined by $\alpha(p)$ in the fiber $\pi^{-1}(\omega)$. 
Observe that $p_\omega$ is \emph{always} of Type-1.

\smallskip

Fix a sequence of conformal metrics $\rho_n\in\cG(\hat{\C})$ of constant curvature $4\pi$ on the Riemann sphere. 
We identify such a metric $h_n(z) d|z|^2 \in \cG(\hat{\C})$ with the smooth positive $(1,1)$-form
$\lambda_{n}=  h_n(z) \frac{i}{2}dz\wedge d\bar{z}$. By Gauss-Bonnet, this volume form has total mass $1$
so that $\mu_n\in\mathrm{M}^+(\hat{\C})$ in the notation of \S\ref{sec:continuousfamilyofmeasures}. By Corollary~\ref{cor:continuousmeasure}, let $\mu_n=(s_{\epsilon_n})^*\lambda_n$, we obtain
a continuous family of probability measures $(\mu_\omega)_{\omega\in \beta\N}$, and our aim is to describe $\mu_\omega$ for any non-principal ultrafilter $\omega$. 

Recall that the Fubini-Study metric is equal to $\mu(x_\star)$ with $x_\star=(0,1)\in\H^3$; and from Table~\ref{diagram:residue} that a point $p=(p_n)\in (\hat{\C})^\N$
can be identified with a point in $\P^1(A^\epsilon)$ hence defines a point in $p_\omega\in\P^1(\sH(\omega))$ for all $\omega$.

\begin{thm}
\label{convergence-conformal}
Pick  any sequence of points $x_n\in\H^3$, and consider the probability measures $\mu_n = \mu(x_n)\in \mathrm{M}^+(\hat{\C})$. 
Choose any sequence $\epsilon=(\epsilon_{n})\in (0,1]^\N$, and any $\omega\in\beta\N$. 
\begin{enumerate}
    \item Suppose that $e=  \lim_\omega \epsilon_n >0$, so that $\sH(\omega)\simeq\C_e$ is Archimedean.     
    \begin{enumerate}
\item
    If $\lim_\omega d_{\H}(x_n,x_\star) <\infty$, then $x_\omega=  \lim_\omega x_n$ belongs to $\H^3$, and    
    $(s_e)_* \mu_\omega=  \mu(x_\omega)$ is a conformal metric with constant sectional curvature.
\item 
If $\lim_\omega d_{\H}(x_n, x_\star) =\infty$, 
then $x_{\omega}=  \lim_\omega x_n$ belongs to $\hat{\C}$
and $(s_e)_*\mu_\omega$ is the Dirac mass at the point $x_\omega$.
\end{enumerate}
\item 
Suppose that $\lim_\omega \epsilon_n =0$, so that $\sH(\omega)$ is non-Archimedean.
\begin{enumerate}
\item 
 If $\lim_\omega {\epsilon_n} d_{\H}(x_n, x_\star)=0$, then $\mu_\omega$ is a Dirac mass at the Gauss point.
 \item 
 If $\lim_\omega {\epsilon_n} d_{\H}(x_n, x_\star) \in\R_+^*$, $\mu_\omega$ is a Dirac mass at a point of Type-2 
$x_\omega\in\P^{1,\mathrm{an}}_{\sH(\omega)}$. 
\item 
if $\lim_\omega {\epsilon_n} d_{\H}(x_n,x_\star) =\infty$, then $\mu_\omega$ is a Dirac mass at a point $x_\omega\in\P^{1,\mathrm{an}}(\sH(\omega))$ of Type-1. Moreover, 
for any sequence $y_n\in \bar{\D}(x_n)$, we have $y_\omega = x_\omega$.
\end{enumerate}
\item 
Pick any other sequence $y_n\in\H^3$ such that $\lim_\omega \epsilon_n d_{\H}(x_n,y
_n)=0$, and let $\nu_n=  \mu(y_n)$.
Then the two limiting measures coincide: 
$\nu_\omega =\mu_\omega$.
\end{enumerate}
\end{thm}

\begin{rmk}
It is possible to interpret geometrically the Type-2 point $x_\omega$ which appears in Case (2b). 
Indeed the image of the set
 \[\{ (y_n)\in (\hat{\C})^\N, y_n\in\bar{\D}(x_n)\}\]
 in $\P^1(\sH(\omega))$
is a closed projective disk defining $x_\omega$.
\end{rmk}

Recall that the canonical measure $\mucan$ is defined as the Dirac mass at the Gauss point in the non-Archimedean case, and as the Haar measure on the unit circle $\{|z|=1\}$
in the Archimedean case. 

\begin{cor}\label{cor:convergence-can}
The family of measures $\{\mucan|_{\pi^{-1}(\omega)}\}$ is 
continuous. 
\end{cor}
\begin{proof}
It follows from~\eqref{eq:pot-epsilon} that 
for any integer $n$, we have 
$\mucan|_{\pi^{-1}(n)} - \mu_{FS,n} = \Delta g_n$ 
with
\begin{align*}
g_n= \log \frac{\max\{|z_0|^{2/\epsilon_n}, |z_1|^{2/\epsilon_n}\}}{(|z_0|^{2/\epsilon_n}+|z_1|^{2/\epsilon_n})^{1/2}}~.
\end{align*}
Observe that the sequence $\{g_n\}$ defines a continuous function on $\P^{1,\mathrm{an}}_{A^\epsilon}$ whose restriction to $\pi^{-1}(\omega)$ is equal to $0$ when $\sH(\omega)$
is non-Archimedean. 

When $\lim_\omega \epsilon_n =0$, then $\lim_\omega \sup |g_n| =0$. 
For any model function $\f$, we have
$\int \f d\mucan|_{\pi^{-1}(n)} 
-\int \f d\mu_{FS,n} = \int g_n \Delta \f$,
and $\lim_\omega \int g_n \Delta \f=0$. 
Since $\lim_\omega \mu_{FS,n}=\mucan$ by Theorem~\ref{convergence-conformal} (2a)
we conclude that $\lim_\omega \int \f \mucan|_{\pi^{-1}(n)} = \f(x_g)$ for all model function $\f$, hence
$\lim_\omega  \mucan|_{\pi^{-1}(n)}=\mucan|_{\pi^{-1}(\omega)}$ by density. 

In the case  $\lim_\omega \epsilon_n >0$, then we can argue directly observing that the measure $\lim_\omega \mucan|_{\pi^{-1}(n)}$
is then a probability measure invariant by rotation hence
equal to $\mucan|_{\pi^{-1}(\omega)}$.
\end{proof}

The rest of the section is devoted to the proof of Theorem~\ref{convergence-conformal}. 

\begin{proof}[Step 1] 
We suppose that $x_n=x_\star$ for all $n$, and compute $\mu_\omega$ in this case. 

\smallskip

Write $\mu_{FS}= \mu(x_\star)\in\mathrm{M}^+(\hat{\C})$.
Pick any model function of the form
$\varphi_\sigma$ where $\sigma=\{P_i\}$ is a finite subset of sections of $H^0(\cO(d))$ having no base points.  Write $P_i=\sum a_{j} z_0^jz_1^{d-j}$ with $a_j=(a_{j,n})\in A^\epsilon$. By assumption, we have
\[
\int \varphi_\sigma d\mu_\omega
=
\lim_\omega
\int \varphi_\sigma (s_{\epsilon_n})^*d\mu_{FS}
=
\lim_\omega
\int_{\hat{\C}}
{\epsilon_n} \log\frac{\max_j\left\{\left|\sum a_{j,n} z_0^jz_1^{d-j}\right|\right\}}
{\max\{|z_0|,|z_1|\}^d} d\mu_{FS}
~.\]
Suppose first that $\epsilon= \lim_\omega \epsilon_n>0$ so that $\sH(\omega)$ is Archimedean. Then 
$a_{j,n}$ is bounded and $a_{j,\omega}= \lim_\omega a_{j,n} \in \C$ so that
we can define $P_{j,\omega}= \sum a_{j,\omega} z_0^jz_1^{d-j}\in\C[z_0,z_1]$. We then have
\begin{equation}\label{eq:arch-conv}
\int \varphi_\sigma d\mu_\omega
=
\int_{\hat{\C}}
\epsilon \log\frac{\max_j\left\{\left|P_{j,\omega}\right|\right\}}
{\max\{|z_0|,|z_1|\}^d} d\mu_{FS}
\end{equation}
which implies $(s_{\epsilon})_*\mu_\omega=\mu_{FS}$.

\smallskip 

Suppose now that $\lim_\omega \epsilon_n=0$ so that $\sH(\omega)$ is non-Archimedean.
We claim that the support of $\mu_{\omega}$
is the Gauss point in $\P^{1,\mathrm{an}}_{\sH(\omega)}$, which implies $\mu_{\omega}=\delta_{x_g}$. 

Since $\mu_{FS}$ is invariant under $z\mapsto z^{-1}$, the support of $\mu_\omega$ is also invariant under this inversion.
Choose any $c\in\sH(\omega)$ such that $|c|\le 1$, and any $r<1$. It is sufficient to prove that $\mu_\omega(B(c,r))=0$. Since the image of $A^\epsilon$ equals $\sH(\omega)$
by Proposition~\ref{propresidue} (2), we may find $(c_n)\in A^\epsilon$ whose image $[c_n]=c$ in $\sH(\omega)$. 

Consider the set of homogeneous polynomials $L_1= z_{0}-(c_n) z_{1}$, and 
$L_2= z_{0}-(c_{n}+r^{1/\epsilon_{n}})z_{1}$. Note that the collection
$\sigma=\{L_1, L_2\}$ of sections of $\cO(1)$ has no base point.

 In the chart $x=[z\colon 1]$ with 
$(z_n)\in A^\epsilon$, the model function $\varphi_{\sigma}$  can be computed explicitely, and we obtain
\begin{align*}
\varphi_{\sigma}(x)|_{\A^{1,\mathrm{an}}_{\C_{\epsilon_n}}}&=
\epsilon_n \log \max\{|z_n- c_n|, |z_n -c_n +r^{1/\epsilon_{n}}|\}
- \epsilon_n \log\max\{|z_n|, 1\}
\end{align*}
and taking the limit along $\omega$, we get 
\begin{align*}
\varphi_{\sigma}(x)|_{\A^{1,\mathrm{an}}_{\sH(\omega)}}&
\begin{cases}
     =-r,&|z-c|<r;\\
     \leq 0,&r\leq|z-c|\leq 1;\\
     =0,&|z-c|>1.
     \end{cases}
\end{align*}
Observe that we only need to prove that 
$\int \varphi_\sigma d\mu_\omega=0$. 
But
\begin{align*}
   \int\varphi_{\sigma}d\mu_{\omega}=
\lim_\omega
{\epsilon_n}
\int_{\hat{\C}}
\log \left(\frac{\max\{|z- c_n|, |z -c_n +r^{1/\epsilon_{n}}|\}}{\max\{|z|, 1\}}
\right) d\mu_{FS}(z) 
\end{align*}

In order to estimate the right hand side, 
we use basic ideas from potential theory. We leave it to the reader to use a more direct geometric approach. 

Recall $\mucan$ be the Haar measure on the unit circle $\{|z|=1\}\subset \hat{\C}$.
Recall that 
$\mucan = \mu_{FS} +  \Delta g$
for some continuous quasi-subharmonic function $g$, see \S\ref{sec:potential theory}. 
Write $\varphi_n=  \log \max\{|z- c_n|, |z -c_n +(r^{1/\epsilon_{n}})|\} - \log \max\{|z|, 1\}$, and observe that this defines a continuous function  on $\hat{\C}$. 
We have
\[
\int_{\hat{\C}} \varphi_n d\mu_{FS}
= 
\int_{\hat{\C}} \varphi_n d\mucan
- 
\int_{\hat{\C}} \varphi_n \Delta g
\]
By Jensen's formula~\cite[\S 15.16]{rudin}, we have
\begin{align*}
I_n= \int_{\hat{\C}} \varphi_n d\mucan
&= 
\int_{\hat{\C}} \log\max\left\{|z- c_n|, |z -c_n +(r^{1/\epsilon_{n}})|\right\} d\mucan
\\
&\ge 
\max\left\{
\int_{\hat{\C}} \log |z- c_n| d\mucan, 
\int_{\hat{\C}} \log |z -c_n +(r^{1/\epsilon_{n}})| d\mucan
\right\}
\\
&\ge 
\max\{\log^+|c_n|, \log^+|c_n -(r^{1/\epsilon_{n}})|\}
\ge 0~.
\end{align*}
On the other hand, by Lemma \ref{lem:qmlaplacian} we get
\begin{align*}
J_n=  \left|\int_{\hat{\C}} \varphi_n \Delta g\right|
=
\left|\int_{\hat{\C}}g \Delta \varphi_n\right|
\le \sup|g|\times \mathrm{Mass}(|\Delta \varphi_n|)
\le 2 \sup|g|~,
\end{align*}
and using $\lim_\omega \epsilon_n=0$, we conclude that 
\begin{align*}
\int\varphi_{\sigma}d\mu_{\omega}
=
\lim_\omega
{\epsilon_n}
\int_{\hat{\C}} \varphi_n d\mu_{FS}
\ge 
\lim_\omega
\epsilon_n (I_n - J_n)
\ge 0
\end{align*}
This implies $\int\varphi_{\sigma}d\mu_{\omega}=0$ and concludes the proof.
\end{proof}

\begin{proof}[Step 2]
We suppose now $\lim_\omega \epsilon_n d_{\H}(x_n,x_\star)<\infty$ and determine $\mu_\omega$. In other words, we prove (1a) and the fact that 
$\mu_\omega$ is a Dirac mass at a Type-2 point when $\lim_\omega \epsilon_n=0$.

\smallskip

Choose any sequence of affine maps 
$M_{n}(z)=a_{n}z+b_{n}$ such that $x_{n}=M_{n}(x_\star)\in\H^{3}$. 
\begin{lemma}
\label{degenerationmobius}
    Pick $\omega\in\beta\N$, and $\epsilon\in(0,1]^\N$. Let $M_{n}(z)=a_{n}z+b_{n}$ be any sequence of complex affine Möbius transformations, and write $x_{n}=M_{n}(x_\star)\in\H^{3}$.
    
     Then we have $\lim_{\omega}\epsilon_{n}d_{\H}(x_{n},x_\star)<\infty$ if and only if  $\lim_\omega|b_n|^{\epsilon_n}<\infty$ and $\lim_\omega|a_n|^{\pm \epsilon_n}<\infty$.
\end{lemma}
By this lemma, there exists an $\omega$-big set $E$
such that
\begin{align*}
    \sup_{n\in E} \max\{|b_n|^{\epsilon_n}, |a_n|^{\pm \epsilon_n}\}<\infty.
\end{align*}
 We replace $(a_n, b_n)$ for all $n\notin E$
by the constant sequence $(1,0)$. 
This implies $(b_n)\in A^\epsilon$, and $(a_n)\in (A^\epsilon)^\times$ and does not change
$\lim_\omega \mu_n$. In particular, the sequence $(a_nz+b_n)$ induces an affine automorphism $M(z)=az+b$ of the affine line over $A^\epsilon$.

Let $\varphi_\sigma$ be any model function associated with a finite family $\sigma =\{P_i\}$ of sections
$P_i\in\cO(d)$ over $\P^{1,\mathrm{an}}_{A^\epsilon}$.
Observe that $M^*\sigma =\{P_i\circ M\}$ is again a family of sections over the same space, and
$\varphi_\sigma\circ M = \varphi_{M^*\sigma}$ is a model function.
We have:
\begin{align*}
   \int \varphi_{\sigma} d\mu_\omega &= 
   \lim_{\omega} \int \varphi_{\sigma}(s_{\epsilon_n})^* d\mu_{n}
   \\
   &= 
    \lim_{\omega} \int_{\hat{\C}}
\epsilon_n \log\frac{\max_j\left\{\left|P^n_j\circ M_n\right|\right\}}
{\max\{|a_n z_0 +b_n z_1|,|z_1|\}^d}
d\mu_{FS}
\end{align*}
If $e= \lim_\omega \epsilon_n >0$, then $\sH(\omega)$ is Archimedean, $\sup_E d_{\H}(x_n,x_\star) <\infty$ on some $\omega$-big set $E$, 
and $\lim_\omega x_n = x_\omega\in \H^3$. 
Since $x_n = M_n (x_\star)$, we have $x_\omega= M_\omega (x_\star)$, and the preceding step implies
\[
 \int \varphi_{\sigma} d\mu_\omega = 
\int_{\hat{\C}} e
\log\frac{\max_j\left\{\left|P_{j,\omega}\circ M_\omega\right|\right\}}
{\max\{|a_\omega z_0 +b_\omega z_1|,|z_1|\}^d}
d\mu_{FS} 
\]
which proves (1a). 

\smallskip

Suppose now that $\lim_\omega \epsilon_n =0$. Then $\sH(\omega)$ is non-Archimedean, and by the preceding step we have 
\[
 \int \varphi_{\sigma} d\mu_\omega = 
\int_{\pi^{-1}(\omega)}
\log\frac{\max_j\left\{\left|P_{j,\omega}\right|\right\}}
{\max\{|z_0|,|z_1|\}^d}
(M_\omega)_*\delta_{x_g} ~,
\]
hence $\mu_\omega$ is a Dirac mass at the Type-2 point 
\begin{align}
\label{eqq:surjective}
    x_\omega=  M_\omega (x_g)\in\P^{1,\mathrm{an}}_{\sH(\omega)}.
\end{align}

This ends Step 2.
\end{proof}

\begin{proof}[Proof of Lemma~\ref{degenerationmobius}]

By~\eqref{distanceequation}, we have
    \begin{equation}\label{eq:computx}
        \cosh \, d_{\H}(M_{n}(x_\star),x_\star) 
        = 1+ \frac{|b_n|^2+(|a_n|-1)^2}{2|a_n|}~.
    \end{equation}
Observe that $\frac12 e^t \le \cosh t \le e^t$ for all $t\ge0$. 
Write $\Theta= \lim_{\omega}\epsilon_{n}d_{\H}(x_{n},x_\star)$ and suppose that $\Theta<\infty$. Then 
\begin{align*}  
    \lim_\omega |a_n|^{\pm \epsilon_n} 
\le 
\lim_\omega \max\{2,|a_n|^{\pm 1}\}^{\epsilon_n} 
\le 
\max\{2, 8 e^{\Theta}\} < \infty
\end{align*}
since $\max\{t,t^{-1}\}\le 4(t-1)^2/t$ for all $t>2$;
and 
\[\lim_\omega |b_n|^{2\epsilon_n}
\le 
\lim_\omega (2|a_n|)^{\epsilon_n}
\times 
 e^{\Theta}
<\infty~.\]
Conversely, suppose 
$\lim_\omega|a_n|^{\pm\epsilon_n}<\infty$ and   $\lim_\omega|b_n|^{\epsilon_n}<\infty$. Then 
\begin{align*}
    e^\Theta 
\le 
\lim_\omega \left(
6 \max\{
1, |b_n|^2 |a_n|^{-1}, (|a_n|-1)^2|a_n|^{-1}
\}
\right)^{\epsilon_n}<\infty
\end{align*}

which concludes the proof.
\end{proof}

\begin{proof}[Step 3]
We finally treat the case $\lim_\omega \epsilon_n d_{\H}(x_n,x_\star)=\infty$.

By Lemma~\ref{lem:associated-ball}, the 
projective disk $\bar{\D}(x_n)\subset \hat{\C}$ has a radius $\sqrt{t_n}\in (0,1]$ and 
$\cosh d_{\H} (x_n,x_\star)= 1+(t_n-1)^2/2t_n$. Since
$\lim_\omega d_{\H}(x_n,x_\star)=\infty$ in all cases, we get 
$\lim_\omega t_n =0$, and $d_\H (x_n,x_\star) =  \log t_n^{-1} + o(1)$
on an $\omega$-big set.

Recall that $x_n\in \H^3$ can be mapped to the point $(0,t_n)$ with $t_n\in (0,1)$ as above by an element $g$ of 
the maximal compact subgroup $G$ of $\PGL_2(\C)$ fixing $x_\star$. 
The next lemma thus implies
\begin{equation}\label{eq:winstep4}
\lim_\omega \mu_n(\hat{\C}\setminus \bar{\D}(x_n))= 0~.
\end{equation}

\begin{lemma}\label{lem:proj-measure}
There exists a constant $C>0$ such that for any $t\in (0,1)$, we have 
\[
\mu(x_t) \left(\hat{\C}\setminus \bar{\D}(x_t)\right)
\le 
Ct
\]
where $x_t=(0,t)\in\H^3$.
\end{lemma}

Suppose that  $\lim_\omega \epsilon_n>0$.
Then we can consider the limit along $\omega$ of $x_n$ in $\bar{\H}^3$: this is a point $x_\omega\in \hat{\C}= \bar{\H}^3\setminus \H^3$, and it follows from~\eqref{eq:winstep4} that $\mu_\omega$
is supported on $x_\omega$.

\smallskip

Suppose now that  $\lim_\omega \epsilon_n=0$. Pick any two sequences $y_n,y'_n \in \bar{\D}(x_n)\subset \hat{\C}$. Since $d_\H (x_n,x_\star) =  \log t_n^{-1} + o(1)$, 
we get in the spherical distance 
\begin{align*}
    \lim_\omega  d_{\P^1(\C)}(y_n,y'_n)^{\epsilon_n}
\le 
\lim_\omega 2(\sqrt{t_n})^{\epsilon_n}=0
\end{align*}
and $y_\omega = y'_\omega \in \P^1(\sH(\omega))$  by Lemma~\ref{lem:limit-proj-dist}. We denote by $x_\omega\in \P^1(\sH(\omega))$ the point determined by any sequence $y_n\in \bar{\D}(x_n)$.

To conclude, we need to prove that $\mu_\omega$ is supported at $x_\omega$.
Pick any positive continuous function $\f$ on $\P^{1,\mathrm{an}}_{A^\epsilon}$ which is vanishing at $x_\omega$. Write $\f_n=  (s_{\epsilon_n})_* \f\colon \hat{\C}\to \R_+$. Then we have\[ 0\le  \int \f_\omega d\mu_\omega  = \lim_\omega \int \f_n d\mu_n
   = \lim_\omega \int_{\bar{\D}(x_n)} \f_n d\mu_n~.
\]
We also have $\lim_\omega \sup_{\bar{\D}(x_n)} \f_n =0$, since if $\sup_{\bar{\D}(x_n)} \f_n = \f_n(y_n)$ for some $y_n\in \bar{\D}(x_n)$,
then $y_\omega = x_\omega$ by our previous considerations. It follows that $\int \f_\omega d\mu_\omega =0$ hence $\mu_\omega$ is supported at $x_\omega$. 

This ends the proof of Step 3.    
\end{proof}

\begin{proof}[Proof of Lemma~\ref{lem:proj-measure}]
The point $(0,t)$ is the image by the Möbius map $z\mapsto tz$ of $x_\star$ so that 
we have $\mu(x_t)= \frac{t^2 d|z|^2}{4(t^2+|z|^2)^2}$. 
It follows
\begin{align*}
\mu(0,t) \left(\hat{\C}\setminus \bar{\D}(0,t)\right)
&=
\int_{|z|\ge\sqrt{t}}\frac{t^2 d|z|^2}{4(t^2+|z|^2)^2}
= 2^{-1}\pi t \int_1^\infty\frac{r dr}{(t+r^2)^2}
\end{align*}
hence the result. 
\end{proof}

\begin{proof}[End of the proof of Theorem~\ref{convergence-conformal}]

It remains to prove (3). Pick any sequence $y_n\in \H^3$ such that 
$\lim_\omega \epsilon_n d_\H(x_n,y_n)=0$, and 
let $\nu_n\in\mathrm{M}^+(\hat{\C})$ be the sequence of conformal measures associated with $y_n$.

When $\lim_\omega \epsilon_n>0$ and
$\lim_\omega d_\H(x_n,x_\star)<\infty$,
then $\lim_\omega \mu_n$ is the conformal measure associated with the point $x_\infty= \lim_\omega x_n\in\H^3$ by Step 2. Our assumption implies 
$\lim_\omega y_n= x_\omega\in\H^3$, and Step 2 applied to $y_n$ implies $\nu_\omega$ is also the conformal measure associated with $x_\omega$. 
The same argument applies using Step 3 in the case $\lim_\omega d_\H(x_n,x_\star)=\infty$.

Suppose that $\lim_\omega \epsilon_n=0$. 
When $\lim_\omega \epsilon_n d_\H(x_n,x_\star)< \infty$
then we also have
\[\lim_\omega \epsilon_n d_\H(y_n,x_\star)< \infty,\] and the result follows from the next lemma.
\begin{lemma}\label{lem:distance-omega limite}
Suppose $\lim_\omega \epsilon_n =0$, and $(x_n)$  is a  sequence of points in $\H^3$ such that 
$\lim_\omega \epsilon_n d_{\H}(x_n,x_\star)<\infty$.
Then we have
\begin{equation}\label{eq:computxn}d_{\sH(\omega)} (x_\omega, x_g) = 
\lim_\omega \epsilon_n d_{\H}(x_n,x_\star)
~.
\end{equation}
Moreover, for any other sequence $(y_n)$ such that $\lim_\omega \epsilon_n d_{\H}(x_n,y
_n)=0$, we have
$y_\omega =x_\omega$.
\end{lemma}

When $\lim_\omega \epsilon_n d_\H(x_n,x_\star)= \infty$, then Step 3 implies $\mu_\omega$ is supported on a Type-1 point $x_\omega$. If $y_n\in\H^3$ satisfies 
$\lim_\omega \epsilon_n d_{\H}(x_n,y
_n)=0$, then by the triangle inequality we obtain $\lim_\omega \epsilon_n d_{\H}(x_\star,y_n)=\infty$, and
for $n$ in an $\omega$-big set
then  $d_{\H}(y_n,x_n)$ is much smaller than $d_{\H}( x_\star, x_n)$. This implies 
$y_n\in H(x_n)$ (see~\eqref{def:Hx}), and since $H(x_n)$ and $H(y_n)$ are convex it follows 
that $\bar{\D}(x_n)\cap \bar{\D}(y_n)\neq \emptyset$. 
Choose any sequence of points $z_n\in\hat{\C}$ in this intersection. It follows from Step 3 that $z_\omega= x_\omega= y_\omega$, which concludes the proof of the theorem.
\end{proof}

\begin{proof}[Proof of Lemma~\ref{lem:distance-omega limite}]
Write $x_n = M_n(x_\star)$ with $M_n(z)=a_n z+b_n$. 
By~\eqref{distanceequation}, we have 
\begin{equation}
        \cosh \, d_{\H}(x_{n},x_\star) 
        = 
        \cosh \, d_{\H}(M_{n}(x_\star),x_\star) 
        = 1+ \frac{|b_n|^2+(|a_n|-1)^2}{2|a_n|}~.
    \end{equation}
Observe that $x_\omega$ is the image of the Gauss point under the affine map $a_\omega z + b_\omega$ so that 
    $d_{\sH(\omega)}(x_\omega , x_g)=
2 \log\max\{|a_{\omega}|,|b_{\omega}|,1\} - \log|a_{\omega}|$.

If $\lim_\omega d_{\H}(x_n,x_\star)<\infty$, then 
$a_n^{\pm 1}$ and $b_n$ are $\omega$-bounded, and $|a_\omega|=1$ and $|b_\omega|\le 1$ which implies  $d_{\sH(\omega)}(x_\omega , x_g)
=0$. The result follows since we have $\lim_\omega \epsilon_n d_{\H}(x_n,x_\star)=0$.

Suppose $\lim_\omega d_{\H}(x_n,x_\star)=\infty$. Then 
\begin{align*}
    L=  \lim_\omega \epsilon_n d_{\H}(x_n,x_\star)&= \lim_\omega \epsilon_n \log \left(\frac{|b_n|^2+(|a_n|-1)^2}{|a_n|}+2\right)\\
    &=\lim_\omega \epsilon_n \log \left(\frac{|b_n|^2+|a_n|^2+1}{|a_n|}\right).
\end{align*}
Observe that
\begin{align*}
    \epsilon_{n}\log\max\{|b_n|^{2},|a_n|^2,1\}
&\le
\epsilon_n
\log (|b_n|^2+|a_n|^2+1)\\
&\le
\epsilon_n
\log
\max\{3|b_n|^2,3|a_n|^2,3\}.
\end{align*}
Hence $\lim_{\omega}\epsilon_n
\log (|b_n|^2+|a_n|^2+1)=2\log\max\{1,|a_{\omega}|,|b_{\omega}|\}$. It follows that \[L=d_{\sH(\omega)}(x_\omega,x_g).\]

Pick another sequence $y_n\in \H^3$ such that $\lim_\omega \epsilon_n d_{\H}(x_n,y_n)=0$. Write $y_n= M'_n (x_g)$, with $M_n'(z)= a'_n z + b'_n$. 
Since Möbius transformations are isometries for $d_\H$, we have 
$\lim_\omega \epsilon_n d_{\H}(x_\star,M_n^{-1}(y_n))=0$.
By what precedes, we get $x_g= M_\omega^{-1}y_\omega$. Since we proved $x_\omega= M_\omega(x_g)$, we have $y_\omega=x_\omega$. 
\end{proof}

\subsection{Convergence of equilibrium measures on the hybrid space}
\label{sec: convergence of equilibrium measure}

Note that every quasi-model function is integrable with respect to the equilibrium measure $\mu_{f_\omega}$ of $f_{\omega}$, since the potentials are continuous, see, e.g.,  \cite[Lemmes 2.3 et 4.2]{favre2006equidistribution}). 
\begin{thm}
\label{theoremconvergenceofequilibriummeasure}
Pick $\epsilon\in(0,1]^\N$, and let
$f\in\Rat_d(A^\epsilon)$ for some $d\ge2$. 
For each $\omega\in\beta\N\simeq \cM(A^\epsilon)$, denote by $\mu_{f_\omega}$ the equilibrum measure of $f_{\omega}$. 

Then $(\mu_{f_\omega})_{\omega\in\beta\N}$ defines a continuous family of positive measures on $\P^{1,\mathrm{an}}_{A^\epsilon}$. Moreover, for any quasi-model function $\varphi \colon \P^{1,\mathrm{an}}_{A^\epsilon}\to [-\infty,\infty)$, 
\begin{equation}\label{eq:666}
I_\varphi(\omega)=  \int_{\P^{1,\mathrm{an}}_{\sH(\omega)}} \varphi d\mu_{f_\omega}
\end{equation}
is finite for any $\omega\in\beta\N$, and the function 
$\omega \mapsto I_\varphi(\omega)$
is continuous. 
\end{thm}

Recall the definition of the pushforward of a continuous function from~\eqref{eq:push-continuous-fnts}.
The rest of this section is devoted to the proof of 
Theorem~\ref{theoremconvergenceofequilibriummeasure}. 

\begin{proof}[Proof of Theorem~\ref{theoremconvergenceofequilibriummeasure}]
Since model functions are dense in the space of continuous functions, 
~\eqref{eq:666}  implies that the family of measures
$\mu_{f_\omega}$ is continuous. 

Let us prove~\eqref{eq:666}.
By Theorem~\ref{extension}, the function $I_\f\colon \beta\N\to\R$
is continuous if and only if $\lim_\omega I_\f(n) = I_\f(\omega)$
for any non-principal ultrafilter $\omega$. 
For any integers $n$ and $m$, write
\[\left|\int \varphi d\mu_{f_n} - \int \varphi d\mu_{f_\omega}\right| \le I_1 + I_2 + I_3
\]
where
\begin{align*}
I_1
&=  
\left|\int \varphi \left(\mu_{f_n} - \frac1{d^m} f^{m*}_n \mucan^n\right)\right|
\\
I_2&= \left|\int \varphi \left(\frac1{d^m} f^{m*}_n \mucan^n - \frac1{d^m} f^{m*}_\omega \mucan^\omega\right)\right|
\\
I_3& = 
\left|\int \varphi \left(\mu_{f_\omega} - \frac1{d^m} f^{m*}_\omega \mucan^\omega\right)\right|
\end{align*}
Here $\mucan^n$ (resp. $\mucan^\omega$) denotes the canonical measure $\mucan$ on $\pi^{-1}(n)$ (resp. $\pi^{-1}(\omega)$). Pick $\eta>0$. 
Write $f =[P\colon Q]$ with $\Res(P,Q)\in (A^\epsilon)^\times$, so that the function \[
g_{f,0}(z)
=\frac1{d}
\log \frac{\max\{|P|,|Q|\}}{\max\{|z_0|,|z_1|\}^d}
\]
is continuous on $\P^{1,\mathrm{an}}_{A^\epsilon}$.

Observe that
$\mu_{f_n} - d^{-m} f^{m*}_n \mucan^n 
= \Delta g_{n,m}$ with 
$g_{n,m}= \sum_{k\ge m} d^{-k} g_{f,0}\circ f_n^k$.
Choose $m$ large enough such that 
$\sup |g_{n,m}| \le \eta$ for all $n$. 
It follows that 
\[
I_1 = \left|
\int \f \Delta g_{n,m}\right|=
\left|
\int g_{n,m} \Delta \f\right|
\le \eta \times \mathrm{Mass}(\Delta \f|_{\pi^{-1}(n)})\le C\eta~,
\]
for some $C>0$ by Lemma~\ref{lem:qmlaplacian}. 
The same argument gives $I_3\le C\eta$ as well.

It remains to check that $I_2\to0$ as $n$ tends to $\infty$
along $\omega$. 
Suppose first that $\f$ is a model function. Then it is continuous, hence $f^m_*\f$ is also continuous on $\P^{1,\mathrm{an}}_{A^\epsilon}$ by Proposition~\ref{prop:push-beta}. 
Since $\mucan^n\to\mucan^\omega$ by Corollary~\ref{cor:convergence-can}, we conclude that $\lim_\omega I_2=0$. 
This proves~\eqref{eq:666}  in the case $\f$ is a model function. In particular, we have
$\lim_\omega f^{m*}\mucan^n=f^{m*}\mucan^\omega$.

\smallskip

We next treat the case $\f = \psi_P$ is a quasi-model function associated with a single section $P\in H^0(\P^1_{A^\epsilon}, \cO(l))$.  By Proposition~\ref{prop:push-quasimodel}, we are reduced to the case $m=0$. We assume that $P[1\colon0] \neq 0$, and 
work in the affine chart $[z\colon1]$. Jensen's formula (over $\C$) yields
\[
\int_{\pi^{-1}(n)} \psi_P d\mucan
=
\epsilon_n
\int_\C \log |P_n(z,1)| d\mucan
=
\epsilon_n
(\log|a_n| + \sum \log^+|z_{i,n}|)
\]
where $P_n(z,1)= a_n\prod (z-z_{i,n})$. 
The result follows since Jensen's formula is also valid over $\sH(\omega)$.

\smallskip

We now suppose $\f= \max \psi_{P_i}$ is an arbitrary quasi-model function, where $P_i\in A^\epsilon[z_0,z_1]$ are homogeneous polynomials of degree $l$. First observe that for any large $A>0$, the function
$\max\{\f,-A\}$ is continuous hence
\[
\lim_\omega \int \max\{\f,-A\}f^{m*}\mucan^n = 
 \int \max\{\f,-A\}f^{m*}\mucan^\omega~,
\]
and the monotone convergence theorem implies \begin{equation}\label{eq:omega-upper}
\lim_\omega \int \f\, f^{m*}\mucan^n
\le \int \f\, f^{m*}\mucan^\omega.
\end{equation}
 To prove the converse inequality, we fix $\delta>0$. Let $w_1, \cdots, w_k\in \P^1(\sH(\omega))$ be the base points of $\sigma$, i.e., the common zeroes of $P_{i,\omega}$. Write $P_{i,\omega}= Q_\omega \hat{P}_{i,\omega}$
such that $\hat{P}_{i,\omega}$ have no common factors, and $Q_\omega^{-1}(0)=\{w_1,\cdots, w_k\}$. 
It follows from Proposition~\ref{propalgc}, that we have a decomposition
$P_{i,n}= Q_{i,n} \hat{P}_{i,n}$ where $(Q_{i,n})$ (resp. $(\hat{P}_{i,n})$) represents $Q_\omega$
(resp. $P_{i,\omega}$) for all $i$.

Choose any finite open cover $U_1,\cdots, U_p$ of $\P^{1,\mathrm{an}}_{A^\epsilon}$ 
satisfying the following properties
for each $i\in \{1,\cdots, k\}$:
\begin{enumerate}
    \item $U_i$ contains no zeroes of $\hat{P}_{j,n}$ for all $j$ and all $n$;
     \item $U_i \cap Q_\omega^{-1}(0) =\{w_i\}$;
    \item  there exists $j(i)$ such that 
    $|\hat{P}_{j(i)}|\ge (1-\delta) \max_i |\hat{P}_i|$ on $U_i$.
\end{enumerate}
Pick any partition of unity associated with this cover, i.e., continuous functions $\rho_j\colon \P^{1,\mathrm{an}}_{A^\epsilon} \to [0,1]$ supported in $U_j$ and such that 
$\sum_j \rho_j =+1$. We further assume that 
$\rho_j=1$ in a neighborhood
of $w_j$ for all $j$.

\smallskip

Observe that  the upper bound~\eqref{eq:omega-upper} is still valid for the function $\rho_j \psi_{P_i}$ for any $i$ and $j$. It 
follows from the previous step that 
\begin{align*}
\int \psi_{P_i} f^{m*} \mucan^\omega
&\ge 
\sum_{j=1}^p
\lim_\omega \int \rho_j \psi_{P_i} f^{m*} \mucan^n
\\ &= 
\lim_\omega 
\int\psi_{P_i} f^{m*} \mucan^n
= 
\int \psi_{P_i} f^{m*} \mucan^\omega
\end{align*}
so that $\lim_\omega \int \rho_j \psi_{P_i} f^{m*} \mucan^n= \int \rho_j \psi_{P_i} f^{m*} \mucan^\omega$ for all $i,j$.

\smallskip

Now fix any $j\in\{1, \cdots, k\}$. From Property (3) of the open cover, we have
\begin{align*}
\lim_\omega
\int \rho_j \f f^{m*}\mucan^n
&\ge
\lim_\omega
\int \rho_j \psi_{P_{i(j)}} f^{m*}\mucan^n
\\
&= 
\int \rho_j \psi_{P_{i(j)}} f^{m*}\mucan^\omega
\ge 
\int \rho_j \f f^{m*}\mucan^\omega + \log (1-\delta)
\end{align*}
Letting $\delta \to0$, we get 
$\lim_\omega
\int \rho_j \f f^{m*}\mucan^n
\ge \int \rho_j \f f^{m*}\mucan^\omega$.

Now, the function $\f (1-\sum_{j=1}^k \rho_j)$ is continuous in $\pi^{-1}(E)$ for some $\omega$-big set $E$, and we have $\lim_\omega f^{m*}\mucan^n = f^{m*}\mucan^\omega$, hence
\begin{align*}
\lim_\omega
\int \f f^{m*}\mucan^n
&=
\lim_\omega
\int \f (1-\sum_{j=1}^k \rho_j)
 f^{m*}\mucan^n
 + 
\lim_\omega
\sum_{j=1}^k 
\int \rho_j \f f^{m*}\mucan^\omega
\\
&\ge 
\int \f f^{m*}\mucan^\omega
\end{align*}
which concludes the proof.
\end{proof}

\subsection{Convergence of Lyapunov exponents}
\label{sec:convergenceoflyapunovexponent}
Let $(k,|\cdot|)$ be an algebraically closed complete metric field of characteristic $0$.
Recall from~\S\ref{sec:Berkovich projective line over field} that we defined the projective
distance $d_{\P^{1}(k)}$ on $\P^1(k)$.

For any rational map $f=[P\colon Q]\in\Rat_{d}(k)$, and any $z\in\P^1(k)$ define 
\begin{align*}
    |df|(z)= \lim_{w\to z}
\frac{d_{\P^{1}(k)}(f(w),f(z))}{d(w,z)}~.
\end{align*}
A direct computation shows
\begin{align*}
    |df|[z_0\colon z_1]=\begin{cases}
        \frac{1}{d}\left|\frac{\partial P}{\partial z_{0}}\frac{\partial Q}{\partial z_{1}}-\frac{\partial P}{\partial z_{1}}\frac{\partial Q}{\partial z_{0}}\right|\frac{\max\{|z_{0}|^{2/e}+|z_{1}|^{2/e}\}^{e}}{\max\{|P|^{2/e}+|Q|^{2/e}\}^{e}},~\text{if $k=\C_{e}$};\\
        \frac{1}{d}\left|\frac{\partial P}{\partial z_{0}}\frac{\partial Q}{\partial z_{1}}-\frac{\partial P}{\partial z_{1}}\frac{\partial Q}{\partial z_{0}}\right|\frac{\max\{|z_{0}|^{2},|z_{1}|^{2}\}}{\max\{|P|^{2},|Q|^{2}\}},~\text{if $k$ is non-Archimedean},
    \end{cases}
\end{align*}
which proves that $|df|$ extends  continuously to $\P_{k}^{1,\mathrm{an}}$. 
Observe that $|df(z)|=0$ if and only if $z\in\P^1(k)$ is a critical point in the sense that $\deg_z(f)\ge2$. By the Riemann-Hurwitz formula, there are exactly $2d-2$ such points (counted with multiplicity). 
Finally, $\log|df|$ is the difference of a quasi-model function and a model function in the terminology of \S\ref{sec:modelfunction}.

\smallskip

Recall the definition of the equilibrium measure $\mu_f$ from \S\ref{sec:equilibriummeasure}. Since this measure integrates any quasi-model functions by \S\ref{sec: convergence of equilibrium measure}, we can define the Lyapunov exponent of $f$ as follows:
\[\chi_f=  \int\log|df| d\mu_{f}~,\]
and Theorem~\ref{theoremconvergenceofequilibriummeasure} immediately implies
\begin{cor}
\label{cor:lyapunov}
For any $\epsilon\in(0,1]^\N$, and for any rational map $f\in\Rat_d(A^\epsilon)$
the function $\omega \mapsto \chi_{f_\omega}$ is continuous on $\cM(A^\epsilon) \simeq \beta\N$.
\end{cor}

\begin{rmk}
The proof of Theorem~\ref{thm: sequential hyridation construction} (3) is now complete: the continuity of equilibrium measures follows from Theorem~\ref{theoremconvergenceofequilibriummeasure}, and the continuity of the Lyapunov exponents from the previous corollary.
\end{rmk}

\section{Luo's approach to degeneration}
\label{sec:luo's approach}
In this section, we compare the sequential hybridation of \S\ref{sec:sequentialhybridation} with the degeneration constructed by Luo in~\cite{YLuo22}, 
and prove in particular Theorem~\ref{thm:equivalence} from the introduction.

\subsection{The asymptotic cone}\label{sec:asymp-cone}

We recall the construction of the asymptotic cone in the context of the hyperbolic $3$-space, see, e.g.,~\cite[\S 9]{Ka09} for a general reference on this construction. 
Recall from~\S\ref{sec: hyperbolicgeometry} that we let $x_\star = (0,1)\in \H^3$.

Let $r_{n}\in\R^*_{+}$ be a sequence of positive real numbers such that $r_{n}\rightarrow \infty$ and let $\omega\in\beta\N\setminus \N$ be a non-principal ultrafilter.

Let $X(\omega)\subset (\H^3)^\N$ be the subset of all sequences $x$ in $\H^3$ such that 
\[\lim_{\omega}\frac1{r_n}d_{\H}(x_\star,x_n)<\infty\]
The asymptotic cone $\H^3_\omega$ of the sequence of pointed metric spaces \[\left(\H^{3},x_\star,r_n^{-1}d_{\H}\right)\] is by definition the quotient of $X(\omega)$ by the equivalence relation $(x_n)\sim (y_n)$ if and only if 
$\lim_{\omega}d_{\H}(x_n,y_n)=0$.

We endow  $\H_{\omega}^3$ with the distance 
\[d([x_{n}],[y_{n}])=\lim_{\omega}\frac{1}{r_n}d_{\H}(x_n,y_n)~,\]
for which it is complete.
Also since $\H^3$ is $\log 2$-hyperbolic in the sense of Gromov, the metric space  $(\H^3, r_n^{-1} d_\H)$ is $\log 2/r_n$-hyperbolic, therefore $\H^3_{\omega}$ is $0$-hyperbolic. 
In other words, $\H^3_{\omega}$ is a complete metric $\R$-tree.

\smallskip

Set $\epsilon_{n}= \frac{1}{r_n}$. Recall that $\sH(\omega)$ is the residue field of $\omega$ viewed as a point in $\cM(A^\epsilon)\simeq \beta\N$ .

Choose any sequence of points $x_n\in X(\omega)$.
By Theorem~\ref{convergence-conformal} (2a) and (2b), the sequence of conformal measures 
$(s_{\epsilon_n})_*\mu(x_n)$ converges to a Dirac mass at a Type-2 point $\psi(x)\in\H_{\sH(\omega)}$.

Moreover, if  $y\in X(\omega)$ is another representative sequence, then $\psi(y)=\psi(x)$
by Theorem~\ref{convergence-conformal} (3) so that we have a well-defined map $\psi \colon \H_{\omega}^3\rightarrow \H_{\sH(\omega)}$.
\begin{thm}
\label{theorem:asympH3}
The map $\psi \colon \H_{\omega}^3\rightarrow \H_{\sH(\omega)}$ is a bijective isometry.
\end{thm}
\begin{proof}
The fact that $\psi$ is an isometry follows from  Lemma~\ref{lem:distance-omega limite}. 
 Pick $x\in\H_{\sH(\omega)}$. By~\cite[Remark~7.11]{Ben19}, the group
 $\PGL_{2}(\sH(\omega))$ acts transitively on Type-2 points. 
 By Proposition~\ref{propresidue}, $\sH(\omega)$ is spherically complete, and its value group is $\R_+$ hence $x$ is a Type-2 point. 
 It follows that we can find a Möbius transformation $M_{\omega}\in \PGL_2(\sH(\omega))$ such that $x=M_{\omega}(x_g)$. Recall from the discussion in \S\ref{sec:complete residue fields} that $\sH(\omega)=A^\epsilon/\ker(\omega)$, hence
 we may suppose that $M_\omega$ is determined by some $M\in \PGL_2(A^{\epsilon})$. By Step 1 of Theorem \ref{convergence-conformal}, $\psi(p_\star)=\delta_{x_g}$, where $p_\star=(x_\star)\in(\H^{3})^{\N}.$ Since $M$ is defined over $A^\epsilon$, the sequence $y_n:= M_n(x_\star)$ belongs to $X(\omega)\subset (H^3)^\N$. 
 
Now observe that $M$ 
 induces a homeomorphism $\P^{1,\mathrm{an}}_{A^\epsilon}\to \P^{1,\mathrm{an}}_{A^\epsilon}$ commuting with the canonical morphism $\P^{1,\mathrm{an}}_{A^\epsilon}\to \cM(A^\epsilon)$. It follows that
\[
\lim_{\omega}\mu(y_n)
=
\lim_{\omega}\mu(M_n(x_\star))
=
\lim_{\omega} M_{n*} \mu(x_\star)
=
(M_{\omega})_*\delta_{x_g} = \delta_x,\]  
so that $\psi(y_n)=x$. This proves $\psi$ is surjective.
\end{proof}

\subsection{Barycentric extension of a rational map}
\label{sec:Barycentric extension}
Luo's work is based on the notion of barycenter of a probability measure $\mu$ on $\hat{\C}$ which is a point in 
$\H^3$. We follow the approach of~\cite{petersen2011conformally} for its definition
which uses the unit ball model $\B^3= \{x\in\R^3\mid |x|<1\}$ of the hyperbolic space where $|\cdot|$ is the standard Euclidean norm on $\R^3$. Recall that the Riemannian metric $d_\B$ is induced by $ds^2= \frac{dx^2}{(1-|x|^2)^2}$.
The map $\bar{\phi}\colon \bar{\H}^{3}\to \bar{\B}^3$ defined by $\infty \mapsto (0,1)$, and
\[
    (z,h) \longrightarrow \left(\frac{2z}{|z|^{2}+(h+1)^{2}},\frac{|z|^{2}+h^{2}-1}{|z|^{2}+(h+1)^2}\right)
    \in \C\times \R = \R^3
\]
is homeomorphism from $\bar{\H}^3$ onto $\bar{\B}^3$ which restrict to an isometry $(\H^3,d_\H)\to (\B^3,d_\B)$
and a biholomorphism $\hat{\C}\to S^2$ where $S^2$ is endowed with the unique structure of Riemann surface for which $ds^2$
is conformal. 

For each $w\in \B^{3}$, define 
\[g_{w}(x)=\frac{x(1-|w|^2)+w(1+|x|^{2}+2\langle w,x\rangle)}{1+|w|^{2}|x|^{2}+2\langle w,x\rangle},\]
where $\langle\cdot,\cdot \rangle$ denotes the standard scalar product. Then $g_w$ is an isometry of $(\B^3,d_\B)$.

Let $\mu$ be any probability measure on $S^2$ having no atoms. 
Introduce the  vector field on $\B^3$:
\[V_{\mu}(w)=\frac{1-|w|^2}{2}\int_{S^2}g_{-w}(x)d\mu(x).\]
According to \cite[Proposition 1 and 2]{petersen2011conformally}, $V_{\mu}(w)$ has a unique zero, 
and this point is called the {barycenter} of $\mu$. 
One can transport this definition using $\bar{\phi}$.
Any probability measure $\mu$ on $\hat{\C}$ having no atoms
thus admits a {barycenter} $B(\mu)\in \H^3$
which is the image by $\Bar{\phi}^{-1}$ of the unique zero of 
$\bar{\phi}_*\mu$. 
For any isometry $M$ of $(\B^3,d_\B)$, we have
$ V_{M_{*}\mu}(M(w))=D_{w}M(V_{\mu}(w))$, so that 
$B(M_{*}\mu)=M(B(\mu))$ for all $M\in\PGL_2(\C)$.

It is not difficult to see that the map $\mu \mapsto B(\mu)$ is continuous in the weak-$*$ topology of measures since $V_\mu$ depends continuously on $\mu$.

\smallskip

Recall from \S\ref{sec: hyperbolicgeometry} that we can attach to any point $x\in\H^3$ a unique probability measure $\mu(x)$ 
associated with a Riemannian metric
of constant curvature and invariant by the (compact) subgroup $G$ of $\PGL_2(\C)$ fixing the point $x$. 
The \emph{barycentric extension} of $f\in\Rat_d(\C)$
is then defined as follows:
\[\cE(f)(x)= 
\begin{cases}
    f(x)& \text{ if } x\in\hat{\C}\\
    B(f_{*}\mu(x)) &\text{ if } x\in\H^3.
\end{cases}\]
The next theorem is a combination of~\cite{petersen2011conformally} and~\cite[Theorem 1.1, Proposition 5.7]{luo2021trees}.
\begin{thm}
\label{propbarycentricextension}
For any $f\in\Rat_d(\C)$ and for any $M_1,M_2\in \PGL_2(\C)$
we have
\begin{equation}\label{eq:bary}
\cE(M_{1}\circ f \circ M_{2})=M_{1}\circ \cE(f)\circ M_{2}.
\end{equation}
Moreover, the barycentric extension $\cE(f)$ is proper, surjective and continuous on $\bar{\H}^{3}$.
It is real analytic and Lipschitz on $(\H^{3},d_\H)$, with Lipschitz bounded by $C \deg(f)$ where $C$ is a universal constant.
\end{thm}

\subsection{Asymptotic cone and Luo's degeneration map}
\label{sec:Asymptotic cone and Luo's degeneration map}
For any $f\in \Rat_d(\C)$, we set 
\[
\rluo (f) =  \sup_{\cE(f)(x)=0} d_\H(x,0) \in \R_+~.
\]
It follows from~\cite[Proposition 8.1]{luo2021trees} that $\rluo\colon\Rat_d(\C)\to\R_+$ is proper
(i.e., $\rluo(f_n)\to\infty$ as $f_n\to\infty$ in $\Rat_d(\C)$). It is not known whether $\rluo$ is continuous (see~\cite[\S8.6]{luo2021trees}).
Write $C'_d= e \sup_{\Rat_d(\C)}|\Res|$.
Using Luo's degeneration techniques, we shall prove that 
\begin{thm}\label{theorem:equivalence-luo-berko}
There exists a constant $C>1$ such that 
\begin{equation}\label{eq:luo-berk}
\frac1{C}
\le
\frac{1+\rluo(f)}{-\log(|\Res(f)|/C'_d)}
\le C
\end{equation}
for all $f\in\Rat_d(\C)$.
\end{thm}
For any $f\in\Rat_d(\C)$, we set 
\[
\rluof([f])=  \inf_{M\in\PGL_2(\C)} \rluo(M\cdot f)= \inf_{x\in\H^3} \sup_{\cE(f)(y)=x} d_\H(x,y)~.
\]
This defines a function $\rluof\colon\rat_d(\C)\to\R_+$ which is proper, and the previous result implies
Theorem~\ref{thm:equivalence} from the introduction.
%

\begin{proof}[Proof of Theorem~\ref{theorem:equivalence-luo-berko}]
It is sufficient to prove that for any sequence $f_n\in\Rat_d(\C)$ and for any non principal ultrafilter
$\omega\in\beta\N$ there exists a constant $C>1$ such that 
\begin{equation}\label{eq:estim}
\frac1{C}
\le
\frac{1+\rluo(f_n)}{-\log(|\Res(f_n)|/C_d)}
\le C
\end{equation}
for all integer $n$ lying in an $\omega$-big set. 

Indeed suppose~\eqref{eq:luo-berk} is not true. Then we can find a sequence $f_n\in\Rat_d(\C)$
such that the quantity ${1+\rluo(f_n)}/{-\log(|\Res(f_n)|/C_d)}$
either tends to $0$ or to $\infty$ which contradicts~\eqref{eq:estim}.

So pick any degenerating sequence $f_n\in\Rat_d(\C)$.
We consider the ultraproduct of $\H^3$ as defined in \S\ref{sec:asymp-cone} using the sequence $\rluo(f_n)\to\infty$.
Pick any sequence of points $(x_n)\in X(\omega)\subset(\H^3)^\N$ such that $d_{\H}(x_n,x_\star)\le M \rluo(f_n)$ for some $M<\infty$.
By Theorem~\ref{propbarycentricextension}, there exists $C>0$ such that for all $n$
\begin{align*}
   d_{\H}(\cE(f_n)(x_n),\cE(f_n)(x_\star))
\le 
C
d_{\H}(x_n,x_\star)\le CM \rluo(f_n) 
\end{align*}
and the sequence $(\cE(f_n)(x_n))$ also belongs to $X(\omega)$. The uniform Lipschitz property on $(\H^3,d_{\H})$
also implies that if $\lim_\omega d_{\H}(x_n,y_n)/\rluo(f_n)=0$ for some sequence $(y_n) \in (\H^3)^\N$, then 
\begin{align*}
    \lim_\omega d_{\H}(\cE(f_n)(x_n),\cE(f_n)(y_n))/\rluo(f_n)=0
\end{align*}
so that 
we get a natural map  $F\colon\H_{\omega}^{3}\rightarrow \H_{\omega}^{3}$
defined by $F(x_n)=  \left(\cE(f_n)(x_n)\right)$.

Set $\epsilon_n=  \rluo(f_n)^{-1}$, and consider the Banach ring $A^\epsilon$. For
a given $\omega\in\beta\N$, $\sH(\omega)$ is the quotient of $A^\epsilon$
by $\ker(\omega)= \{(z_n)\in\C \mid \lim_\omega |z_n|^{\epsilon_n}=0\}$. 

For each $n$, pick any normalized representation $f_n=[P_n\colon Q_n]$. Write 
 $P_n(z_0,z_1)=\sum_0^d a_{k,n}z_0^kz_1^{d-k}$ and $Q_n(z_0,z_1)=\sum_0^d b_{k,n}z_0^kz_1^{d-k}$. 
Consider the rational map $f_\omega\colon \P^1_{\sH(\omega)}\to \P^1_{\sH(\omega)}$
so that in homogeneous coordinates we have $f_\omega=[P_\omega\colon Q_\omega]$
with  $P_\omega(z_0,z_1)=\sum_0^d a_{k,\omega}z_0^kz_1^{d-k}$, and $Q_n(z_0,z_1)=\sum_0^d b_{k,\omega}z_0^kz_1^{d-k}$,
where $a_{k,\omega}=(a_{k,n})$ and $b_{k,\omega}=(b_{k,n})\in\sH(\omega)$.
Note that $\deg(f_\omega)\in\{0,\cdots,d\}$.

A key result is the following lemma, see~\cite[7.3]{YLuo22}.
\begin{lemma}\label{lem:keystepluo}
 The following diagram is commutative: 
 \begin{equation}
\begin{tikzcd}
{\H_{\omega}^{3}} \arrow[d,"F"] \arrow[r, "\psi"] 
& \H_{\sH(\omega)} \arrow[d, "f_{\omega}"] 
\\
{\H_{\omega}^{3}} \arrow[r, "\psi"]                 
&  
\H_{\sH(\omega)} .  
\end{tikzcd}
\label{eq:commutative diagram}
 \end{equation}
    \end{lemma}
 By~ \cite[Lemma~8.8]{luo2021trees}, one can find a point $y\in\H_\omega$
which admits exactly $d$ preimages by $F$ (in fact it follows from Theorem~1.2 in op.cit that $F$
is a covering map of degree $d$ on the real tree $\H^3_\omega$).
From Lemma~\ref{lem:keystepluo}, we infer that $f_\omega$ has degree at least $d$ which implies $\deg(f_\omega)=d$. 
As a consequence $P_\omega$ and $Q_\omega$ have no common factor hence
\[
\lim_\omega |\Res(P_n,Q_n)|^{1/r_n}
=
 |\Res(P_\omega,Q_\omega)|>0~,
\]
and we conclude that $-\log|\Res(P_n,Q_n)|/{\rluo(f_n)}$
is bounded from above and from below 
for all $n$ in an $\omega$-big set, proving~\eqref{eq:estim}.
\end{proof}

\begin{proof}[Proof of Lemma~\ref{lem:keystepluo}]

The proof is due to Luo and uses various results scattered in two papers~\cite{luo2021trees,YLuo22}. 
We include  a streamlined version of his arguments for convenience of the reader. 
Observe first that the diagram is commutative when $f_n$ is a sequence of Möbius transformations
since an element of $\PGL_2(\C)$ preserves the space of conformal metrics  of constant curvature. 

\begingroup
Denote by $p_{*}$ the point in $\H_{\omega}^3$ defined by $(x_*)\in(\H^{3})^{\N}$. Suppose we know that $F(p_\star)=p_\star$ if  $f_\omega(x_g)=x_g$. Then, pick any sequence $x=(x_n)\in X(\omega)\subset\H^3$ 
and choose a sequence  $y=(y_n)\in X(\omega)$ such that 
$f_\omega(\psi(x))=\psi(y)$.
We want to show that $F(x)=y$. 

Choose sequences $M_n,L_n\in\PGL_2(\C)$
 such that $M_n(x_\star)= x_n$ and 
 $L_n (y_n) =x_\star$ for all $n$. Then, $\psi(x)= M_\omega (x_g)$ and $\psi(y)=L_{\omega}(x_g)$.  We obtain 
 $L_\omega^{-1} \circ f_\omega \circ M_\omega(x_g)=x_g$, hence 
 $G(x_\star)=x_\star$ where $G$ is obtained as the limit of
 $\cE(L_n^{-1}\circ f_n\circ M_n)= L_n^{-1}\circ \cE(f_n)\circ M_n$. Therefore, $F(x)=y,$ implying that 
$ f_\omega\circ \psi = \psi \circ F$.

\smallskip

So, suppose that $f_\omega(x_g)=x_g$. We will apply the following lemma. 
\begin{lemma}\label{lem:keystepluo3}
There exists a sequence  $M_n\in\PGL_2(\C)$
 such that the transformation $M_\omega \in\PGL_2(\sH(\omega))$ fixes $x_g$ and $\deg(g_\infty) \ge1$
 where $g_n= M_n\circ f_n$ and $g_\infty$ is obtained by taking the $\omega$-limit of $g_n$ in $\overline{\Rat}_{d}(\C)$.
\end{lemma}

As observed by Luo~\cite[\S 5]{luo2021trees}, it follows from the proof of~\cite[Lemma 4.5 and 4.6]{zbMATH05004325demarcoiteration}
 that $\deg(g_\infty)\ge1$ if and only if  $\lim_\omega g_{n*}\mu_{FS}$ is a smooth measure. 
 Since $\deg(g_\infty)\ge1$, the sequence of measures
 $g_{n*}\mu_{FS}$ converges to a smooth measure on $\hat{\C}$ hence
 $\cE(g_n)(x_\star)= B(g_{n*}\mu_{FS})$ converges to some point $y\in\H^3$. 
 Since $\rluo(f_n)\to\infty$, it follows that $\lim_\omega d_\H(\cE(g_n)(x_\star),x_\star)/\rluo(f_n)= 0$
 so that $G(p_\star)=p_\star$. 
 But since $M_\omega (x_g)=x_g$
 we have $\lim_\omega d_\H(M_n(x_\star), x_\star)/\rluo(f_n)=0$ and
   \begin{align*}
 \lim_\omega 
 \frac1{\rluo(f_n)}
 d_\H(\cE(f_n)(x_\star), x_\star)
 \le &
 \lim_\omega 
 \frac1{\rluo(f_n)}
  d_\H(M_n^{-1}\circ \cE(g_n)(x_\star), M_n^{-1}(x_\star))
\\
&+
 \lim_\omega 
 \frac1{\rluo(f_n)}
 d_\H(M_n^{-1}(x_\star), x_\star)
 =
 0
 \end{align*}
 as required. 
 \end{proof}
\endgroup

 \begin{proof}[Proof of Lemma~\ref{lem:keystepluo3}]
 It remains to prove the claim.  We may suppose that $f_\infty=[1\colon0]$. 
 Let $I\subset\{0,\cdots,d\}$ be the set of indices $k$ such that 
 $\lim_\omega a_{k,n}\neq0$.
Then 
 $P_n=P^+_n+p_n$ 
where $P^+_n= \sum_I  a_{k,n}z_0^kz_1^{d-k}$ and
$p_n$ is a sequence of polynomials converging along $\omega$ to $0$. 
By assumption, $Q_n$ is also converging along $\omega$ to $0$. 
Set $c_n= \max_k |b_{k,n}|$ and let $J$ be the set of indices $j$
such that $\lim_\omega b_{k,n}/c_n\neq 0$. Observe that $\lim_\omega c_n=0$ 
but $|c_\omega|=1$ (i.e., that $\lim_\omega c_n^{1/\rluo(f_n)}>0$) since $f_\omega=[P_\omega\colon Q_\omega]$ fixes the Gauss point.

Set $L_n[z_0\colon z_1]=[z_0\colon c_n^{-1}z_1]$. Then $L_\omega(x_g)=x_g$ and 
$g_n=L_n\circ f_n= [P^+_n+p_n\colon Q^+_n+q_n]$
where $Q^+_n= \sum_J  \frac{b_{k,n}}{c_n}z_0^kz_1^{d-k}$ and $q_n$ is converging to $0$ along $\omega$. 
If $Q^+_\infty$ is not proportional to $P^+_\infty$, then we are done. Otherwise, 
we have $I=J$, and we set $t_n=a_{k_0,n}/b_{k_0,n}$ for some fixed $k_0\in I$, and $T_n[z_0\colon z_1]=[z_0\colon z_1-t_nz_0]$. 
Since $\lim_\omega t_n\neq0$, we have $T_\omega(x_g)=x_g$ and 
$T_n\circ g_n=[P^+_n+p_n\colon(Q^+_n-t_nP^+_n)+(q_n-t_np_n)]$. 

The polynomial 
\[R_n=  (Q^+_n-t_nP^+_n)+(q_n-t_np_n) =\sum_k  \gamma_{k,n}z_0^kz_1^{d-k}
\] tends to $0$ along $\omega$, 
and the key observation is that its coefficient $\gamma_{k_0,n}$ in the monomial $z_0^{k_0}z_1^{d-k_0}$ is $0$ for all $n$. 
We repeat the process as above. Set $d_n=  \max_k |\gamma_{k,n}|$. 
Since $f_\omega$ has degree $d$, we again have $d_n\to0$ and $|d_\omega|=1$.
Introduce $D_n[z_0\colon z_1]=[z_0\colon d_n^{-1}z_1]$, and 
let $K=\{k \in\{0,\cdots,d\}\mid \lim_\omega \gamma_{k,n}/d_n \neq0\}$. 
Then $D_\omega$ fixes $x_g$, and we can write
$d_n^{-1}R_n= R^+_n+r_n$
where all coefficients of $r_n$ tend to $0$ along $\omega$, 
and $R^+_n=\sum_{k\in K} \frac{\gamma_{k,n}}{d_n} z_0^kz_1^{d-k}$. 

Since
\[h_n= D_n\circ T_n\circ L_n\circ f_n=
[P^+_n+p_n\colon R^+_n+r_n]~,\]
we have
$h_\infty=[P^+_\infty\colon R^+_\infty]$
which is a non-constant map since 
$k_0\notin K$ so that 
$P^+_\infty$ and $R^+_\infty$ cannot be proportional. 
\end{proof}

\bibliographystyle{alpha}
\bibliography{ref}
\end{document}